\documentclass{compositio}

\usepackage{amsmath,amsfonts,yhmath,hyperref}
\usepackage{amsthm}
\usepackage{amssymb}
\usepackage[latin9]{inputenc}
\usepackage[nottoc,notlof,notlot]{tocbibind}
\usepackage{mathtools}
\usepackage{enumitem}
\usepackage{tikz-cd}
\usepackage{multirow}
\usepackage{dsfont}

\usepackage{float}
\usepackage{array}
\usepackage{mathrsfs}
\usepackage{appendix}

\usepackage{pgf,tikz}

\usetikzlibrary{arrows}
\usetikzlibrary[patterns]
\usetikzlibrary{shapes}
\definecolor{qqqqff}{rgb}{0.,0.,1.}
\definecolor{cqcqcq}{rgb}{0.7529411764705882,0.7529411764705882,0.7529411764705882}
\definecolor{ttqqqq}{rgb}{0.2,0.,0.}
\definecolor{qqqqff}{rgb}{0.,0.,1.}
\definecolor{xdxdff}{rgb}{0.49019607843137253,0.49019607843137253,1.}
\definecolor{zzttqq}{rgb}{0.6,0.2,0.}
\definecolor{cqcqcq}{rgb}{0.7529411764705882,0.7529411764705882,0.7529411764705882}

\definecolor{yqyqyq}{rgb}{0.5019607843137255,0.5019607843137255,0.5019607843137255}
\definecolor{uuuuuu}{rgb}{0.26666666666666666,0.26666666666666666,0.26666666666666666}
\definecolor{xdxdff}{rgb}{0.49019607843137253,0.49019607843137253,1.}
\definecolor{qqqqff}{rgb}{0.,0.,1.}

 \font\ncsc=cmcsc10

\newcommand{\PP}{\mathbb{P}}

\newcommand{\ZZ}{\mathbb{Z}}
\newcommand{\RR}{\mathbb{R}}
\newcommand{\CC}{\mathbb{C}}
\newcommand{\EE}{\mathbb{E}}

\newcommand{\QQ}{\mathbb{Q}}
\newcommand{\HH}{\mathbb{H}}

\newcommand{\CCC}{\mathscr{C}}

\newcommand{\FFF}{\mathscr{F}}
\newcommand{\PPP}{\mathscr{P}}
\newcommand{\SSS}{\mathscr{S}}

\newcommand{\JJJ}{\mathscr{J}}

\newcommand{\tFFF}{\widetilde{\mathscr{F}}}

\newcommand{\mfk}{\mathfrak{m}}

\newcommand{\efk}{\mathfrak{e}}

\newcommand{\Lfk}{\mathfrak{L}}

\newcommand{\V}{\mathcal{V}}
\newcommand{\M}{\mathcal{M}}

\newcommand{\R}{\mathcal{R}}

\newcommand{\J}{\mathcal{J}}

\renewcommand{\J}{\mathcal{J}}

\newcommand{\sfe}{\mathsf{e}}
\newcommand{\sff}{\mathsf{f}}
\newcommand{\sfs}{\mathsf{s}}
\newcommand{\sfp}{\mathsf{p}}
\newcommand{\sfq}{\mathsf{q}}
\newcommand{\sfA}{\mathsf{A}}
\newcommand{\sfC}{\mathsf{C}}
\newcommand{\sfH}{\mathsf{H}}
\newcommand{\sfT}{\mathsf{T}}
\newcommand{\sfR}{\mathsf{R}}

\newcommand{\bfh}{\mathbf{h}}
\newcommand{\bfw}{\mathbf{w}}
\newcommand{\bfmu}{\boldsymbol{\mu}}

\newcommand{\gen}[1]{\langle #1 \rangle}

\newcommand{\dd}{ \mathrm{d} }

\newcommand{\vir}[1]{[#1]^\mathrm{vir}}
\newcommand{\ev}{\mathrm{ev}}
\newcommand{\pt}{\mathrm{pt}}

\newcommand{\modulo}{\ \mathrm{mod}\ }

\renewcommand{\Im}{\mathrm{Im}\ }

\newcommand{\abcd}{\left(\begin{smallmatrix} a & b \\ c & d \\ \end{smallmatrix}\right)}
\newcommand{\tabcd}{\frac{a\tau+b}{c\tau+d}}

\newcommand{\bino}[2]{\begin{pmatrix}
#1 \\
#2 \\
\end{pmatrix}}

\newcommand{\sbino}[2]{\left(\begin{smallmatrix}
#1 \\
#2 \\
\end{smallmatrix}\right)}

\def\pearl (#1) at (#2:#3) label=#4; {
    \node[draw,rectangle, minimum width=0.5cm, minimum height = 0.5 cm] (#1) at (#2:#3) {$#4$} ;
}

\def\flatpearl (#1) at (#2:#3) label=#4; {
    \node[draw,circle, minimum width=0.5cm, minimum height = 0.5 cm] (#1) at (#2:#3) {$#4$} ;
}

\def\dottedpearl (#1) at (#2:#3) label=#4; {
    \node[draw,dotted,circle, minimum width=0.5cm, minimum height = 0.5 cm] (#1) at (#2:#3) {$#4$} ;
}

\newtheorem{theo}{Theorem}[section]
\newtheorem*{theom}{Theorem}
\newtheorem{prop}[theo]{Proposition}
\newtheorem{coro}[theo]{Corollary}
\newtheorem{lem}[theo]{Lemma}

\theoremstyle{definition}
\newtheorem{defi}[theo]{Definition}

\theoremstyle{remark}
\newtheorem{remark}[theo]{Remark}
\newenvironment{rem}[1]{
    \begin{remark}#1}{
    \xqed{\blacklozenge}\end{remark}
}

\theoremstyle{remark}
\newtheorem{example}[theo]{Example}
\newenvironment{expl}[1]{
    \begin{example}#1}{
    \xqed{\lozenge}\end{example}
}

\newcommand{\xqed}[1]{
    \leavevmode\unskip\penalty9999 \hbox{}\nobreak\hfill
    \quad\hbox{\ensuremath{#1}}}

\keywords{Enumerative geometry, Bielliptic surfaces, Gromov-Witten invariants, refined invariants, relative invariants, floor diagrams\\
Thomas Blomme,
{\ncsc Universit\'e de Gen\`eve, 5-7 rue du Conseil G\'en\'eral,
1205 Gen\`eve,
Switzerland} \\
\textit{Email :} thomas.blomme@unige.ch\\}
 
\begin{document}
 
%\renewcommand{\qedsymbol}{\filledbox}
%Good resources for looking up how to do stuff:
%Binary operators: http://www.access2science.com/latex/Binary.html
%General help: http://en.wikibooks.org/wiki/LaTeX/Mathematics
%Or just google stuff
 
\title{Gromov-Witten Invariants of Bielliptic Surfaces}
\author{Thomas Blomme}

\begin{abstract}
Bielliptic surfaces appear as quotient of a product of two elliptic curves and were classified by Bagnera-Franchis. We give a concrete way of computing their GW-invariants with point insertions using a floor diagram algorithm. Using the latter, we are able to prove the quasi-modularity of their generating series by relating them to generating series of graphs for which we also prove quasi-modularity results. We propose a refinement of these invariants by inserting a $\lambda$-class in the considered GW-invariants.
\end{abstract}

\maketitle

\tableofcontents

\section{Introduction}

	\subsection{Setting}
	
		\subsubsection{Bielliptic surfaces}, or sometimes called \textit{hyperelliptic surfaces}, form a class of complex surfaces with Kodaira dimension $0$ in the classification of compact complex surfaces \cite{beauville1978surfaces}. The latter is comprised of four kinds of surfaces. If the surface has trivial canonical bundle, it is either a K3 surface or an abelian surface (complex torus with some polarization). If the canonical bundle is not trivial (but still numerically $0$), the surface is either an Enriques surface, \textit{i.e.} the quotient of a K3 surface by an involution, or a bielliptic surface, which are actually quotients of particular abelian surfaces.
		
		\smallskip
		
		The classification of bielliptic surfaces was established by Bagnera-Franchis \cite{bagnera1908superficie} and Enriques-Severi \cite{enriques1910memoire} at the beginning of the XXth century. Bielliptic surfaces appear as quotient of very specific abelian surfaces, actually product of elliptic curves $E\times F$, by the action of a finite abelian group acting by translations on $F$, and by some non-trivial automorphisms on $E$. If $E$ is chosen generically, up to translation, the only automorphism is the symmetry $z\mapsto -z$, leading to a first kind of bielliptic surfaces, called of type $(a)$. To get different families, one has to assume that $E$ is either $\CC/\gen{1,i}$ or $\CC/\gen{1,\rho}$, where $i^4=\rho^3=1$, since these elliptic curves have more complicated automorphisms of respective order $3$, $4$ and $6$, leading to types $(b)$, $(c)$ and $(d)$. Accounting for these various possibilities, and the splitting of types $(a)$, $(b)$ and $(c)$ in two by adding a translation to the group action, there are actually seven types of bielliptic surfaces. We refer to Section \ref{sec-classification-biell-surf} for the detailed classification.
		
		\smallskip
		
		The homology groups of the bielliptic surfaces have been computed by Serrano \cite{serrano1989divisors,serrano1991picard}, and have rank $1,2,2,2,1$. Though their Betti numbers are the same, the seven different types of bielliptic surfaces are actually distinguished by their topology. Furthermore, $H_1(S,\ZZ)$ and $H_2(S,\ZZ)$ have a torsion part which varies from one type of surface to another. In this paper, we only care about the homology modulo torsion, so that all the bielliptic surfaces may be treated similarly. The group $H_2(S,\ZZ)$ modulo torsion is generated by classes $\sfe$ and $\sff$, which are proportional to push-forwards of the fibers $E$ and $F$ to the quotient.

		\subsubsection{Enumerative problems and GW-invariants.} To understand the geometry of a surface $S$, it is natural to study the curves that it contains. This can be done by studying stable maps $f:C\to S$, where $C$ is a nodal curve, realizing a fixed homology class $\varpi\in H_2(S,\ZZ)$ (modulo torsion). For a general surface, these curves are expected to be deformed in a space of dimension $-K_S\cdot\varpi+g-1$, so that imposing the right number of geometric constraints, we get a finite number of curves. This is unfortunately not true in general, since the moduli space $\M_g(S,\varpi)$ of genus $g$ stable maps to $S$ realizing the class $\varpi$ contains actually many components of various dimensions. This is especially not true for bielliptic surfaces, since even if we consider solely the immersed curves with a smooth domain of a given genus $g$, some components are of dimension $g-1$, which is the expected dimension, but some components are of dimension $g$. The latter consist in fact in the curves coming from the abelian cover, and are known to vary in dimension $g$.
		
		\smallskip
		
		To remedy this problem, the moduli space of genus $g$ stable maps with $n$ marked points $\M_{g,n}(S,\varpi)$ is endowed with a virtual fundamental class $\vir{\M_{g,n}(S,\varpi)}$, which is of the expected dimension \cite{behrend1997intrinsic}. The moduli space is also endowed with an evaluation map
		$$\ev:\M_{g,n}(S,\varpi)\longrightarrow S^n.$$
		We then obtain numerical invariants by capping pull-backs of cohomology class in $H^*(X,\QQ)$ on this virtual fundamental class. These invariants are called \textit{enumerative} Gromov-Witten invariants. We are mainly interested in the case where $S$ is a bielliptic surface and the classes are pull-backs of the cohomology class Poincar\'e dual to a point $\pt\in H^4(S,\QQ)$.
		
		\smallskip
		
		As the Picard group of a bielliptic surface is of dimension $1$, it is possible to look at curves belonging to a given linear system rather than just fixing the homology class. It is shown in \cite{bryan1999generating}, see also \cite[Section 4]{kool2011reduced}, that imposing this kind of constraints amounts to insert cohomology classes Poincar\'e dual to a basis of $H_1(S,\ZZ)$ modulo torsion. Thus, we also care about the GW-invariants where we insert two odd dimensional cycles.

		\subsubsection{Refinement of the enumerative invariants.} The moduli space stable maps is also endowed with a forgetful map to the moduli space of stable curves $\overline{\M}_{g,n}$. So it is also possible to pull-back cohomology classes from the latter. For various reasons, we are interested in the insertion of so-called $\lambda$-classes, which are the Chern classes of the Hodge bundle. The latter is the bundle where the fiber over a curve $C$ is the space of holomorphic forms $H^0(C,\omega_C)$, which is of dimension $g$. The choice of these particular cohomology classes is important for at least two different reasons:
		\begin{itemize}[label=$\ast$]
		\item The GW-invariants with a $\lambda$-class insertion are related to the study of \textit{tropical refined invariants}. F. Block and L. G\"ottsche \cite{block2016fock,block2016refined} propose to replace the tropical multiplicity involved in the tropical correspondence theorem proven by G. Mikhalkin \cite{mikhalkin2005enumerative} by a polynomial one, yielding polynomial counts of tropical curves. It was then proven by I. Itenberg and G. Mikhalkin \cite{itenberg2013block} that counts of tropical curves in $\RR^2$ of fixed genus and degree passing through a given number of points using this new refined multiplicity does not depend on the choice of the points as long as it is generic, leading to \textit{tropical refined invariants}. Due to their mysterial nature, tropical refined invariants have been at the center of multiple works since their discovery. It was proven by P. Bousseau that through a change of variable, tropical refined invariants actually compute the generating series of some GW-invariants with a $\lambda$-class insertion \cite{bousseau2019tropical,bousseau2021floor}. The advantage of this approach is that it allows one to define refined invariants even if we lack a tropical picture to the situation, by considering these $\lambda$-classes insertions. This is the case of some bielliptic surfaces.
		
		\item One other reason to care about $\lambda$-classes is that they appear when relating GW-invariants of surfaces and some specific threefolds. They are thus necessary when trying to compute GW-invariants of some Calabi-Yau threefolds. See for instance \cite{oberdieck2023enriques}.
		\end{itemize}

		\subsubsection{Degeneration and enumeration techniques.} The computation of GW-invariants remains in many cases a challenge. Yet, there exists some tools to tackle some computations. One of the main tools is Li's degeneration formula \cite{li2002degeneration} (also proven in \cite{kim2018degeneration}). The formula relies on the fact that GW-invariants are invariant in families: when considering a families of complex surfaces, fibers have the same GW-invariants. The idea is then to consider a family where one of the surfaces is reducible (with suitable notions of GW-invariants), consisting of several irreducible components each meeting all the other on a smooth divisor. It is then possible to express the GW-invariants of the reducible fiber (and thus any fiber) in terms of the GW-invariants of its irreducible components \textit{relative} to this intersection divisor. Here, \textit{relative} means that we deal with curves having a prescribed intersection profile with the divisor. Concretely, the formula states that when going to the reducible fiber, the curves in a generic fiber of the family break into the various components of the reducible fibers, and that it is possible to reconstruct everything from the various pieces, at the price of handling the combinatorics of their gluing.
		
		\smallskip
		
		The combinatorics involved in the decomposition formula \cite{li2002degeneration} are decorated bipartite graphs, where the colors correspond in fact to the two components of the reducible fibers. However, one application of the degeneration formula may not be enough to compute the GW-invariants since the relative GW-invariants of the components may still be complicated. It is possible to degenerate further, breaking into smaller pieces, but increasing the difficulty of the combinatorial part of the computation. These iterated applications of the decomposition formula lead to the introduction of more complicated combinatorial objects called \textit{floor diagrams}. The latter were introduced by E. Brugall\'e and G. Mikhalkin \cite{brugalle2007enumeration,brugalle2008floor} using the tropical geometry approach. They can also be seen as iterated version of the degeneration formula, as done for instance by Y. Cooper and R. Pandharipande \cite{cooper2017fock}. In our situation, the combinatorics of curves in bielliptic surfaces are encoded by \textit{pearl diagrams}, which are cyclic versions of the floor diagrams.

	\subsection{Results}
	
		\subsubsection{Pearl diagram algorithm.} We now explain how to compute the GW-invariants of bielliptic surfaces. For the latter, as the canonical class is torsion, the virtual dimension of the moduli space is $g-1$, so that we get a finite number by capping with $g-1$ point classes. For a bielliptic surface, homology class $\varpi\in H_2(S,\ZZ)$, the GW-invariant is denoted by
		$$\gen{\pt^{g-1}}_{g,\varpi}^S.$$
		To compute the GW-invariants of bielliptic surfaces, we apply the degeneration formula \cite{li2002degeneration} to some specific degenerations constructed in Section \ref{sec-construction-degeneration}. The computation is reduced down to the enumeration of some decorated graphs called \textit{pearl diagrams}, with an explicit multiplicity. We refer to Section \ref{sec-different-kinds-of-diagrams} for precise definitions.
		
		\smallskip
		
		A pearl diagram $\PPP$ of genus $g$ is an oriented graph with two kinds of vertices:  \textit{flat} vertices which are bivalent, and \textit{non-flat} vertices which are of genus $1$. The vertices are labeled by $[\![1;g-1]\!]$. The genus of the graph (including the non-flat vertices) is equal to $g$. It has the following decorations:
		\begin{itemize}[label=$\ast$]
		\item A non-flat vertex $V$ is endowed with a positive integer $a_V\geqslant 1$.
		\item Each edge $e$ is decorated with a \textit{weight} $w_e\geqslant 1$ such that the oriented graph is balanced: the sum of incoming weights matches the sum of outgoing weights.
		\item Each edge $e$ with extremities labeled $i$ and $j$ has an associated \textit{height} denoted by $h_e$, which is $\geqslant 0$ if $i<j$, and $\geqslant 1$ if $i\geqslant j$.
		\end{itemize}
		We furthermore assume that the complement of flat vertices is a union of components each having a unique cycle $\gamma_k$. The degree of the pearl diagram is $a\sfe+b\sff\in H_2(S,\ZZ)$, where $a=\sum_V a_V$, and $b=\sum w_eh_e$. Examples of pearl diagrams can be found on Figure \ref{fig-expl-pearl-diagram-1}, \ref{fig-expl-pearl-diagram-2} and \ref{fig-expl-pearl-diagram-3}.
		
		\smallskip
		
		Each pearl diagram is associated the following multiplicity:
		$$m(\PPP)=\prod_k\tau_n(h_{\gamma_k})\prod_V a_V^{n_V-1}\sigma_1(a_V)\prod_{E_\mathrm{flat}}w_e\prod_{E\backslash E_\mathrm{flat}}w_e^3,$$
		where the first product is to be explained, the second over the non-flat vertices, the third over edges adjacent to a flat vertex, and the last over the edges not adjacent to any flat vertex. The first product is over the connected components of the complement of marked points. By assumption, each contains a unique cycle $\gamma_k$, to which is associated a number $h_{\gamma_k}$ by summing the heights of its edges (with sign for the orientation). The function $\tau_n(h)$ is $n$-periodic and its values are as follows:
		\begin{itemize}[label=$\ast$]
		\item $n=2$: $\tau_2(0)=0$, $\tau_2(1)=4$,
		\item $n=3$: $\tau_3(0)=0$, $\tau_3(1)=\tau_3(2)=3$,
		\item $n=4$: $\tau_4(0)=0$, $\tau_4(1)=\tau_4(3)=2$, $\tau_4(2)=4$,
		\item $n=6$: $\tau_6(0)=0$, $\tau_6(1)=\tau_6(5)=1$, $\tau_6(2)=\tau_6(4)=3$, $\tau_6(3)=4$.
		\end{itemize}
		
		\begin{theom}(Proposition \ref{prop-decomposition} and \ref{prop-computation-multiplicity})
		The count of pearl diagrams with multiplicity $m(\PPP)$ yields the GW-invariant $\gen{\pt^{g-1}}_{g,\varpi}^S$.
		\end{theom}
		
		The results from \cite{bousseau2019tropical,bousseau2021floor} suggest that the enumerative invariants should be refined by considering the generating series with a $\lambda$-class insertion:
		$$\sum_{g\geqslant g_0}\gen{\lambda_{g-g_0};\pt^{g_0-1}}_{g,\varpi}^S u^{2g-2}.$$
		The computation of the latter is also enabled by the degeneration formula and the pearl diagrams, provided the latter are counted with a new adhoc \textit{refined} multiplicity given in Proposition \ref{prop-computation-multiplicity-lambda-case}.

		\subsubsection{Regularity of the generating series.} There already exists an extensive litterature on the generating series of GW-invariants, which often display nice regularity properties depending on the features of the considered surface. Most impressive results occur for the (reduced) GW-invariants of surfaces with trivial canonical bundle, where they are proved to be Fourier development of (quasi-)modular forms. For instance, in the case of K3 surface, we have the Yau-Zaslow conjecture \cite{beauville1999counting,bryan2000enumerative,klemm2010noether},  giving explicit expressions for the generating series. We also have explicit expressions in the case of abelian surfaces \cite{bryan1999generating,bryan2018curve,blomme2022abelian3}. Getting away from the surfaces with trivial canonical bundle, we also find quasi-modularity properties for surfaces associated to elliptic curves. For instance the generating series of their covers (see for instance \cite{goujard2019counting,boehm2018tropical}), or GW-invariants of $\PP^1$-bundles over an elliptic curve \cite{blomme2021floor}.
		
		\smallskip
		
		Quasi-modular forms are some specific functions on the Poincar\'e half-plane satisfying a transformation law under the action of $SL_2(\ZZ)$ by homography. Quasi-modularity is a desirable property because it enables a strong control over the coefficients of the series, such as a polynomial growth. Moreover, the finiteness of the dimension of quasi-modular forms of a given weight allows one to recover all the coefficients only knowing a finite number of them.
		
		\smallskip
		
		When the canonical class is non-trivial but still numerically zero, we may lose the quasi-modularity for $SL_2(\ZZ)$, but generating series may remain quasi-modular for smaller subgroups of $SL_2(\ZZ)$ called \textit{congruence subgroups} $\Gamma_0(n)$ or $\Gamma_1(n)$. The latter consist of matrices which are upper triangular ($\Gamma_0(n)$) or unipotent ($\Gamma_1(n)$) modulo some fixed number $n$. For instance, in the case of Enriques surfaces, G. Oberdieck proved that some generating series are quasi-modular for the congruence subgroup $\Gamma_0(2)$.
		
		\smallskip
		
		Our main regularity result is Theorem \ref{theo-quasi-modularity-GW-invariants}, concerning the double generating series
		$$F_g^S(\sfp,\sfq)=\sum_{a,b\geqslant 1}\gen{\pt^{g-1}}_{g,a\sfe+b\sff}^S\sfp^a\sfq^b.$$
		
		\begin{theom}\ref{theo-quasi-modularity-GW-invariants}
		The generating series $F_g^S$ is a finite sum of products $A(\sfp)B(\sfq)$, where $A$ is quasi-modular for $SL_2(\ZZ)$, and $B$ is quasi-modular for the congruence subgroup $\Gamma_1(n)$, $n=2,3,4$ or $6$ depending on the bielliptic surface.
		\end{theom}
		
		The proof uses the pearl diagram approach, reducing the computation to the study of some finite graphs. The quasi-modularity comes out of the following results:
		\begin{itemize}[label=$\ast$]
		\item Theorem \ref{theo-quasi-modularity-graph-sums} expresses the generating series associated to a given pearl diagram as the Fourier coefficient of a (quasi-)Jacobi form associated to the graph. This is a generalization from \cite[Theorem 6.1]{goujard2019counting} where we incorporate congruence constraints to the generating series.
		\item Theorem \ref{theo-quasi-modularity-periods-multivariable} which gives the quasi-modularity of the Fourier coefficients of (quasi-)Jacobi forms for congruence subgroups. This result follows the steps of the result for $SL_2(\ZZ)$ proven in \cite[Appendix]{oberdieck2018holomorphic}, which was itself a generalization of \cite[Theorem 5.8]{goujard2019counting}.
		\end{itemize}
		
		We are able to provide some explicit computations of generatings series. For simplicity, we give them for bielliptic surfaces of type $(a)$, and there are direct analogs for types $(b)$, $(c)$ and $(d)$:
		\begin{align*}
		\sum_{a,b\geqslant 1}\gen{\pt^{3}}_{3,a\sfe+b\sff}^S\sfp^a\sfq^b = & 4D^3E_2(\sfp)DE_2(\sfp)\left[E_4(\sfq)-E_4(\sfq^2)\right] + 4DE_2(\sfp)^2\left[DE_6(\sfq)-2DE_6(\sfq^2)\right], \\
		\sum_{a,b\geqslant 1}\gen{\lambda_{g-2};\pt}_{g,a\sfe+b\sff}^S\sfp^a\sfq^b = &   8\frac{(-1)^g}{(2g-2)!}DE_{2g-2}(\sfp)(E_{2g}(\sfq)-E_{2g}(\sfq^2) ), \\
		\end{align*}
		with $E_{2k}$ the Eisenstein series.

	\subsection{Possible directions}
		
		\subsubsection{More GW-invariants.} After dealing with the GW-invariants having point insertions and a $\lambda$-class, one is often interested in inserting other natural cohomology classes, such as the $\psi$-classes, getting so-called \textit{descendant invariants}. The question is to see if the regularity results extend as well as computation techniques. The decomposition formula still applies, so this amounts to compute the corresponding relative invariants of $E\times\PP^1$.
		
		\smallskip
		
		Another direction in which to extend the family of GW-invariants is to look at the torsion: our computation of the GW-invariants does not take care of the torsion part in the homology group of the considered bielliptic surface. It may be interesting to see if the pearl diagram approach enables to recover the torsion part of the degree.
		
		\subsubsection{More precise computation and regularity.} In the case of Enriques surfaces, using the quasi-modularity of the generating series, G. Oberdieck \cite{oberdieck2023enriques} is able to show that up to a change of variable, the GW-invariants of Enriques surfaces only depend on the square of the homology class $\varpi^2$. This may not be expected since we do not necessarily deformations and automorphisms relating two classes with the same square. In the case of bielliptic surfaces, homology classes are expressed as $a\sfe+b\sff$, but it does not depend only on the square $ab$. It would be interesting to see if a similar change of variables yields a nice expression for the generating series. 
		
		\subsubsection{Regularity of the refined invariants.} We only proved a regularity result for the generating series of enumerative invariants, \textit{i.e.} not involving any $\lambda$-class. We also expect such a quasi-modularity for the latter. It would be interesting to see how it translates to the coefficients of the refined invariants $BG_{g,\varpi}^S(q)$. Meanwhile, the refined invariants have been shown to satisfy regularity results of their own. For instance, in the toric setting, they satisfy a polynomiality property \cite{brugalle2020polynomiality}: coefficients of fixed codegree are ultimately polynomials in the curve class $\varpi$. Computations lead in the toric setting by the author and G. M\'evel in a forthcoming paper suggest there may be a dual \textit{G\"ottsche conjecture}, in the sense that these polynomials may be polynomials in $\varpi\cdot K_S$ and $\varpi^Z$, and depend on the surface only through the numbers $K_S^2$, $\chi(S)$. Given that refined invariants may be defined out of the toric setting through the insertion of $\lambda$-classes, it may be interesting to see if such regularity results extend out of the toric picture.
		
		\subsubsection{Extension of techniques to different situations.} The techniques used to prove the quasi-modularity of the generating series relies on a regularity result for generating series of graphs, inspired by \cite{goujard2019counting}. If we add some congruence conditions to the generating series, we are able to prove quasi-modularity for the congruence subgroup $\Gamma_1(n)$. In our situation, as $n=2,3,4,6$, this group is not really different from the bigger $\Gamma_0(n)$. It should be possible to generalize the techniques to look at generating series of GW-invariants for surfaces which are fibered over an elliptic curve, but with a fiber potentially not an elliptic curve (which is the case of bielliptic surfaces). To prove such regularity results, it may be necessary to look at congruence subgroups for $n\neq 2,3,4,6$, and to refine the modularity properties proven in Theorem \ref{theo-quasi-modularity-periods-multivariable} and \ref{theo-quasi-modularity-graph-sums}.

	\subsection{Organization of the paper}

The paper is organized as follows. In the second section, we give a brief description of the bielliptic surfaces, their classification, their homology groups, and how to construct some families of bielliptic surfaces degenerating to some peculiar reducible surfaces. In the third section, we recall the definition of GW-invariants to set notations, and compute some relative GW-invariants of $E\times\PP^1$, where $E$ is an elliptic curve. Then, we get to the computation of the invariants of bielliptic surfaces, introducing the pearl diagram algorithm. Last, we prove the quasi-modularity of the generating series of enumerative GW-invariants by studying generating series of graphs.
	
\textit{Acknowledgements.} The author would like to thank Francesca Carocci for helpful discussions during a workshop in Belalp, and Georg Oberdieck for helpful discussions concerning modular forms, GW-invariants and pointing out some references on the matter.

\section{Complex bielliptic surfaces and some of their degenerations}

	\subsection{Short historical presentation}
	
\textit{Bielliptic surfaces} (also called sometimes \textit{hyperelliptic surfaces}) can be roughly described as surfaces which are fibered by elliptic curves in two different (transversal) ways, hence the name bielliptic. Bielliptic surfaces form one of the classes of Kodaira dimension $0$ in the Kodaira-Enriques classification of surfaces. The latter is exposed in \cite{beauville1978surfaces}. The Kodaira dimension $0$ surfaces include K3 surfaces, Enriques surfaces, Abelian surfaces and bielliptic surfaces. The latter play to Abelian surfaces the role of Enriques surfaces to K3 surfaces as they appear as quotient of the latter, but have a non-trivial canonical class. More precisely, a complex surface with trivial canonical bundle is either a K3 surface or an Abelian surface, depending on whether it is simply connected or not. Enriques surfaces are quotients of K3 surfaces, while bielliptic surfaces appear as quotients of Abelian surfaces. The canonical classes of the latter are not trivial but only numerically trivial.

\smallskip

The classification of bielliptic surfaces into seven different types was established in Bagnera-de-Francis \cite{bagnera1908superficie} and Enriques-Severi \cite{enriques1910memoire} at the beginning of the XXth century. The current classification presents the bielliptic surfaces as the quotient of a product of elliptic curves by the action of a finite abelian group. It appears in \cite{tatsuo1970hyperelliptic}. Their Picard group has been computed by F. Serrano in \cite{serrano1989divisors}, who also computed their homology groups \cite[Theorem 4.3]{serrano1991picard}. The homology with rational coefficients can easily be computed, but the torsion part is more delicate. One of the results of \cite{tatsuo1970hyperelliptic} is that the seven different types of bielliptic surfaces are classified by their topology. However, this does their GW-invariants from being equal. As we will see, different families of bi-elliptic have the same GW-invariants of we do not care about the torsion part.

	\subsection{Presentation of the classification} \label{sec-classification-biell-surf}

Let $E$ and $F$ be two elliptic curves, and let $G$ be a finite subgroup of $F$, acting on the latter by translation. With additional assumptions, we can make it act on $E$ as well, depending on the automorphisms of $E$. Provided the quotient $S=(E\times F)/G$ is not an abelian surface, we call it a \textit{bielliptic surface}. They are classified as follows.

\begin{tabular}{lp{6cm}lp{6cm}}
$(a1)$ & Take any $F$ and $E$, $G=\ZZ/2\ZZ$ acts on $E$ by
$\sigma(z)=-z$. & $(a2)$ & Take $G=(\ZZ/2\ZZ)^2$ by adding $\sigma'(z)=z+c$, with $c$ of $2$-torsion. \\
$(b1)$ & Take any $F$ and $E=\CC/\gen{1,\rho}$, $G=\ZZ/3\ZZ$ acts on $E$ by
$\sigma(z)=\rho z$. & $(b2)$ & Take $G=(\ZZ/3\ZZ)^2$ by adding $\sigma'(z)=z+c$, with $c$ is fixed by $\sigma$. \\
$(c1)$ & Take any $F$ and $E=\CC/\gen{1,i}$, $G=\ZZ/4\ZZ$ acts on $E$ by
$\sigma(z)=iz$. & $(c2)$ & Take $G=\ZZ/4\ZZ\times \ZZ/2\ZZ$ by adding $\sigma'(z)=z+c$, with $c=\frac{1+i}{2}$. \\
$(a1)$ & Take any $F$ and $E=\CC/\gen{1,\rho}$, $G=\ZZ/6\ZZ$ acts on $E$ by
$\sigma(z)=-\rho^2 z$. & & \\
\end{tabular}

\begin{rem}
For a generic elliptic curve, its automorphism group is generated by $z\mapsto -z$ and translations. To get more sophisticated automorphisms, one needs to consider the elliptic curves $\CC/\gen{1,i}$ and $\CC/\gen{1,\rho}$ which have a bigger automorphism group.
\end{rem}

The bielliptic surface $S=(E\times F)/G$ inherits two projections:
\begin{itemize}[label=$\ast$]
\item The projection onto $\widetilde{F}=F/G$, which is an elliptic curve since $G$ acts on $F$ by translations. The fiber of this projection is $E$.
\item The projection onto $E/G$, which is biholomorphic to the Riemann sphere $\PP^1$ since the action is not free. The generic fiber is $F$. More generally, the fiber over a point $p$ is the elliptic curve $F/\mathrm{Stab}(p)$.
\end{itemize}

It is possible to get families of bi-elliptic surfaces by varying the base curve $F$. It is only possible to vary $E$ for surfaces of type $(a)$.

\begin{figure}
\begin{center}
\begin{tabular}{|c|c|c|c|c|}
\hline
type & $\tau$ & $G$ & $\mathrm{Tors}(H^2(S,\ZZ))$ & basis $(\sff,\sfe)$ of $\mathrm{Num}(S)$\\
\hline
$(a1)$ & any & $\ZZ/2\ZZ$ & $(\ZZ/2\ZZ)^2$ & $(1/2)[F],[E]$ \\
$(a2)$ & any & $\ZZ/2\ZZ\oplus\ZZ/2\ZZ$ & $\ZZ/2\ZZ$ & $(1/2)[F],(1/2)[E]$  \\
\hline
$(b1)$ & $\rho$ & $\ZZ/3\ZZ$ & $\ZZ/3\ZZ$  & $(1/3)[F],[E]$  \\
$(b2)$ & $\rho$ & $\ZZ/3\ZZ\oplus\ZZ/3\ZZ$ & $0$  & $(1/3)[F],(1/3)[E]$  \\
\hline
$(c1)$ & $i$ & $\ZZ/4\ZZ$ & $\ZZ/2\ZZ$ & $(1/4)[F],[E]$  \\
$(c2)$ & $i$ & $\ZZ/4\ZZ\oplus\ZZ/2\ZZ$ & $0$  & $(1/4)[F],(1/2)[E]$  \\
\hline
$(d)$ & $\rho$ & $\ZZ/6\ZZ$ & $0$ & $(1/6)[F],[E]$  \\
\hline
\end{tabular}
\label{fig-table-classification}
\caption{Summary of the classification of the bielliptic surface. On the second column, the parameter of the elliptic curve $E=\CC/\gen{1,\tau}$, third row the torsion part in the homology groups, and last, a basis of the second homology group modulo torsion.}
\end{center}
\end{figure}

	\subsection{Homology and cohomology groups}
	
	Given that a bielliptic surface is obtained as the quotient of a torus by a fixed point free group action, it is easy to compute its homology with $\QQ$-coefficient using Poincar\'e duality, universal coefficients, and that its cohomology groups are the subgroups of $H^*(E\times F,\QQ)$ stabilized by the group action. We get that its Betti numbers are $1,2,2,2,1$. In particular, The group $\mathrm{Num}(S)$ (equal to $H_2(S,\ZZ)$ modulo torsion) is of rank $2$.
	
	For each bielliptic surface $S$, the group $\mathrm{Num}(S)$ contains the classes of fibers $[E]$ and $[F]$ of the two projections. Moreover, they satisfy $[E]\cdot [F]=|G|$, so that they span a sublattice of $\mathrm{Num}(S)$ of index $|G|$. A basis of $\mathrm{Num}(S)$ in terms of $[E]$ and $[F]$ is provided by \cite[Theorem 1.4]{serrano1991picard}. We denote this basis by $(\sfe,\sff)$, with $\sfe$ proportional to $[E]$ and $\sff$ proportional to $[F]$. We now have $\sfe\cdot\sff = 1$.
	
	The first cohomology group (up to torsion) is generated by $\alpha$ and $\beta$, which are actually pull-backs of generators of $H^1(\widetilde{F},\ZZ)$ by the projection $S\to\widetilde{F}$. We get elements of $H^2(S,\ZZ)$: $\alpha\beta$ on one side, pull-back of a generator of $H^2(\widetilde{F},\ZZ)$ by the projection $S\to\widetilde{F}$, and $\omega$ on the other side, which is the pull-back of a generator of $H^2(\PP^1,\ZZ)$ by the projection $S\to E/G=\PP^1$. The cup product yields the generators of $H^3(S,\ZZ)$ $\omega\alpha$ and $\omega\beta$. Up to torsion, $H^*(S,\ZZ)$ is generated by $\alpha,\beta$ and $\omega$, which are of respective degree $1,1$ and $2$, are squared to $0$ and commute (up to the grading). The class $\omega\alpha\beta\in H^4(S,\ZZ)$ is Poincar\'e dual to the class of a point and is denoted by $\pt$.
	
	\begin{rem}
	Universal coefficients theorem proves that the (co)homology of a bielliptic surface has the following form:
	$$H_*(S,\ZZ)=\left\{\begin{array}{l}
	\ZZ \\
	\ZZ^2\oplus T \\
	\ZZ^2\oplus T \\
	\ZZ^2 \\
	\ZZ \\
	\end{array}\right. , \  H^*(S,\ZZ)=\left\{\begin{array}{l}
	\ZZ \\
	\ZZ^2 \\
	\ZZ^2\oplus T \\
	\ZZ^2\oplus T \\
	\ZZ \\
	\end{array}\right.$$
	where $T$ is a torsion group. This torsion group is given  in Figure \ref{fig-table-classification}, and is computed in \cite{serrano1991picard}. We do not need this torsion in this paper.
	\end{rem}
	
	%In the case of a bielliptic surface of type $(a1)$, the projection to $F/G$ has four double fibers. They are all numerically equivalent, but their differences  are exactly the torsion classes of $H_2(X,\ZZ)$. To see this, we can pass to homology with coefficients in $\ZZ_2$ and intersect with a basis.

\subsection{Construction of some degenerations}
\label{sec-construction-degeneration}

To compute the Gromov-Witten invariants of bielliptic surfaces, we use the degeneration formula \cite{li2002degeneration}. To do so, we first need to construct some degenerations, \textit{i.e.} some families of bielliptic surfaces with a reducible central fiber. This is done by degenerating the base $\widetilde{F}=F/G$ of the first elliptic fibration. More precisely, the curve $\widetilde{F}$ is degenerated to a necklace of $\PP^1$, glued to each other at $0$ and $\infty$. Meanwhile, the surface $S$ degenerates to a necklace of surfaces $X$ which are $\PP^1$-bundle over an elliptic curve $\widetilde{E}$ obtained as projective completion of a torsion line bundle, glued to each other along the divisors given by their $0$-section and $\infty$-section. The elliptic curve $\widetilde{E}$ is the quotient of $E$ by the subgroup of $G$ which acts on $E$ by translations. The number of surfaces in the necklace can be arbitrary. We give below the precise nature of $X$ according to the type of $S$:

\begin{center}
\begin{tabular}{rl}
$(a1),(b1),(c1),(d)$ & $X$ is the trivial $\PP^1$-bundle over $\widetilde{E}(=E)$, \\
$(a2),(c2)$ & $X$ is a $2$-torsion $\PP^1$-bundle over $\widetilde{E}$, \\
$(b2)$ & $X$ is a $3$-torsion $\PP^1$-bundle over $\widetilde{E}$. \\
\end{tabular}
\end{center}

	\subsubsection{The base curve.}
	
	Consider the (infinite) fan in $\RR^2$ formed by rays $\RR_{\geqslant 0}\sbino{n}{1}$ for each $n\in\ZZ$. We apply the construction of toric surfaces from \cite{fulton1993introduction}, yielding a complex surface. More precisely, to each cone of the fan generated by two consecutive rays corresponds an affine chart $U_n=\CC^2$ with coordinates $(x_n,y_n)$, and $U_n$ is glued to $U_{n+1}$ over $\CC^*\times\CC$ by the following isomorphism:
	$$(x_n,y_n)\longmapsto (x_n^2 y_n,\frac{1}{x_n}).$$
	
	\medskip
	
	The fan possesses a projection onto the second coordinate. Therefore, the surface inherits a map to $\CC$. In each chart $U_n$, it expresses as the monomial $t_n=x_ny_n$. The surface is denoted by $\tFFF$ and fibers of the map to $\CC$ by $\tFFF_t$. The fiber $\tFFF_t$ for $t\neq 0$ is $\CC^*$, while the central fiber $\tFFF_0$ is an infinite chain of $\PP^1$, each one glued to the next along their toric divisors $0$ and $\infty$.
	
	\medskip
	
	For any $L\geqslant 1$ and $\lambda\in\CC^*$, we have a $\ZZ$-action on $\tFFF$. Taking as coordinates on the dense torus $x=\sbino{1}{0}$ and $t=\sbino{0}{1}$, it is generated by the following automorphism, restricting to
	$$(x,t)\longmapsto (\lambda t^L x,t),$$
	on the dense torus. This action is a lift to $(\CC^*)^2$ of the action of $\left(\begin{smallmatrix} 1 & L \\ 0 & 1 \\ \end{smallmatrix}\right)$ on the fan in $\RR^2$. Thus, the action extends to $\tFFF$ and maps the chart $U_n$ to $U_{n+L}$, and the induced action on the central fiber is just the translation by $L$ copies of $\PP^1$ and multiplication by $\lambda\in\CC^*$. We denote by $\varphi_L:\tFFF\to\tFFF$ the generator of this action. Taking the quotient by this action, we get a family $\FFF$ whose fiber over $t\neq 0$ is the elliptic curve $\FFF_t=\CC^*/\gen{\lambda t^L}$, and the central fiber $\FFF_0$ is a necklace of $L$ copies of $\PP^1$ glued over their toric divisors.
	
	\begin{rem}
	The family $\FFF$ is related to the Tate curve $\CC[[t]]/t^\ZZ$.
	\end{rem}

	\subsubsection{Degenerations of the first kind.} We extend the previous construction by making the product with a fixed elliptic curve $E$ to construct families of bielliptic surfaces. Let $E$ be an elliptic curve and $\psi:E\to E$ be an automorphism. Then, we consider the joint action
	$$\Phi:(p,z)\in\tFFF\times E\longmapsto (\varphi_L(p),\psi(z))\in\tFFF\times E.$$
	This action is a lift of the action on $\tFFF$. Let $\SSS$ be the quotient by this action, which by construction possesses a map to $\FFF$ with fibers $E$. Composing with $\FFF\to\CC$, the fiber $\SSS_t$ over $t$ of this family is a bielliptic surface which depends on the nature of $\psi$.
	\begin{itemize}[label=$\ast$]
	\item any $E$ and $\psi(z)=-z$ leads to type $(a1)$,
	\item $E=\CC/\gen{1,\rho}$ and $\psi(z)=\rho z$ leads to type $(b1)$,
	\item $E=\CC/\gen{1,i}$ and $\psi(z)=iz$ leads to type $(c1)$,
	\item $E=\CC/\gen{1,\rho}$ and $\psi(z)=-\rho z$ leads to type $(d)$.
	\end{itemize}
	
	Meanwhile, the central fiber is also a $E$-bundle over $\FFF_0$, which is thus a reducible surface. Each irreducible component is the product $\PP^1\times E$. However, the global space is not $\FFF_0\times E$ since there is some monodromy when going across the direction of the necklace of $\PP^1$.

	\subsubsection{Degenerations of the second kind.} We now twist the previous construction to get the families of type $(\bullet 2)$, when there is some monodromy in the other direction of each fiber $\FFF_t$. We still assume that we have an elliptic curve $E$ with an automorphism $\psi$. Moreover, let $c$ be a fixed point of $\psi$, which is thus of $n$-torsion for some $n$, and let $\zeta$ be a $n$th-root of unity. We now consider the additional $\ZZ/n\ZZ$-action on $\tFFF$ generated by
	$$\zeta\cdot-:(x,t)\longmapsto (\zeta x,t).$$
	The fixed points of this action are exactly the corners of the toric surface, \textit{i.e.} the $0$ in each chart $U_n$. The quotient develops singularities at the fixed points according to $n$.
	
	\begin{expl}
	When you take $n=2$ and $\zeta=-1$, you have quadratic singularities at the nodes. For bigger $n$, one would get an $A_n$ singularity.
	\end{expl}
	
	We lift this automorphism to an automorphism of $\tFFF\times E$ as follows:	
	$$\Psi:(p,z)\in\tFFF\times E\longmapsto (\zeta\cdot p,z+c).$$
	Due to the translational part, this automorphism is now fixed point free. Moreover, as $c$ is a fixed point of $\psi$, it commutes with $\Phi$ so that we now have a free action of $\ZZ/n\ZZ\times\ZZ$ on $\tFFF\times E$. We still denote the quotient by $\SSS$, and it is still endowed with a map to $\CC$ whose fibers are denoted by $\SSS_t$. If we only quotient by the $\ZZ/n\ZZ$-action, the fibers are as follows:
	\begin{itemize}[label=$\ast$]
	\item The fiber over $t\neq 0$ is the quotient of $\CC^*\times E$ by $(x,z)\mapsto (\zeta x,z+c)$. It can be seen as a bundle in two ways: as a $E$-bundle over $\CC^*/\gen{\zeta}\simeq\CC^*$ which has monodromy $z\mapsto z+c$, and as a $\CC^*$-bundle over $E/(z\mapsto z+c)$ with monodromy $x\mapsto\zeta x$. It is thus the open stratum of a $n$-torsion line bundles over $E/(z\mapsto z+c)$.
	\item Each irreducible component of the central fiber $\tFFF_0$ is $\PP^1\times E$ and stabilized by $\Psi$. The action on each component is also $(x,z)\mapsto (\zeta\cdot x,z+c)$. Using the projection to $E/(z\mapsto z+c)$, the quotient is the total space of a $n$-torsion $\PP^1$ bundle over $E/(z\mapsto z+c)$. The central fiber is thus a chain of $n$-torsion $\PP^1$-bundles over $E/(z\mapsto z+c)$ glued along their $0$ and $\infty$-section.
	\end{itemize}
	
	When quotienting by the induced action of $\Phi$, we get a family $\SSS_t$ whose generic member is a bielliptic surface, and whose special fiber is a necklace of $L$ $n$-torsion $\PP^1$-bundles glued along their $0$ and $\infty$-sections. We take the following special values:
	\begin{itemize}[label=$\ast$]
	\item any $E$, $\psi(z)=-z$, $n=2$ and $\zeta=-1$ leads to generic member of type $(a2)$,
	\item $E=\CC/\gen{1,\rho}$, $\psi(z)=\rho z$, $n=3$ and $\zeta=\rho$ leads to generic member of type $(b2)$,
	\item $E=\CC/\gen{1,i}$, $\psi(z)=iz$, $n=2$ and $\zeta=-1$ leads to a generic member of type $(c2)$.
	\end{itemize}

\section{Gromov-Witten invariants}

We review in this section the definition of the Gromov-Witten invariants (GW-invariants) that we compute for bielliptic surfaces. The use of the \textit{degeneration formula} expresses them in terms of \textit{relative} GW invariants of $\PP^1$-bundles over an elliptic curve which we also review. This section is mainly used to fix notations.

	\subsection{Definitions}

	\subsubsection{Gromov-Witten invariants of bielliptic surfaces.} Let $S$ be a bielliptic surface. Let $\varpi=a\sfe+b\sff\in H_2(S,\ZZ)$ be an homology class. We have the moduli space of genus $g$ stable maps to $S$ with $n$ marked points realizing the class $\varpi$, which is denoted by
$$\M_{g,n}(S,\varpi).$$
It is endowed with a virtual fundamental class $\vir{\M_{g,n}(S,\varpi)}$ of real dimension $2(n+g-1)$. The real dimension is the double of the complex dimension, whic we need to consider due to the appearance of odd-dimensional cohomology classes. For each marked point we have an evaluation map
$$\ev_i:\M_{g,n}(S,\varpi) \to S,$$
which maps a stable map to the image of the $i$-th marked point, and the forgetful map
$$\mathrm{ft}:\M_{g,n}(S,\varpi)\to\overline{M}_{g,n}.$$

\smallskip

Moreover, if $\pi:\CCC\to\M_{g,n}(S,\varpi)$ is the universal curve, and $\omega_\pi$ is the relative dualizing sheaf, we have the Hodge bundle $\EE$ obtained by pushing forward the latter. Its fiber over a stable map $f:C\to S$ is $H^0(C,\omega_C)$. The bundle $\EE$ is $g$-dimensional since $H^0(C,\omega_C)$ is a vector space of dimension $g$. Its Chern classes are called the $\lambda$-classes:
$$\lambda_j=c_j(\EE)\in H^{2j}(\M_{g,n}(S,\varpi)).$$

\smallskip

We obtain GW invariants by capping cohomology classes obtained by pull-back with $\mathrm{ft}$ or some $\ev_i$ with the virtual fundamental class. If $\lambda\in H^*(\overline{\M_{g,n}},\QQ)$ and $\gamma_i\in H^*(S,\QQ)$, we set
$$\gen{\lambda;\gamma_1,\cdots,\gamma_n}^S_{g,\varpi} = \int_{\vir{\M_{g,n}(S,\varpi)}}\mathrm{ft}^*(\lambda)\prod_1^n\ev_i^*(\gamma_i),$$
which is non-zero if the cohomology class has degree $2(n+g-1)$. If there is no $\lambda$, the invariants are called \textit{enumerative}. We consider the following class insertions.
\begin{itemize}[label=$\ast$]
\item First, we only consider point insertions via the evaluation map. We take $n=g-1$ and consider the following enumerative invariants:
$$\gen{\pt^{g-1}}_{g,\varpi}^S =\int_{\vir{\M_{g,n}(S,\varpi)}} \prod_{i=1}^n\ev_i^*(\pt).$$

\item We take $n=g$ and add cohomology insertions Poincar\'e dual to $1$-cycles:
$$\gen{\omega\alpha,\omega\beta,\pt^{g-2}}_{g,\varpi}^S =\int_{\vir{\M_{g,n}(S,\varpi)}} \ev_1^*(\omega\alpha)\ev_2^*(\omega\beta)\prod_{i=3}^g\ev_i^*(\pt).$$

\item Following \cite{bousseau2019tropical}, the refinement corresponds to the insertion of a $\lambda$-class in the GW-invariants. To refine the above enumerative invariant, we insert a $\lambda$-class: we take $n=g_0-1$ and get
$$\gen{\lambda_{g-g_0};\pt^{g_0-1}}_{g,\varpi}^S =\int_{\vir{\M_{g,n}(S,\varpi)}} \lambda_{g-g_0}\prod_{i=1}^n\ev_i^*(\pt).$$

\item Similarly with the $1$-cycle insertions:
$$\gen{\lambda_{g-g_0};\omega\alpha,\omega\beta,\pt^{g_0-2}}_{g,\varpi}^S =\int_{\vir{\M_{g,n}(S,\varpi)}} \lambda_{g-g_0} \ev_1^*(\omega\alpha)\ev_2^*(\omega\beta)\prod_{i=3}^{g_0}\ev_i^*(\pt).$$
\end{itemize}

\begin{rem}
By construction, we consider all the complex curves realizing a given homology class (up to torsion). As the Picard group of a bielliptic surface is of dimension $1$, it is instead possible to consider curves belonging to a fixed linear system having given Chern class in $H^2(S,\ZZ)$. Using Theorem 2.1 and Corollary 2.2 from \cite{bryan1999generating} or \cite[Section 4]{kool2011reduced}, these conditions are actually equivalent to the insertions of cohomology classes Poincar\'e dual to a basis of $H_1(S,\ZZ)$ (modulo torsion), which are precisely the invariants considered above. As we did not care about torsion, the invariants counting curves in a fixed linear system are to be understood as the sum over the Chern classes differing by a torsion element.
\end{rem}

\begin{rem}
As the cohomology ring is graded commutative, everything commutes as long as all classes have even degree, which is the case of the classes $\pt$ and $\lambda_j$. However, the classes $\omega\alpha$ and $\omega\beta$ have odd degree. This is why we can only insert one of each. The skew-symmetry relation forces the GW-invariant to be $0$ if we insert more.
\end{rem}

\subsubsection{Relative invariants of $E\times\PP^1$.} The other family of GW invariants that we consider are \textit{relative invariants} of $X=\widetilde{E}\times\PP^1$, where $\widetilde{E}$ is an elliptic curve. The invariants are called \textit{relative} because we care about the intersection profile of the curves with the $0$-section and $\infty$-section, denoted respectively by $D^-$ and $D^+$. The second homology group $H_2(X,\ZZ)$ is generated by the class of a fiber $[\PP^1]$ and the class of a section $[\widetilde{E}]$, which we respectively denote by $\sfp$ and $\sfs$. Notice that $D^++D^-$ is an anticanonical divisor.

\begin{rem}
Technically, we also care about the surfaces $X$ which are the projective completion of a torsion line bundle over $\widetilde{E}$. However, it is possible to include all degree $0$ line bundle over $\widetilde{E}$ in a family, so that they are deformation equivalent and have the same GW invariants. Therefore, we may only care about $X=\widetilde{E}\times\PP^1$.
\end{rem}

Let $\varpi=a\sfs+b\sfp\in H_2(X,\ZZ)$ and $\mu\in\ZZ^n$ be a tuple of numbers. We have the moduli space of genus $g$ stable maps to $X$ relative to $D^\pm$ having tangency profile $\mu$, denoted by
$$\M_{g,n}(X/D^-\cup D^+,\varpi,\mu).$$
The tangency profile encoded in $\mu$ means the following:
	\begin{itemize}[label=-]
	\item If $\mu_i=0$, the marked point is mapped to the interior of $X$, \textit{i.e.} not to $D^\pm$.
	\item If $\mu_i>0$, the marked point is mapped to $D^+$ and the tangency order is $\mu_i$.
	\item If $\mu_i<0$, the marked point is mapped to $D^-$ and the tangency order $-\mu_i$.
	\end{itemize}
	For the moduli space to be non-empty, we need to have $\varpi\cdot D^+=\varpi\cdot D^-=\varpi\cdot [\widetilde{E}]$. In other words:
	$$\sum_{\mu_i>0}\mu_i=-\sum_{\mu_i<0}\mu_i=\varpi\cdot [E].$$

The above moduli space is a proper DM stack which possesses a virtual fundamental class in degree $2(g-1+n)$. The moduli space has evaluation maps for each marked point
$$\ev_i:\M_{g,n}(X/D^\pm,\varpi,\mu) \longrightarrow \left\{ \begin{array}{l}
X \text{ if }\mu_i=0,\\
D^\pm\simeq \widetilde{E} \text{ if }\mu_i\neq 0.\\
\end{array}\right.$$

Relative GW invariants are obtained by capping cohomology classes against the virtual fundamental class. We will consider the following cycle classes:
	\begin{itemize}[label=$\ast$]
	\item For points not mapped to $D^\pm$, we have the class $\pt_0\in H^4(X,\ZZ)$, which is Poincar\'e dual to a point. The index $0$ is to recall that it is a class associated to the interior of $X$.
	\item For points mapped to one of the divisors $D^\pm$, we pull-back  cohomology classes from $H^*(\widetilde{E},\ZZ)$. Let $1_\bullet,\alpha_\bullet,\beta_\bullet,\pt_\bullet=\alpha_\bullet\beta_\bullet$ be a basis of $H^*(\widetilde{E},\ZZ)$. The indices are meant to be the tangency order of the corresponding points. See below.
	\end{itemize}
	
\begin{rem}
	The cup-product is skew-symmetric, so relative invariants are symmetric in the entries of even degree, but skew-symmetric in the entries of odd degree, provided they have the same index, otherwise we do not have symmetry. Thus, one needs to be careful about the order of the insertions in case we have insertions of odd degree.
\end{rem}	
	
	We consider the following (relative) GW invariants.
	\begin{itemize}[label=$\ast$]
	\item Assume that $\mu\in\ZZ^{n+1}$, and all elements are nn-zero except $\mu_0=0$. We have the genus $1$ enumerative invariant with only point insertions:
	$$\gen{\pt_0,1_{\mu_1},\pt_{\mu_2},\cdots,\pt_{\mu_n}}_{1,a\sfs+b\sfp}^{X/D^\pm}.$$
	The indices encode the tangency orders of the marked points. Here, each marked point is coupled to a point constraint, except the one associated to $\mu_1$ which is free.
	
	\item Similarly, we have the invariants with two $1$-cycle insertions:
	$$\gen{ \pt_0,\alpha_{\mu_1},\beta_{\mu_2},\pt_{\mu_3},\cdots,\pt_{\mu_n}  }_{1,a\sfs+b\sfp}^{X/D^\pm}.$$
	
	\item We can consider the same cycle insertions if we increase the genus, provided we insert a $\lambda$-class to match dimensions:
	$$\gen{\lambda_{g-1};\pt_0,1_{\mu_1},\pt_{\mu_2},\cdots,\pt_{\mu_n}}_{g,a\sfs+b\sfp}^{X/D^\pm}.$$
	$$\gen{ \lambda_{g-1};\pt_0,\alpha_{\mu_1},\beta_{\mu_2},\pt_{\mu_3},\cdots,\pt_{\mu_n}  }_{g,a\sfs+b\sfp}^{X/D^\pm}.$$
	\end{itemize}

	\subsection{A vanishing result}
	
	In this section, we prove that it is possible to trade the insertion of odd-dimensional cohomology classes of a bielliptic surface $S$ to point constraints. This is inspired by \cite[Lemma 4]{bryan2018curve}. This allows us to only care about the GW-invariants with point insertions.
	
	We consider the moduli space $\M_{g,n}(S,\varpi)$, endowed with its evaluation map. The elliptic surface is endowed with the action by translation of the elliptic curve $F$. The action is not free, but stabilizers are finite. The action extends to the moduli space of curves $\M_{g,n}(S,\varpi)$. Let $\M_{g,n}^0(S,\varpi)$ be the quotient space (\textit{i.e.} stable maps to $S$ up to translation by $F$). We have the following commutative diagram, where $S^n/F$ is the quotient of $S^n$ by the diagonal action of $F$:
	$$\begin{tikzcd}
	\M_{g,n}(S,\varpi) \arrow[d,"\mathrm{pr}"] \arrow[r,"\ev"] & S^n \arrow[d,"p"] \\
	\M_{g,n}^0(S,\varpi) \arrow[r] & S^n/F \\
	\end{tikzcd} ,$$
	where $\mathrm{pr}$ and $p$ are the projection to the quotient.

\begin{prop}\label{prop-vanishing-result-translation-classes}
Let $\lambda\in H^*(\overline{\M}_{g,n},\QQ)$, $\gamma\in H^*\left( S^n/F ,\QQ\right)$ be an arbitrary class. If $\gamma_1\in H^*(\widetilde{E},\QQ)$ is of degree smaller than $1$, then
$$\int_{\vir{\M_{g,n}(S,\varpi)}} \mathrm{ft}^*(\lambda)\cup\ev_1^*(\gamma_1)\cup\ev^*p^*(\gamma)=0.$$
\end{prop}

\begin{proof}
Similarly to \cite[Lemma 4]{bryan2018curve}, the class
$$(\mathrm{ft}^*(\lambda)\cup\ev^*p^*(\gamma))\cap\vir{\M_{g,n}(S,\varpi)},$$
is the pull-back of a class $\theta$ on $\M_{g,n}^0(S,\varpi)$ by $\mathrm{pr}$. Using push-pull formula, we get
$$\mathrm{pr}_*(\ev_1^*(\gamma_1)\cap\mathrm{pr}^*(\theta))=\mathrm{pr}_*\ev_1^*(\gamma_1)\cap\theta.$$
But as $\gamma_1$ is of rank smaller than $1$ and fibers of $\mathrm{pr}$ are (finite quotients of ) $F$, we get $0$ for dimensional reasons.
\end{proof}

\begin{rem}
An intuitive way of viewing the vanishing is as follows: constraints from $\mathrm{ft}^*(\lambda)\cup\ev^*p^*(\gamma)$ only fix the curves up to the action by $F$, and the remaining condition $\ev_1^*(\gamma_1)$ is not enough to finish fixing them.
\end{rem}

We now apply Proposition \ref{prop-vanishing-result-translation-classes} to some specific classes. Recall that $H^*(S,\QQ)$ is generated by $\alpha,\beta$ pull-backs of generators of $H^1(\widetilde{F},\QQ)$, and $\omega$ pull-back for a generator of $H^2(\PP^1,\QQ)$, with relations $\alpha^2=\beta^2=\omega^2=0$. Using de Rham cohomology, these generators can be seen as the elements $\dd x_1\wedge\dd x_2$, $\dd x_3$ and $\dd x_4$ of $H^*(E\times F,\RR)\simeq H^*\left( (S^1)^4,\RR\right)$ which are invariant by the $G$-action whose quotient is the bielliptic surface $S$.

\begin{coro}
We have the following equality:
$$\gen{\lambda;\omega\alpha,\omega\beta,\pt^{n-2}}_{g,\varpi}^S = \frac{\omega\cdot\varpi}{n-1}\gen{\lambda;\pt^{n-1}}_{g,\varpi}^S.$$
\end{coro}

\begin{proof}
We apply Proposition \ref{prop-vanishing-result-translation-classes} to some specific classes. We have $H^*(S^n,\QQ)=H^*(S,\QQ)^{\otimes n}$. We use indices to denote generators associated to different copies, so that $H^*(S^n,\QQ)$ is generated by $\alpha_i,\beta_i$ and $\omega_i$, with $1\leqslant i\leqslant n$. We use de Rham cohomology to determine some classes which are of the form $p^*\gamma$, for some $\gamma\in H^*(S^n/F,\QQ)$. Cohomology classes are represented by $k$-forms on $(E\times F)^n$ with constant coefficients which are invariant by the $G$-action. A sufficient condition for a form on $S^n$ to be the pull-back by $p$ is that it descends to the quotient by $F$, meaning that the form vanishes whenever one of the evaluation belongs to the tangent space of the orbit. The tangent space to the orbits consists in the vector of the form $(0,v,0,v,\cdots,0,v)$.
	\begin{itemize}[label=$\ast$]
	\item The $1$-forms $\alpha_i-\alpha_1$ and $\beta_i-\beta_1$ vanish on the tangent space to the orbits. Thus, they descend to the quotient $S^n/F$, so that they are actually pull-backs $p^*\widetilde{\alpha}_i$ and $p^*\widetilde{\beta}_i$.
	\item The $2$-forms $\omega_i\in H^2(S^n,\RR)$ are invariant by the $F$-action: the evaluation does not depend on the point, and the tangent space to the orbits belongs to the kernel of the form. Thus, these $2$-forms descend to $2$-forms on $S^n/F$, so that $\omega_i$ are of the form $p^*\widetilde{\omega}_i$.
	\end{itemize}
	Thus, the following class is of the form $p^*\gamma$:
	$$(\beta_2-\beta_1)\omega_2\prod_3^n (\alpha_i-\alpha_1)(\beta_i-\beta_1)\omega_i.$$
	We now apply Proposition \ref{prop-vanishing-result-translation-classes} with $\gamma_1=\alpha_1\omega_1$. We have
	\begin{align*}
	\gamma_1\cup p^*\gamma = & \alpha_1\omega_1 (\beta_2-\beta_1)\omega_2\prod_3^n (\alpha_i-\alpha_1)(\beta_i-\beta_1)\omega_i \\
	= & \alpha_1\omega_1\beta_2\omega_2\pt_3\cdots\pt_n - \pt_1\omega_2\pt_3\cdots\pt_n - \sum_{i=3}^n \pt_1\beta_2\omega_2\pt_3\cdots \alpha_i\omega_i\cdots\pt_n.\\
	\end{align*}
	Using the skew-symmetry and the divisor equation to get rid of $\omega_2$, we get the expected relation.
\end{proof}

	\subsection{Invariants with boundary constraints}
	
	We consider relative GW invariants of $X=\widetilde{E}\times\PP^1$ with a minimal number of cycle insertions in the main stratum. Before making any explicit computations, we prove some vanishing results. As $X$ is endowed with an action of $\CC^*$ inherited from the action on $\PP^1$, we need at least one insertion in the main stratum, or for the class to be of the form $w\sfp\in H_2(X,\ZZ)$.
	
	\subsubsection{Enumerative invariants.} We start with the invariants where we only allow cycle insertions.
	
	\begin{lem}\label{lem-non-zero-enum-inv}
	\begin{enumerate}[label=(\roman*)]
	\item If $a\neq 0$, the only non-zero enumerative invariants with a unique point insertion in the main stratum are
	$$\gen{\pt_0,1_{\mu_1},\pt_{\mu_2},\cdots,\pt_{\mu_n}}_{1,a\sfs+b\sfp}^{X/D^\pm},$$
	$$\gen{ \pt_0,\alpha_{\mu_1},\beta_{\mu_2},\pt_{\mu_3},\cdots,\pt_{\mu_n}  }_{1,a\sfs+b\sfp}^{X/D^\pm}.$$
	\item For a class $w[\PP^1]$, the only non-zero invariant with a point insertion is
	$$\gen{\pt_0,1_{-w},1_w}_{0,w\sfp}^{X/D^\pm}.$$
	\item For a class $w\sfp$, the only non-zero invariants without point insertion are
	$$\gen{\pt_{-w},1_w}_{0,w\sfp}^{X/D^\pm} \text{ and }\gen{\alpha_{-w},\beta_w}_{0,w\sfp}^{X/D^\pm},$$
	along with their symmetric ones.
	\end{enumerate}
	\end{lem}
	
	\begin{proof}
	First, we have \cite[Proposition 2.13]{blomme2021floor} which gives a relation between the position of the intersection points with the boundary divisors. More precisely, let $\CCC$ be a curve in the class $a\sfs+b\sfp\in H_2(X,\ZZ)$. Let $C_0$ and $C_\infty$ be the divisors on $\widetilde{E}$ obtained by intersecting $\CCC$ with $D^\pm\simeq\widetilde{E}$. Then
$$C_0-C_\infty\equiv 0\in E.$$
Thus, it is not possible to evaluate constraints which would prevent this relation from being realized.
	
	\begin{enumerate}[label=(\roman*)]
	\item Let $g$ be the genus, and $n+1$ be the length of $\mu$:
	$$\gen{\pt_0,\gamma_{\mu_1},\cdots,\gamma_{\mu_n}}_{g,a\sfs+b\sfp}^{X/D^\pm},$$
	where $\gamma_{\mu_i}$ are class insertions. By equating the rank of the insertions and the dimension of the moduli space, dividing everything by $2$, we have
	$$n+g=2+\sum_1^n |\gamma_i|\leqslant 2+(n-1)=n+1,$$
	where the inequality comes from the first part of the proof. Moreover, as there are no non-constant map from a rational curve to an elliptic curve, $g\geqslant 1$. Thus, we get $g=1$ and $\sum |\gamma_i|=n-1$. The latter can be rewritten
	$$\sum_i (1-|\gamma_i|)=1.$$
	The only possibility for the latter is to have a free insertion (as the rank is $0$) or a pair of $1$-cycle insertions (each of rank $\frac{1}{2}$), and point insertions (of rank $1$).
	
	\item If we are for a multiple of the fiber class, any stable map factors through some fiber, and the fibers vary in a $1$-parameter family. Thus, the fiber is fixed by the interior point insertion, so that we cannot have any other boundary insertion. The dimension count is
	$$(1+n)+g-1 = 2.$$
	As $n\geqslant 2$ since we intersect each boundary divisor in at least one point, we get that $g=0$ and $n=2$.
	
	\item If we do not have a point insertion, we still have the boundary insertions, and this time we get $\sum |\gamma_i|\leqslant 1$: the ranks add up to at most $1$ since the fibers still vary in a $1$-dimensional space. We get
	$$n+g-1=\sum |\gamma_i| \leqslant 1.$$
	As $n\geqslant 2$ and $g\geqslant 0$, we still have $n=2$ and $g=0$, and $\sum |\gamma_i|=1$, yielding the expected possibilities.
	\end{enumerate}
	\end{proof}
	
	The computation of these invariants is handled in Section \ref{sec-computation-invariant-boundary}. Before computing these invariants, we further consider the invariants where we insert a $\lambda$-class, allowing us to have curves of bigger genus.
	
	\subsubsection{Invariants with a $\lambda$-class insertion.} We further consider the GW invariants where we allow a $\lambda$-class insertions along with cycle insertions. First, we prove a vanishing result on the top $\lambda$-class for moduli spaces of curves in the class $\varpi=a\sfs+b\sfp$ with $a\neq 0$.
	
	\begin{lem}\label{lem-vanishing-top-lambda-class}
The top $\lambda$-class $\lambda_g$ vanishes on the moduli space of stable maps
$$\M_{g,n}(X/D^\pm,\varpi,\mu),$$
with $\varpi=a\sfs+b\sfp$ and $a\neq 0$. In particular, every GW-invariant involving a $\lambda_g$ is $0$.
\end{lem}

\begin{proof}
We have the projection $\pi:X\to \widetilde{E}$, which can be used to pull-back the holomorphic $1$-form $\dd z$ on $\widetilde{E}$ to each curve in the moduli space. This yields a section of the Hodge bundle. As the map $\pi\circ f:C\to \widetilde{E}$ is non-constant, this section is everywhere non-zero. As the Hodge bundle has a nowhere vanishing section, we deduce that $\lambda_g=0$.
\end{proof}

Due to Lemma \ref{lem-vanishing-top-lambda-class}, we can only insert $\lambda_j$ with $j$ up to $g-1$. We now get an analog result to Lemma \ref{lem-non-zero-enum-inv} in case we insert a $\lambda$-class.

	\begin{lem}\label{lem-non-zero-lambda-inv}
	\begin{enumerate}[label=(\roman*)]
	\item If $a\neq 0$, the only non-zero enumerative invariants with a unique point insertion in the main stratum are
	$$\gen{\lambda_{g-1};\pt_0,1_{\mu_1},\pt_{\mu_2},\cdots,\pt_{\mu_n}}_{g,a\sfs+b\sfp}^{X/D^\pm},$$
	$$\gen{\lambda_{g-1}; \pt_0,\alpha_{\mu_1},\beta_{\mu_2},\pt_{\mu_3},\cdots,\pt_{\mu_n}  }_{g,a\sfs+b\sfp}^{X/D^\pm}.$$
	\item For a class $w\sfp$, the only non-zero invariant with a point insertion is
	$$\gen{\lambda_g;\pt_0,1_{-w},1_w}_{g,w\sfp}^{X/D^\pm}.$$
	\item For a class $w\sfp$, the only non-zero invariants without point insertion are
	$$\gen{\lambda_g;\pt_{-w},1_w}_{g,w\sfp}^{X/D^\pm} \text{ and }\gen{\lambda_g;\alpha_{-w},\beta_w}_{g,w\sfp}^{X/D^\pm},$$
	along with their symmetric ones.
	\end{enumerate}
	\end{lem}
	
	\begin{proof}
	Let $\lambda_j$ be the $\lambda$-insertion.
	\begin{enumerate}[label=(\roman*)]
	\item Assume the class is $a\sfs+b\sfp$ with $a\neq 0$. The only difference with Lemma \ref{lem-non-zero-enum-inv} is that now, we have
	\begin{align*}
	(n+1)+g-1 = & j+2+\sum|\gamma_i| \\
	\leqslant & (g-1)+2+(n-1) \\
	= & g+n.\\
	\end{align*}
	Thus, we have necessarily $j=g-1$ and $\sum |\gamma_i|=n-1$, and we conclude the same way.
	
	\item For a class of the form $w\sfp$, we proceed similarly. This time, we only have $j\leqslant g$ since we are in the case $a=0$, and Lemma \ref{lem-vanishing-top-lambda-class} does not apply. Assume we have an interior point marking. We get
		\begin{align*}
		(n+1)+g-1 = & j+2+\sum |\gamma_i| \\
		\leqslant & g + 2 + 0.\\
		\end{align*}
		As $n\geqslant 2$, we have $n=2$ and $j=g$.
	\item If we do not have an interior point insertion, we get
		\begin{align*}
		n+g-1 = & j+\sum |\gamma_i| \\
		\leqslant g+1.\\
		\end{align*}
		Thus, $n=2$, $j=g$ and $\sum |\gamma_i|=1$.
	\end{enumerate}
	\end{proof}
	
	\subsubsection{Skew-symmetry of the invariants.} We conclude by proving that the GW invariants with $1$-cycle insertions are skew-symmetric, even if the tangency orders are not the same. In the case of usual GW invariants, this follow from the skew-commutativity of the cup product and the symmetry between the marked points. However, in the relative case, the marked points do not play a symmetric role since they are assigned the tangency order with $D^\pm$.
		
	\begin{lem}\label{lem-skew-symmetry-invariants}
Assume $\mu_0,\mu_1\neq 0$. The bilinear function
$$ (\gamma,\gamma')\in H^1(\widetilde{E},\ZZ)^2 \longmapsto \gen{\gamma_{\mu_0},\gamma'_{\mu_1},\bullet}_{g,a\sfs+b\sfp}^{X/D^\pm},$$
where ``$\bullet$" denotes any complementary insertions to define the invariant ($\lambda$-class or point insertions), is skew-symmetric.
	\end{lem}
	
	\begin{proof}
We use automorphisms of $\widetilde{E}$ for a suitable choice of $\widetilde{E}$. Every such automorphism can be lifted to an automorphism of $X=\widetilde{E}\times\PP^1$ which acts trivially on $H_2(X,\ZZ)$. Assume $(\alpha,\beta)$ is a basis of $H^1(E,\ZZ)$. We need to show that the function vanishes on $(\alpha,\alpha)$ and $(\beta,\beta)$, and takes opposite values on $(\alpha,\beta)$ and $(\beta,\alpha)$.
\begin{itemize}[label=$\ast$]
\item Take $\widetilde{E}=\CC/\gen{1,i}$, so that we have the automorphism $\varphi(z)=iz$, which acts in cohomology by the matrix $\left(\begin{smallmatrix} 0 & -1 \\ 1 & 0 \\ \end{smallmatrix}\right)$. Using the basis $(\alpha,\beta)$, we get that
	\begin{align*}
	\int\ev_0^*(\alpha)\ev_1^*(\beta) = & \int \ev_0^*(\beta)\ev_1^*(-\alpha) \\
	= & -\int \ev_0^*(\beta)\ev_1^*(\alpha). \\
	\end{align*}
\item Take $\widetilde{E}=\CC/\gen{1,\rho}$ which possesses an order $6$ automorphism acting in cohomology via the matrix $\left(\begin{smallmatrix} 0 & -1 \\ 1 & 1 \\ \end{smallmatrix}\right)$. Thus, we get
	\begin{align*}
	\int\ev_0^*(\alpha)\ev_1^*(\beta) = & \int \ev_0^*(\beta)\ev_1^*(-\alpha+\beta) \\
	= & -\int \ev_0^*(\beta)\ev_1^*(\alpha) + \int \ev_0^*(\beta)\ev_1^*(\beta). \\
	\end{align*}
	Using equation from the first point, we deduce that $\int \ev_0^*(\beta)\ev_1^*(\beta)=0$. 
\item Similarly, using the square of the last automorphism, we have that
\begin{align*}
	\int\ev_0^*(\alpha)\ev_1^*(\beta) = & \int \ev_0^*(-\alpha+\beta)\ev_1^*(-\alpha) \\
	= & -\int \ev_0^*(\beta)\ev_1^*(\alpha) + \int \ev_0^*(\alpha)\ev_1^*(\alpha), \\
	\end{align*}
	so that $\int \ev_0^*(\alpha)\ev_1^*(\alpha)=0$.
\end{itemize}
\end{proof}

	\subsection{Computations of invariants with boundary constraints}
	\label{sec-computation-invariant-boundary}
	
	We finish this section about GW invariants by computing all the needed relative GW invariants of $X=\widetilde{E}\times\PP^1$. The latter are elementary bricks that we need to compute the GW invariants of bielliptic surfaces. Most of the computations can already be found elsewhere (see \cite{blomme2021floor}, \cite{bousseau2021floor} or \cite{oberdieck2023quantum} for instance), but we include them out of completeness.

\subsubsection{Boundary invariants with point insertions.} We start by computing the invariants where we do not have any $1$-cycle insertions. To express the result for the invariants with a $\lambda$-class insertion, we set the following:
$$\sfR_{a\sfs+b\sfp}(\mu)=\sum_{g=1}^\infty \gen{\lambda_{g-1};\pt_0,1_{\mu_1},\pt_{\mu_2},\cdots,\pt_{\mu_n} }^{X/D^\pm}_{g,a\sfs+b\sfp} u^{2g-2+n} ,$$
and through the change of variable $q=e^{iu}$ and $[k]=q^{k/2}-q^{-k/2}$, for $\mu$ an integer vector with sum $0$ and non-zero entries, we set
$$\sfR_a(\mu)=(-i)^n\sum_{k|a}\left(\frac{a}{k}\right)^{n-1}\prod_j\frac{\left[k|\mu_j|\right]}{|\mu_j|}.$$
We do not emphasize the dependence in $u$ or $q$ in the notation.

\begin{prop}\label{prop-computation-bdry-inv-pt-case}
We have the following equalities:
\begin{align*}
\gen{\pt_0,1_{\mu_1},\pt_{\mu_2},\cdots,\pt_{\mu_n}}^{X/D^\pm}_{1,a\sfs+b\sfp}= & \mu_1^2\cdot a^{n-1}\sigma_1(a) , \\
\sfR_{a\sfs+b\sfp}(\mu) = & \mu_1^2 \cdot \sfR_a(\mu) . \\
\end{align*}
\end{prop}

\begin{rem}
We have several possible methods. First, it is possible to use the double ramification cycle formula with a target variety from \cite{janda2020double}. This strategy was actually carried out independently and quite at the same time by G. Oberdieck and A. Pixton in \cite[Theorem 6.10]{oberdieck2023quantum}. Alternatively, we can go to the tropical limit and use the correspondence theorem from \cite{blomme2021floor}, which we do for the $\lambda$-class insertion using the correspondence results from \cite{bousseau2019tropical}. For the invariant without $\lambda$-class, it is also possible to make a direct computation, since the moduli space of genus $1$ stable maps to the considered $\PP^1$-bundles is particularly simple.
\end{rem}

\begin{proof}
A genus $1$ stable map $f:C\to X$ in the class $a\sfs+b\sfp$ with $a\neq 0$ yields a degree $a$ cover of $\widetilde{E}$, and a meromorphic function over the latter. Conversely, given a cover $g:C\to \widetilde{E}$ together with a meromorphic function, we get a stable map from $C$ to $X$. The intersection profile with $D^\pm$ is given by the zeros and poles of the function.

There are $\sigma_1(a)$ covering maps of degree $a$ of $\widetilde{E}$ by a genus $1$ curve. If we fix the image of the marked point inside the main stratum, there are no automorphism. Let $g:C\to\widetilde{E}$ be one of these covers. To get the map to $X$ with prescribed intersection profile with $D^\pm$, we only need to choose the meromorphic function. For each point constraint on the boundary, we multiply by $a$, since the map to $E$ is of degree $a$.

Let $z_i$ be the position of the $i$th marked point on $C$. The position of the last point is determined by the fact that the divisor $\sum \mu_i (z_i)$ on $\widetilde{E}$ needs to be principal. Thus, the divisor $\mu_1 (z_1)$ is fixed, and the position of $z_1$ is determined up to $\mu_1$-torsion, and we get $\mu_1^2$ positions. In the end, we get
$$a^{n-1} \sigma_1(a)\cdot \mu_1^2.$$

To compute the invariant with a $\lambda$-class insertion, we go to the tropical limit using the degenerations presented in \cite{blomme2021floor} and the correspondence theorem from \cite{bousseau2019tropical} which relies on the degeneration formula from \cite{abramovich2020decomposition}. Theorem $1$ from \cite{bousseau2019tropical} can indeed be applied to the setting of projective completions of line bundles over an elliptic curve considered in \cite{blomme2021floor}: in \cite[Section 4.2]{blomme2021floor} is constructed a family of these surfaces whose central fiber is a union of toric surfaces glued along their toric divisors. As the total space is locally toric threefold with a monomial projection, it enjoys every property needed along the proof of Theorem 1 in \cite{bousseau2019tropical}.
\end{proof}

\begin{rem}
In \cite{bousseau2021floor}, P. Bousseau gives a direct computation of some relative GW invariants with $\lambda$-classes relying on an induction relation obtained by a blow-up algorithm, and initialization for certain genus $0$ curves in the blown-up surfaces, thus avoiding the use of tropical methods. Such a direct computation might be possible here provided one is able to compute initial values, which would be for genus $1$ curves.
\end{rem}

\subsubsection{Boundary invariants with $1$-cycle insertions.} We now get to the computation of the invariants
$$\gen{\pt_0,\alpha_{\mu_1},\beta_{\mu_2},\pt_{\mu_3},\cdots,\pt_{\mu_n}}^{X/D^\pm}_{g,\varpi},$$
and its analog with a $\lambda$-class insertion. Similarly to the point case, we set 
$$\widetilde{\sfR}_{a\sfs+b\sfp}(\mu)(\alpha,\beta)=\sum_{g=1}^\infty \gen{\lambda_{g-1};\pt_0,\alpha_{\mu_1},\beta_{\mu_2},\pt_{\mu_3},\cdots,\pt_{\mu_n} }^{X/D^\pm}_{g,a\sfs+b\sfp} u^{2g-2+n} .$$
Notice that both are skew-symmetric function in the classes $(\alpha,\beta)$ following Lemma \ref{lem-skew-symmetry-invariants}, and that is why we may only compute them on a basis of $H^1(\widetilde{E},\QQ)$. To compute the GW invariants with $1$-cycle insertions and a $\lambda$-class, we use a vanishing result.

%\label{lem-computation-vertex-lambda-case-1-cycle}
\begin{prop}\label{prop-computation-bdry-inv-1-cycle-case}
We have the following identities:
\begin{align*}
\gen{\pt_0,\alpha_{\mu_1},\beta_{\mu_2},\pt_{\mu_3},\cdots,\pt_{\mu_n}}^{X/D^\pm}_{g,\varpi} = & -\mu_1\mu_2\cdot a^{n-1}\sigma_1(a),\\
\widetilde{\sfR}_{a\sfs+b\sfp}(\mu)(\alpha,\beta)= & -\mu_1\mu_2\cdot\sfR_a(\mu) .\\
\end{align*}
\end{prop}

\begin{rem}
In particular, the sign depends on the position of $\mu_1$ and $\mu_2$: it is positive if they do not belong to the same divisor, and negative if they lie on the same component.
\end{rem}

\begin{proof}
We have $n+1$ marked points, labeled by $[\![0;n]\!]$, only the first one is mapped to $\widetilde{E}\times\PP^1$, the other one are mapped to $D^\pm$. To lighten notation, we replace the indices $\mu_i$ of the insertions by $i$.

There is a relation between the position of the intersection points between a curve and $D^\pm$ (see proof of Proposition \ref{prop-computation-bdry-inv-pt-case} or \cite{blomme2021floor}). More precisely, if we denote by $(z_i)_{1\leqslant i\leqslant n}$ the position of the marked point, we have
$$\sum_1^n\mu_i z_i\equiv 0\in\widetilde{E}.$$
In other terms, the image by the evaluation map of the boundary points in $\widetilde{E}^n$ is contained in the hypersurface defined by the above equation, which is Poincar\'e dual to a class $\Upsilon\in H^2(\widetilde{E}^n,\QQ)$. The latter can be obtained as the pull-back of the fundamental class of $\widetilde{E}$ by the map
$$\sigma:(z_i)\in\widetilde{E}^n\mapsto\sum_1^n\mu_iz_i\in\widetilde{E}.$$
As $\Upsilon^2=0$, when we integrate $\ev^*\Upsilon$, we get $0$. In particular, we get the relation
$$\int_{\M_{g,n}(X/D^\pm,\varpi,\mu)}\lambda_{g-g_0}\ev^*(\Upsilon\pt_0\pt_3\cdots\pt_n)=0.$$

We now reformulate the relation by expanding the class $\Upsilon$. We use indices from $0$ to $n$ to denote the cohomology class relative to the $i$th marked point. In other words, the cohomology algebra relative to the $n$ boundary points $H^*(\widetilde{E}^n,\QQ)$ is generated by $\alpha_k,\beta_k$ of degree $1$ with the relations $\alpha_k^2=\beta_k^2=0$ and we have $\pt_k=\alpha_k\beta_k$. As $\Upsilon=q^*(\alpha\beta)$, the class $\Upsilon$ is computed as follows:
$$\Upsilon = \left( \sum_1^n\mu_i\alpha_i\right)\left( \sum_1^n\mu_i\beta_i\right).$$
Multiplying by $\pt_3\cdots\pt_n=\alpha_3\beta_3\cdots\alpha_n\beta_n$, we can delete any $\alpha_k,\beta_k$ with $k\neq1,2$ from the expression of $\Upsilon$, so that we get
\begin{align*}
\Upsilon\pt_3\cdots\pt_n= & (\mu_1\alpha_1+\mu_2\alpha_2)(\mu_1\beta_1+\mu_2\beta_2)\pt_3\cdots\pt_n \\
= & (\mu_1^2\pt_1+\mu_2^2\pt_2+\mu_1\mu_2\alpha_1\beta_2+\mu_1\mu_2\alpha_2\beta_1)\pt_3\cdots\pt_n.\\
\end{align*}
To conclude, we multiply with $\lambda_{g-g_0}$ and integrate over $\vir{\M_{g,n}(X/D^\pm,\varpi,\mu)}$. We have the following identity between relative GW invariants, where to lighten the relation, we only write the classes associated to the marked points $1$ and $2$, and not the $\lambda_{g-g_0}$, nor the classes $\pt_0$, $\pt_3\cdots\pt_n$:
$$\mu_1^2\gen{\pt_1,1_2}^{X/D^\pm}_{g,\varpi}+\mu_2^2\gen{1_1,\pt_2}^{X/D^\pm}_{g,\varpi}+\mu_1\mu_2\gen{\alpha_1,\beta_2}^{X/D^\pm}_{g,\varpi}-\mu_1\mu_2\gen{\beta_1,\alpha_2}^{X/D^\pm}_{g,\varpi}=0.$$
The minus sign in the last term comes from the skew-symmetry: $\alpha_2\beta_1=-\beta_1\alpha_2$. Using skew-symmetry for the invariants proven in Lemma \ref{lem-skew-symmetry-invariants}, we have finally
$$\mu_1\mu_2\gen{\alpha_1,\beta_2}^{X/D^\pm}_{g,\varpi} = -\mu_1^2\gen{\pt_1,1_2}^{X/D^\pm}_{g,\varpi}-\mu_2^2\gen{1_1,\pt_2}^{X/D^\pm}_{g,\varpi} .$$
The right-hand side has already been computed in Proposition \ref{prop-computation-bdry-inv-pt-case}, which allows us to conclude.
\end{proof}

\subsubsection{GW-invariants for fiber classes.} The last relative GW invariants we need are the one in the class $w\sfp$. They can be computed using \cite[Lemma 3.2,3.3]{bousseau2021floor}. Recall $(\alpha,\beta)$ is a basis of $H^1(\widetilde{E},\ZZ)$.

\begin{prop}\label{prop-computation-fibers}
We have
$$\gen{\lambda_g;\pt_0,1_{-w},1_w}^{X/D^\pm}_{g,w\sfp} = \left\{ \begin{array}{l}
 1 \text{ if }g=0,\\
 0 \text{ if }g>0.\\
\end{array}\right.$$
$$\gen{\lambda_g;1_{-w},\pt_w}^{X/D^\pm}_{g,w\sfp} = \gen{\lambda_g;\alpha_{-w},\beta_w}^{X/D^\pm}_{g,w\sfp} = \left\{ \begin{array}{l}
 \frac{1}{w} \text{ if }g=0,\\
 0 \text{ if }g>0.\\
\end{array}\right.$$
\end{prop}

\begin{proof}
We follow the proof from \cite[Lemma 3.2,3.3]{bousseau2021floor}, to which we refer for more details. The curve class $\sfp$ satisfies $\sfp^2=0$, and we have a $1$-dimensional family of curves in the class $\sfp$, although they do not form a linear system. Every connected curve in the class $w\sfp$ factors through some fiber of the projection $\pi:X\to \widetilde{E}$, \textit{i.e.} a curve $C$ in the class $\sfp$. As the normal bundle $N_C$ to $C$ is trivial, the perfect obstruction theories for stable maps to $C$ and to $X$ differ by the top Chern class of the bundle whose fiber over the relative stable map $f:\CCC\to X$ is $H^1(\CCC,f^*N_{C})$, which is the dual of the Hodge bundle by Serre duality. Thus, we get the vanishing of the invariants with a top $\lambda$-class, since $\lambda_g^2=0$ when $g>0$.

To compute the invariants in each of the remaining situations, we see that there is a unique fiber in case of a point insertion, and $\alpha\cdot\beta=1$ fiber if we have $1$-cycle insertions. Then, we have a unique stable map $f:\PP^1\to C$ of degree $w$ fully ramified over $0$ and $\infty$. If we have a marked point in the main strata, it kills any automorphism. If we do not, then we have an automorphism group $\ZZ/w\ZZ$ acting by multiplication by roots of unity, so that each solution contributes $\frac{1}{w}$.
\end{proof}

\section{Computation of the GW invariants of bielliptic surfaces}

We now get to one of the main section of the paper, where we give a way to compute the GW invariants of bielliptic surfaces. The idea is to apply the degeneration formula from \cite{li2002degeneration,kim2018degeneration} using the degenerations presented in Section \ref{sec-construction-degeneration}. According to \cite{li2002degeneration}, the invariants express as a sum over some discrete data, which in our situation are basically tropical covers of the circle, with suitable weights. Through this discrete data, the GW invariants of bielliptic surfaces are reduced down to the relative GW invariants of $\widetilde{E}\times\PP^1$, which have been computed in Section \ref{sec-computation-invariant-boundary}.

	\subsection{Different kinds of diagrams}
	\label{sec-different-kinds-of-diagrams}
	
	The decomposition formula expresses the GW invariants as a sum over a discrete data. In the case presented in \cite{li2002degeneration}, the discrete data consists in a collection of decorated bipartite graphs. This is due to the fact that \cite{li2002degeneration} deals with a family of surfaces that degenerate to a reducible surface consisting only of two irreducible components meeting along a smooth divisor. The two colors of the graph correspond to the two components of the reducible fiber. As the central fibers of the families constructed in Section \ref{sec-construction-degeneration} may have more than two component, the discrete data is more complicated: they consist in some decorated graphs we call \textit{decomposition diagrams}. In our situation, most of the decoration and vertices on the decomposition diagram is uniquely determined, so that the discrete data can be indexed simpler graphs called \textit{pearl diagrams}, similar to the pearl diagrams presented in \cite{blomme2022abelian3}. The link between both kind of diagrams is made through \textit{discerning pearl diagrams}.
	
	\begin{defi}\label{def-pearl-diagram}
	A \textit{diagram} is a connected oriented graph $\PPP$ with vertices labeled by $[\![1;n]\!]$ and the following additional data:
	\begin{enumerate}[label=(\roman*)]
	\item each vertex $V$ is assigned a curve class $\varpi_V=a_V\sfs+b_V\sfp\in H^2(X,\ZZ)$ and a genus $g_V$,
	\item each edge $e$ is assigned a weight $w_e\geqslant 1$ and a height $h_e\geqslant 0$.
	\end{enumerate}
	A vertex with $a_V=0$ is called a flat vertex, and other vertices are called non-flat.
	\begin{itemize}[label=$\circ$]
	\item A diagram is called a \textbf{decomposition diagram} if it satisfies the following:
		\begin{enumerate}[label=(\Alph*)]
		\item[(B)] For any vertex $V$, the sum of the weights of the edges leaving $V$ (resp. arriving at $V$) is equal to $b_V$.
		\item[(H)] For any oriented edge $e$ linking vertices $V_1$ and $V_2$, we are in one of the following situations:
			\begin{itemize}[label=$\ast$]
			\item Labels are $k$ and $k+1$ with $k<n$, and $h_e=0$.
			\item Labels are $n$ and $1$, and $h_e=1$.
			\end{itemize}
		\item[(M)] Among the vertices with a given label $k\in[\![1;n]\!]$, we have a distinguished vertex called \textit{preferred} vertex.
		\end{enumerate}
		
	\item A diagram is called a \textbf{discerning pearl diagram} if it satisfies (B), (H) and the following additional conditions, with (M') replacing (M):
		\begin{enumerate}[label=(\Alph*)]
		\item[(G)] Flat vertices have genus $0$, non-flat vertices have genus $1$.
		\item[(F)] Flat vertices are bivalent.
		\item[(M')] Among the vertices with a given label $k\in[\![1;n]\!]$, we have at most one non-flat vertex, and a distinguished vertex called \textit{preferred} vertex which is the unique non-flat vertex when it exists.
		\end{enumerate}
		
	\item A diagram is called a \textbf{pearl diagram} if it satisfies (B), (G), (F) and the new (H'),(M''):
		\begin{enumerate}[label=(\Alph*)]
		\item[(H')] For an oriented edge $e$ linking vertices $V_1$ and $V_2$ with labels $k_1,k_2\in[\![1;n]\!]$, we are in one of the following situations:
			\begin{itemize}[label=$\ast$]
			\item if $k_1<k_2$, then $h_e\geqslant 0$,
			\item if $k_1\geqslant k_2$, then $h_e\geqslant 1$.
			\end{itemize}
		\item[(M'')] There is exactly one vertex with each possible label $k\in [\![1;n]\!]$.
	\end{enumerate}
	\end{itemize}
	\end{defi}
	
	\begin{rem}
	The name of the conditions stand for the following: (B) for ``Balancing", since it means vertices have zero divergence, (H) for ``Height" since it imposes a condition on the heights of the edges, (M) for ``Marking" since it is related to the choice of a distinguished vertex among the vertices with a given label, (G) for ``Genus" and (F) for ``Flat".
	\end{rem}
	
	\begin{rem}
	The motivation for these definitions is that the discrete data indexing the decomposition formula is naturally the set of decomposition diagrams, with a natural multiplicity. When the latter is non-zero, the decomposition diagram is actually a discerning pearl diagram. Finally, it is possible to delete most vertices of a discerning pearl diagram to get a pearl diagram, much easier to handle. We explicit this connection in the rest of the section.
	\end{rem}
	
	Choose $n$ distinct points $0<x_1<\cdots<x_n<1$ on $\RR/\ZZ$. The data of a diagram $\PPP$ uniquely determines a tropical cover of $\RR/\ZZ$ branched exactly at the $x_i$ (see for instance \cite{buchholz2015tropical} for precise definitions of tropical covers):
	\begin{itemize}[label=$\ast$]
	\item Vertices with label $i$ are mapped to $x_i$.
	\item The image of an oriented edge $e$ linking vertices with labels $k_1$ and $k_2$ goes from $x_{k_1}$ to $x_{k_2}$ following the natural orientation of $\RR/\ZZ$, passing $h_e$ times over $0$.
	\item The condition (B) ensures balancing.
	\end{itemize}
	
	Thus, the diagrams can be seen as tropical covers of $\RR/\ZZ$ with an additional decoration by genus and homology classes.
	
	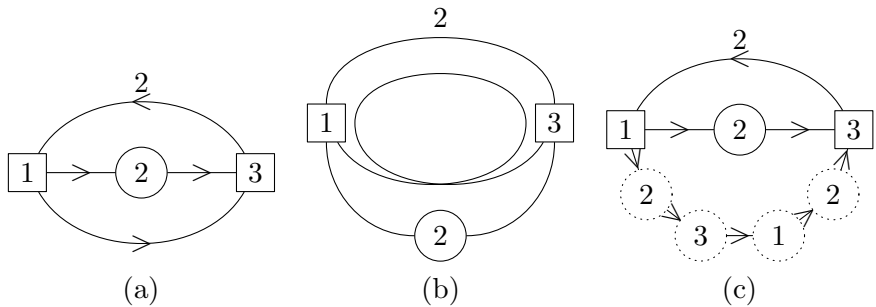
\begin{figure}
	\begin{center}
	\begin{tabular}{ccc}
	\begin{tikzpicture}[line cap=round,line join=round,x=0.75cm,y=0.75cm]
	\pearl (1) at (180:2) label=1;
	\flatpearl (2) at (-90:0) label=2;
	\pearl (3) at (0:2) label=3;
	
	\draw (1) to node[midway,sloped] {$>$} (2);
	\draw (2) to node[midway,sloped] {$>$} (3);
	\draw (1) to[out=-60,in=-120] node[midway,sloped] {$>$} (3);
	\draw (3) to[out=120,in=60] node[midway,sloped] {$<$} node[midway,above] {$2$} (1);
	\end{tikzpicture} & \begin{tikzpicture}[line cap=round,line join=round,x=0.75cm,y=0.75cm]
	\pearl (1) at (180:2) label=1;
	\flatpearl (2) at (-90:2) label=2;
	\pearl (3) at (0:2) label=3;
	
	\draw (1) to[out=-90,in=180] (2);
	\draw (2) to[out=0,in=-90] (3);
	\draw (3) to[out=90,in=90] node[midway,above] {$2$} (1);
	\draw (1) to[out=-60,in=-90] (0:1.5) to[out=90,in=90] (180:1.5) to[out=-90,in=-120] (3);
	\end{tikzpicture} & \begin{tikzpicture}[line cap=round,line join=round,x=0.75cm,y=0.75cm]
	\pearl (1) at (180:2) label=1;
	\flatpearl (2) at (-90:0) label=2;
	\pearl (3) at (0:2) label=3;
	\dottedpearl (D1) at (-145:2) label=2;
	\dottedpearl (D2) at (-110:2) label=3;
	\dottedpearl (D3) at (-70:2) label=1;
	\dottedpearl (D4) at (-35:2) label=2;
	
	\draw (1) to node[midway,sloped] {$>$} (2);
	\draw (2) to node[midway,sloped] {$>$} (3);
	\draw (1) to node[midway,sloped] {$>$} (D1) to node[midway,sloped] {$>$} (D2) to node[midway,sloped] {$>$} (D3) to node[midway,sloped] {$>$} (D4) to node[midway,sloped] {$>$} (3);
	
	\draw (3) to[out=120,in=60] node[midway,sloped] {$<$} node[midway,above] {$2$} (1);
	\end{tikzpicture} \\
	(a) & (b) & (c) \\
	\end{tabular}
	
	\caption{\label{fig-expl-pearl-diagram-1}Example of pearl diagram, map to $\RR/\ZZ$ and associated discerning pearl diagram. Non-flat vertices are depicted as square-box, circle-boxes for flat-vertices, and dotted circles for non-preferred vertices.}
	\end{center}
	\end{figure}

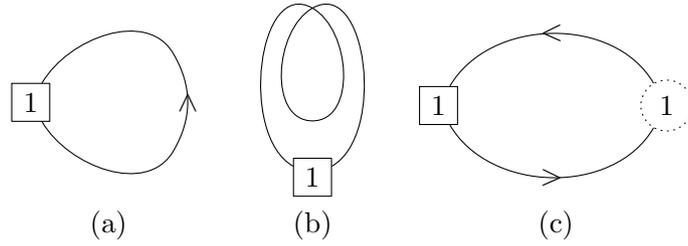
\begin{figure}
	\begin{center}
	\begin{tabular}{ccc}
	\begin{tikzpicture}[line cap=round,line join=round,x=0.75cm,y=0.75cm]
	\pearl (1) at (180:2) label=1;
	\draw (1) to[out=-60,in=-120] (-60:1) to[out=60,in=-60] node[midway,sloped] {$>$} (60:1) to[out=120,in=60] (1) ;
	\end{tikzpicture} & \begin{tikzpicture}[line cap=round,line join=round,x=0.75cm,y=0.75cm]
	\pearl (1) at (-90:2) label=1;
	\draw (1) to[out=30,in=30] (90:1) to[out=-150,in=180] (-90:1) to[out=0,in=-30] (90:1) to[out=150,in=150] (1);
	\end{tikzpicture} & \begin{tikzpicture}[line cap=round,line join=round,x=0.75cm,y=0.75cm]
	\pearl (1) at (180:2) label=1;
	\dottedpearl (D1) at (0:2) label=1;
	\draw (1) to[out=-60,in=-120] node[midway,sloped] {$>$} (D1) to[out=120,in=60] node[midway,sloped] {$<$} (1) ;
	\end{tikzpicture} \\
	(a) & (b) & (c) \\
	\end{tabular}
	
	\caption{\label{fig-expl-pearl-diagram-2}Another example of pearl diagram, map to $\RR/\ZZ$ and associated discerning pearl diagram.}
	\end{center}
	\end{figure}

	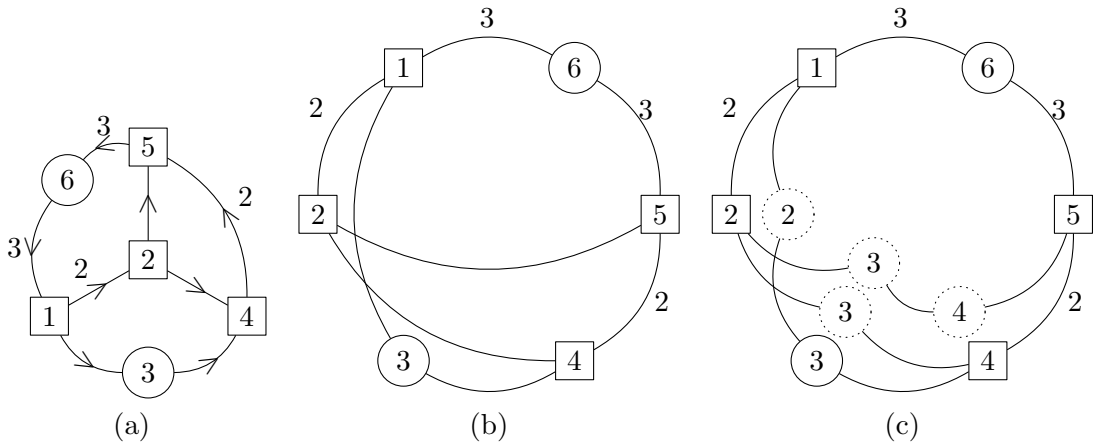
\begin{figure}
	\begin{center}
	\begin{tabular}{ccc}
	\begin{tikzpicture}[line cap=round,line join=round,x=0.75cm,y=0.75cm]
	\pearl (1) at (-150:2) label=1;
	\pearl (2) at (180:0) label=2;
	\flatpearl (3) at (-90:2) label=3;
	\pearl (4) at (-30:2) label=4;
	\pearl (5) at (90:2) label=5;
	\flatpearl (6) at (135:2) label=6;
	\draw (1) to node[midway,sloped] {$>$}  node[midway,above left] {$2$} (2) ;
	\draw (1) to[bend right] node[midway,sloped] {$>$} (3) ;
	\draw (3) to[bend right] node[midway,sloped] {$>$} (4) ;
	\draw (2) to node[midway,sloped] {$>$} (5) ;
	\draw (2) to node[midway,sloped] {$>$} (4) ;
	\draw (4) to[bend right] node[midway,sloped] {$<$}  node[midway,above right] {$2$} (5) ;
	\draw (5) to[bend right] node[midway,sloped] {$<$} node[midway,above] {$3$} (6) ;
	\draw (6) to[bend right] node[midway,sloped] {$<$} node[midway,left] {$3$} (1) ;
	\end{tikzpicture} & \begin{tikzpicture}[line cap=round,line join=round,x=0.75cm,y=0.75cm]
	\pearl (1) at (120:3) label=1;
	\pearl (2) at (180:3) label=2;
	\flatpearl (3) at (-120:3) label=3;
	\pearl (4) at (-60:3) label=4;
	\pearl (5) at (0:3) label=5;
	\flatpearl (6) at (60:3) label=6;
	\draw (1) to[bend right]  node[midway,above left] {$2$} (2) ;
	\draw (1) to[bend right] (3) ;
	\draw (3) to[bend right] (4) ;
	\draw (2) to[bend right] (5) ;
	\draw (2) to[bend right] (4) ;
	\draw (4) to[bend right]  node[midway,right] {$2$} (5) ;
	\draw (5) to[bend right] node[midway,above] {$3$} (6) ;
	\draw (6) to[bend right]  node[midway,above] {$3$} (1) ;
	\end{tikzpicture} & \begin{tikzpicture}[line cap=round,line join=round,x=0.75cm,y=0.75cm]
	\pearl (1) at (120:3) label=1;
	\pearl (2) at (180:3) label=2;
	\dottedpearl (D2) at (180:2) label=2;
	\flatpearl (3) at (-120:3) label=3;
	\dottedpearl (D3) at (-120:2) label=3;
	\dottedpearl (D3bis) at (-120:1) label=3;
	\pearl (4) at (-60:3) label=4;
	\dottedpearl (D4) at (-60:2) label=4;
	\pearl (5) at (0:3) label=5;
	\flatpearl (6) at (60:3) label=6;
	
	\draw (1) to[bend right]  node[midway,above left] {$2$} (2) ;
	\draw (1) to[bend right] (D2);
	\draw (D2) to[bend right] (3) ;
	\draw (3) to[bend right] (4) ;
	\draw (2) to[bend right] (D3bis);
	\draw (D3bis) to[bend right] (D4);
	\draw (D4) to[bend right] (5) ;
	\draw (2) to[bend right] (D3);
	\draw (D3) to[bend right] (4) ;
	\draw (4) to[bend right]  node[midway,right] {$2$} (5) ;
	\draw (5) to[bend right] node[midway,above] {$3$} (6) ;
	\draw (6) to[bend right]  node[midway,above] {$3$} (1) ;
	\end{tikzpicture} \\
	(a) & (b) & (c) \\
	\end{tabular}
	
	\caption{\label{fig-expl-pearl-diagram-3}Yet another example of pearl diagram, map to $\RR/\ZZ$ and associated discerning pearl diagram.}
	\end{center}
	\end{figure}

	\begin{expl}
	On Figure \ref{fig-expl-pearl-diagram-1},\ref{fig-expl-pearl-diagram-2} and \ref{fig-expl-pearl-diagram-3} we depict several diagrams. We do not write the genus of the vertices nor the homology classes, and only the weights $w_e\geqslant 2$ are written. The heights can be read on the depiction of the map to $\RR/\ZZ$ as the number of preimages of $0$, \textit{i.e.} the intersection with a half-line going from center to the top infinity. On (a) we depict the pearl diagram with its orientation. On (b), the tropical cover to $\RR/\ZZ$, which also gives the value of the height. On (c), we depict the discerning pearl diagram, obtained by subdividing the edges.
		\begin{itemize}[label=$\ast$]
		\item For the diagrams on Figure \ref{fig-expl-pearl-diagram-1}, we have two non-flat vertices and one flat vertex for the pearl diagram in (a). Only the edges linking the vertices with label $1$ and $3$ have height $1$. To get the discerning floor diagram, we need to subdivide the edge going from $1$ to $3$ making an additional turn. Assuming non-flat vertices have genus $1$ and flat vertices genus $0$, the genus is $1+1+2=4$. The degree of the tropical cover is $3$, equal to $\sum_e h_ew_e$.
		\item The diagram on Figure \ref{fig-expl-pearl-diagram-2} has genus $2$, and the tropical cover has degree $2$.
		\item For Figure \ref{fig-expl-pearl-diagram-3}, the diagram is more sophisticated. Its genus is $7$, the tropical cover has degree $3$. Several edges need to be subdivided to get the discerning floor diagrams. Only the edge linking vertices with labels $6$ to $1$ has height $1$.
		\end{itemize}
	\end{expl}
	
	We upgrade $\RR/\ZZ$ with its $n$ marked point to a cyclic graph $C_n$ with $n$ vertices labeled by $[\![1;n]\!]$. Given a diagram $\PPP$ with the map $\PPP\to\RR/\ZZ$. If $\PPP$ is a decomposition diagram or a discerning pearl diagram, condition (H) ensures that the map is in fact a graph map to $C_n$, meaning the preimages of vertices are also vertices. This is not necessarily the case for pearl diagrams.
	
	The relation between the different kind of diagrams is as follows:
	\begin{itemize}[label=$\ast$]
	\item A decomposition diagram is a discerning pearl diagram if it satisfies the additional conditions (G), (F), and the enhancement (M') of (M), which gives more condition on the set of vertices having a given label.
	\item Given a pearl diagram, we obtain a discerning pearl diagram as follows: we subdivide the edges by a adding all the preimages of $x_1,\dots,x_n\in\RR/\ZZ$. Each new vertex gets assigned the label of its image by the map to $\RR/\ZZ$, genus $0$ and the class $w\sfp$ if it lies on an edge of weight $w$. Edges linking $n$ to $1$ gets height $1$, and the other height $0$. The preferred vertex with label $k$ is the vertex coming from $\PPP$. This way, condition (M') is satisfied, as well as (H).
	\item Conversely, given a discerning pearl diagram $\PPP$, we can delete every non-preferred (necessarily flat) vertex by merging the adjacent two edges, which share the same weight. The height of an edge obtained through merging is the sum of the heights of the merged edges. After all deletions, condition (M'') is satisfied, as well as (H').
	\end{itemize}
	
	\begin{expl}
	We can observe the correspondence between discerning pearl diagrams and pearl diagrams on Figures \ref{fig-expl-pearl-diagram-1}, \ref{fig-expl-pearl-diagram-2} and \ref{fig-expl-pearl-diagram-3}.
	\end{expl}
	
	To each vertex is assigned a vector $\mu_V\in \ZZ^{n_V(+1)}$ with $n_V$ the valency of $V$, plus one if it is one of the preferred vertices. It consists in the weights of the edges leaving $V$, with a sign according to whether the edge is ingoing or outgoing. If the vertex is one of the preferred vertex, the additional entry is $0$.
	
	\begin{expl}
	On Figure \ref{fig-expl-pearl-diagram-3}, the vertex with label $2$ has associated vector $(0,1,1,-2)$, and the vertex with label $1$ has associated vector $(0,1,2,-3)$.
	\end{expl}
	
	\begin{defi}
	Given a diagram $\PPP$, we define:
	\begin{itemize}[label=$\ast$]
	\item its genus by $g(\PPP)=b_1(\Gamma)+\sum_V g_V$,
	\item its degree $\varpi=a \sfe+b\sff$, with $a=\sum a_V$ and $b$ is the degree of the tropical cover $\PPP\to\RR/\ZZ$, equal to $\sum w_eh_e$ by counting the preimages of $0$.
	\end{itemize}
	\end{defi}
	
	The degree of the tropical cover is the weighted number of preimages at any non-branched point. It does not depend on the chosen point thanks to the balancing condition (B).
	
	\begin{rem}
	If $\PPP$ is a (discerning) pearl diagram, the genus is equal to the first Betti number plus the number of non-flat vertices, since condition (G) states that they are the only vertices having positive genus, equal to $1$.
	\end{rem}
	
	\begin{expl}
	We already computed the degree of the tropical covers for the diagrams in Figures \ref{fig-expl-pearl-diagram-1}, \ref{fig-expl-pearl-diagram-2} and \ref{fig-expl-pearl-diagram-3}, yielding the $b\sff$ part of their degree. The $a\sfe$ part is obtained by adding the $a_V\sfs$ part of the classes associated to each non-flat vertex.
	\end{expl}

	\subsection{Correspondence for enumerative invariants}
	
	In this section, we assign multiplicities to diagrams, so that their count yields the enumerative GW invariants of bielliptic surfaces. This correspondence relies on the decomposition formula.
	
	\subsubsection{Decomposition formula.} We (finally) apply the decomposition formula from \cite{li2002degeneration} (or \cite{kim2018degeneration}) to get a raw multiplicity assigned to each decomposition diagram. Let $S$ be a bielliptic surface, $g\geqslant 2$ and $n=g-1$.
	
	\begin{prop}\label{prop-decomposition}
	The enumerative invariant $\gen{\pt^{g-1} }^S_{g,\varpi}$ decomposes as follows:
	$$\gen{\pt^{g-1} }^S_{g,\varpi} = \sum_{\PPP} \frac{\prod_e w_e}{|\mathrm{Aut}(\PPP)|}\int_{\prod\vir{\M_V}}\prod_{e\in\PPP}\delta_e\prod_1^n\ev_i^*(\pt),$$
	where the sum is over the decomposition diagrams having total genus $g$ and total degree $\varpi$.
	\end{prop}
	
	\begin{proof}
	Using Section \ref{sec-construction-degeneration}, we choose a family $\SSS_t$ of bielliptic surfaces with $\SSS_1=S$ degenerating into a necklace of $n=g-1$ ruled surfaces $X_1\cup X_2\cup \cdots\cup X_n$, with all $X_i$ isomorphic to the same ruled surface $X$. The surface $X_i$ has two divisors $D_i^\pm$, with $D_i^+=D_{i+1}^-$. We choose point constraints such that the $i$th marked point goes to $X_i$. 
	
	The moduli space $\M_{g,n}(\SSS_t,\varpi)$ is endowed with a virtual fundamental class, and the GW invariants are constant in family, so that we can compute them for the central fiber $\SSS_0$, for which the degeneration formula \cite{li2002degeneration} gives an expression as a sum over the decomposition diagrams:
	$$\gen{\pt^{g-1}}_{g,\varpi}^S = \sum_{\PPP}\frac{\prod_e w_e}{|\mathrm{Aut}(\PPP)|}\int_{\vir{\M^\PPP}} \prod_1^n\ev_i^*(\pt),$$
	with $\vir{\M^\PPP}$ the virtual fundamental class associated to $\PPP$ by the decomposition formula. The decomposition diagrams actually encode some specific classes of curves in $\SSS_0$. Let $f:C\to\SSS_0$ be a curve associated to a decomposition diagram $\PPP$:
	\begin{itemize}[label=$\ast$]
	\item vertices are in bijection with the irreducible components of $C$, with the label $k$ meaning the component is mapped to $X_k$,
	\item the preferred vertex among the vertices with a given label corresponds to the component with the marked point $i$,
	\item an edge of weight $w$ between vertices having labels $k$ and $k+1$ corresponds to a node of the curve, mapped to the divisor $X_k\cap X_{k+1}=D_k^+=D_{k+1}^-$, and the components associated to the adjacent vertices have tangency $w$ with the divisor at this node. This may be referred as the \textit{kissing condition} (Thanks to F. Carocci for giving me this very short expression).
	\end{itemize}
	The virtual fundamental class is computed as follows:
	$$\vir{\M^\PPP} = \Delta^!\left(\prod_V \vir{\M_V}  \right),$$
	where $\M_V=\M_{g_V,n_V(+1)}(X,\varpi_V,\mu_V)$, $n_V$ is the valency, plus one if there is a marked point on the component, and the $\Delta^!$ means the capping with the product of diagonal classes by the evaluation morphism. The diagonal class is the product of diagonal classes for each edge of the graph. Thus, we get the desired formula.
	\end{proof}
	
	The formula assigns to each decomposition diagram a multiplicity
	$$m(\PPP)=\frac{\prod_e w_e}{|\mathrm{Aut}(\PPP)|}\int_{\prod\vir{\M_V}}\prod_{e\in\PPP}\delta_e\prod_1^n\ev_i^*(\pt) .$$
	To compute the latter, it suffices to replace each of the diagonal classes $\delta_e$ by its K\"unneth decomposition, expand, and finally split the computation over the various components indexed by the vertices.

	\subsubsection{Diagonal class of an elliptic curve and class insertions.} Let $E$ be an elliptic curve. According to Milnor \cite[Theorem 11.11]{milnor1974characteristic}, the cohomology class $\delta$ Poincar\'e dual to the diagonal $\Delta\subset E\times E$ has the following K\"unneth decomposition in a basis $\beta_i$ of $H^*(E,\ZZ)$:
	$$\delta=\sum (-1)^{\mathrm{rk}\beta_i}\beta_i\otimes \beta_i^\#,$$
	where $\gen{\beta_i\cup\beta_j^\#,[E]}=\delta_{ij}$, meaning they are dual bases. This yields the result in cohomology Using the basis $(1,\alpha,\beta,\pt)$, this yields the following expression:
	\begin{align*}
	\delta = & 1\otimes\pt+\pt\otimes 1-\alpha\otimes\beta+\beta\otimes\alpha \in H^2(E,\ZZ)\otimes H^2(E,\ZZ) \\
	= & \pt_1+\pt_2-\alpha_1\beta_2-\alpha_2\beta_1 .\\
	\end{align*}
	
	\begin{rem}
	Notice that the diagonal class is invariant by the swapping automorphism that exchanges both coordinates using the skew-commutativity.
	\end{rem}
	
	We now define the notion of \textit{insertion} for a diagram $\PPP$, which will help us express the multiplicity of a decomposition diagram. A \textit{flag} for a graph is a pair $(V,e)$ with $V\in e$.
	
	\begin{defi}
	An insertion for a diagram $\PPP$ is the assignment of a cohomology class $\gamma_e^\pm \in H^*(E,\ZZ)$ to each flag of $\PPP$, such that the two cohomology classes $(\gamma_e^+,\gamma_e^-)$ assigned to the two flags containing the edge $e$ are dual to each other: $\gamma_e^+\otimes\gamma_e^-$ appears in the K\"unneth decomposition of $\delta$.
	\end{defi}	
	
	For each summand $\gamma_e^+\otimes\gamma_e^-$ appearing in the K\"unneth decomposition of the diagonal $\delta_e$, $\gamma_e^+$ and $\gamma_e^-$ are assigned to the extremities of the edge $e$, \textit{i.e.} the flags containing the edge. Thus, the multiplicity is a sum over the possible cycle insertions of the diagram. In particular, the multiplicity associated to a given decomposition diagram $\PPP$ is a sum over all the possible class insertions, and for each choice of class insertion, the integral over $\prod\vir{\M_V}$ splits.

	\subsubsection{From decomposition diagrams to discerning pearl diagrams.}\label{sec-reduction-family-diagram} Before getting to the explicit computation of $m(\PPP)$, we prove that if the multiplicity is non-zero, then the decomposition diagram is actually a discerning pearl diagram.
	
		\begin{lem}\label{lem-genus-vertices-enum-case}
		Let $\PPP$ be a decomposition diagram such that $m(\PPP)\neq 0$. Then we have the following:
		\begin{enumerate}[label=(\roman*)]
		\item flat vertices (with $a_V=0$) are bivalent,
		\item there is at most one non-flat vertex with a given label and it is the preferred one,
		\item non-flat vertices have genus $1$ and flat vertices have genus $0$.
		\end{enumerate}
		In particular, $\PPP$ is a discerning pearl diagram.
		\end{lem}
		
		\begin{proof}
		We replace each diagonal class $\delta_e$ by its K\"unneth decomposition and we expand. This amounts to insert cohomology classes at each flag such that the classes associated to an edge are dual to each other. We are reduced to computation of relative invariants for the surfaces $X_k$. By assumption, since only one marked point is mapped to each $X_i$, for each such invariant we have at most one point insertion in the main stratum.
		\begin{enumerate}[label=(\roman*)]
		\item The only non-zero invariants for the classes $w\sfp$ are achieved for valency $2$. Thus, flat vertices have valency $2$.
		\item Invariants for a non-flat vertex are $0$ unless there is at least a point insertion due to the $\CC^*$-action.
		\item For a non-flat vertex, the genus is at least $1$ for the moduli space of curves being non-empty, and the only non-zero invariants are obtained for $g_V=1$, as proved in Lemma \ref{lem-non-zero-enum-inv}. Similarly, for a flat vertex, the only non-zero invariants are obtained for $g_V=0$.
		\end{enumerate}
		\end{proof}
	
	\begin{rem}
	A priori, there seems to be a lot of possible insertions: $4$ power the number of edges of $\PPP$. However, most of them do not contribute to the multiplicity, as we see in Lemma \ref{lem-position-marked-flat-vertices-enum-case}.
	\end{rem}
		
	According to Lemma \ref{lem-genus-vertices-enum-case}, every decomposition diagram with non-zero multiplicity is a discerning pearl diagram. In particular, all the flat vertices have genus $0$, so we can delete them and get a pearl diagramwith only $n=g-1$ vertices. Let $\V_\mathrm{flat}$ be the set of flat vertices. As non-flat vertices have genus $1$, we have
		$$n+1 = g = (n-|\V_\mathrm{flat}|)+b_1(\PPP),$$
		so that there are precisely $b_1(\PPP)-1$ flat vertices in the pearl diagram. The following lemma describes the relative position of marked flat vertices on $\PPP$.
		
		\begin{lem}\label{lem-position-marked-flat-vertices-enum-case}
		Let $\PPP$ be a pearl diagram such that the associated discerning pearl diagram has non-zero multiplicity, then each connected component of the complement flat vertices in $\PPP$ possesses a unique cycle.
		\end{lem}
		
		\begin{proof}
		Let $\PPP$ be a pearl diagram and let $\overline{\PPP}$ be the associated discerning pearl diagram. Consider a flat vertex $V$ of $\PPP$, which is associated to some marked point, where we have a point insertion. By Lemma \ref{lem-non-zero-enum-inv}(ii), the only way to get a non-zero invariant is to insert the neutral class $1$ at each of the adjacent flags. We thus need to insert point classes at the flags adjacent to neighbors of $V$. Therefore, we can disconnect the graph at these marked bivalent vertices, replacing the bivalent vertex by a pair of univalent vertices, with a point class inserted. As there are $b_1(\PPP)-1$ flat vertices in $\PPP$, the Euler characteristic of the resulting graph is
		$$1-b_1(\PPP)+(b_1(\PPP)-1)=0.$$
		
		\smallskip
		
		We now consider a connected component $\PPP_k$ of this new graph. Let $1-g_k$ be its Euler characteristic, with $g_k\geqslant 0$ its genus. By additivity, the sum of Euler characteristics is equal to $0$. We wish to show that $1-g_k\leqslant 0$, so that they all are equal to $0$ and thus $g_k=1$ for every component $\PPP_k$.
		
		\smallskip
		
		Assume that $1-g_k=1$, so that the component has no cycle: it is a tree. We already have point insertions at all the outgoing flags, previously adjacent to the cut edges. We can inductively prune the tree.
		\begin{itemize}[label=$\ast$]
		\item Assume $V$ is a non-flat vertex with every adjacent flag but one assigned a point insertion. By Lemma \ref{lem-non-zero-enum-inv}(i), to get a non-zero local invariant, we need to insert the neutral class at the remaining flag. This forces to insert a point class at the flag adjacent to the neighbor of $V$, and we can prune the vertex.
		\item If $V$ is a flat vertex (without marked point since those have already been deleted), we proceed similarly using Lemma \ref{lem-non-zero-enum-inv}(iii), as the only way to get a non-zero local invariant is to insert the neutral class at the other flag.
		\end{itemize}
		At some point, we are left with a unique vertex with every adjacent flag assigned a point insertion, which leads to a $0$ invariant for dimensional reasons. Thus, we need at least one cycle: $1-g_k\leqslant 0$. As $\sum_i (1-g_k)=0$, then every $g_k$ is equal to $1$ and we get the result.
		\end{proof}
		
		\begin{rem}
		Notice furthermore that through the proof of Lemma \ref{lem-position-marked-flat-vertices-enum-case}, the insertions outside the cycles are uniquely determined (any other choice yields $0$ contribution).
		\end{rem}
		
		So far, we did not prove that the multiplicity is non-zero since we only deduced necessary conditions for the latter to be non-zero. This vastly reduces the number of diagrams, since we only need to care about the pearl diagrams satisfying the condition given by Lemma \ref{lem-position-marked-flat-vertices-enum-case}

		\subsubsection{Computation of the diagram multiplicity.}\label{sec-computation-multiplicity} Up to now, we did not specify which kind of bielliptic surface the family was since it did not affect the family of diagrams. However, the multiplicity of a diagram does depend on which kind of bielliptic surface we look. This only matters concerning the monodromy going around the direction of the the necklace of $\PP^1$. Indeed, the divisors $D_i^\pm$ are all identified with the elliptic curve $\widetilde{E}$, which is the quotient of the elliptic fiber $E$ by the translational part of the group $G$. However, there is a monodromy which affects the cohomology classes, as it acts non-trivially on $H^1(E,\ZZ)$. We denote by $M$ this monodromy in a chosen basis of $H_1(E,\ZZ)$. It takes the following values, according to whether $\SSS_t$ is of type $(a)$, $(b)$, $(c)$ or $(d)$:
		$$M=\begin{pmatrix}
		-1 & 0 \\
		0 & -1 \\
		\end{pmatrix},\ \begin{pmatrix}
		0 & -1 \\
		1 & -1 \\
		\end{pmatrix},\ \begin{pmatrix}
		0 & -1 \\
		1 & 0 \\
		\end{pmatrix},\text{ or } \begin{pmatrix}
		0 & -1 \\
		1 & 1 \\
		\end{pmatrix},$$
	whicha re of order $n=2,3,4,6$. If we were to consider Abelian surfaces, the monodromy would be the identity matrix. The action on $H^1(M,\ZZ)$ is given by ${^t}M$. We set $\tau_n(h)=2-\mathrm{tr}M^h$, where $\mathrm{tr}$ denotes the trace of a matrix. It is possible to check that the values are indeed the values given in the introduction. We forget about the index and write only $\tau_n$ if the context does not require it.

		\begin{prop}\label{prop-computation-multiplicity}
		Let $\PPP$ be a pearl diagram. For each component $\PPP_k$ in the complement of flat vertices, let $\nu_k\in\ZZ\simeq\pi_1(\RR/\ZZ)$ be the class realized by its unique cycle. The multiplicity is given by
		$$m(\PPP)=\prod_k\tau(\nu_k)\prod_V a_V^{n_V-1}\sigma_1(a_V)\prod_{E_\mathrm{flat}} w_e \prod_{E\backslash E_\mathrm{flat}} w_e^3,$$
		where the first product is over the the components $\PPP_k$, second product over vertices, third product over the set $E_\mathrm{flat}$ edges adjacent to a flat vertex, and last over the remaining edges.
		\end{prop}
		
		\begin{proof}
		To compute the multiplicity, we expand each diagonal class to get a sum over all the possible insertions. Let $\overline{\PPP}$ be the associated discerning pearl diagram.
		
		\smallskip
		
	\textbf{Step 1: reducing down to the cycles.} We first proceed as in the proof of Lemma \ref{lem-position-marked-flat-vertices-enum-case} by cutting at flat vertices of $\PPP$ and pruning the branches. All the insertions outside the cycles are uniquely determined.
		\begin{itemize}[label=$\ast$]
		\item When cutting a flat vertex, we have a $w_e^2$ factor appearing for both adjacent edges since $\gen{\pt_0,1_{-w},1_w}^{X/D^\pm}_{0,w\sfp}=1$, and we have two adjacent edges.
		\item When pruning a (non-preferred) flat vertex of $\overline{\PPP}$, we still have a $w$ factor for the remaining adjacent edge, but as
		$$\gen{\pt_{-w},1_w}^{X/D^\pm}_{0,w\sfp}=\gen{1_{-w},\pt_w}^{X/D^\pm}_{0,w\sfp}=\frac{1}{w},$$
		due to automorphisms, both cancel out, so that we may just not care about these vertices.
		\item Last, when we prune a non-flat vertex, the invariant has value
		$$\gen{\pt_0,1_{\mu_1},\pt_{\mu_2},\cdots,\pt_{\mu_n}}^{X/D^\pm}_{1,a_V\sfs+b_V\sfp} = \mu_1^2\cdot a_V^{n_V-1}\sigma_1(a_V).$$
		We have an additional $|\mu_1|$ for the edge, yielding the expected exponent $3$, and we can prune the vertex.
		\end{itemize}
	
	\medskip
		
	\textbf{Step 2: determining the monodromy around the loops.} Once pruned all the possible vertices, we are left we a disjoint union of cycles. Consider one of them and label its vertices and edges by $[\![1;L]\!]$. The monodromy map in homology between the elliptic fiber $E_1$ at the vertex $1$ and $E_L$ at the vertex $l$ is given by the group action of bielliptic surfaces. Let $M$ denote the action on homology:
		$$M:H_1(E_1,\ZZ)\to H_1(E_L,\ZZ),$$
		which is the matrix of the multiplication by $-1$, $i$, $\rho$ or $-\rho^2$. We need to take the $\nu$-power where $\nu\in\ZZ\simeq\pi_1(\Gamma_n)$ is the class realized by the cycle. Thus, the monodromy in cohomology is given by the transpose:
		$${^t}M^\nu:H^1(E_L,\ZZ)\to H^1(E_1,\ZZ),$$
		where $E_L$ is the fiber over the $L$-vertex, and $E_1$ the fiber over the $1$ vertex.
		
	\medskip
		
	\textbf{Step 3: taking care about each cycle.} The pruning algorithm stops due to the presence of cycles in the components of the complement of marked flat vertices. To finish the computation, we need to insert one of the four terms in the expression of the diagonal class at each edge of the cycle, expanding $\prod_e \delta_e$. The evaluation map at the two marked points corresponding to the edges adjacent to the vertex $i$ are denoted by $\ev_i^+$ and $\ev_i^-$. The product of diagonal classes at the each edge takes the following form:
	$$\prod_1^{L-1} (\ev_j^+\times\ev_{j+1}^-)^*(\delta_j) \times (\ev_L^+\times\ev_1^-)^*(\delta_L).$$
	We have the following expression for the diagonal classes:
	\begin{align*}
	\delta_j = & \pt_j^+ + \pt_{j+1}^- - \alpha_j^+\beta_{j+1}^- + \beta_j^+\alpha_{j+1}^-,\\
	\delta_L = & \pt_L^+ + \pt_1^- - \alpha_L^+ f^*(\beta_1^-) + \beta_L^+ f^*(\alpha_1^-).\\
	\end{align*}
	The indices $j$ and exponents $\pm$ are here to recall to which marked point the class refers: $\bullet_j^\pm$ refers to a marked point mapped to $D_j^\pm$. The application $f$ denotes the monodromy around the cycle and has been determined in the second step.
	
	\medskip
	
	\textbf{Step 4: finding the insertions with non-zero contribution.} We expand the product $\prod_j\delta_j$ by replacing each diagonal class by its K\"unneth decomposition. Lemma \ref{lem-non-zero-enum-inv} shows that given an insertion $\gamma_j^\pm$, there is a unique insertion $\gamma_j^\mp$ leading to a non-zero invariant
	$$\gen{\pt_0,(\gamma_j^-)_{w_{j-1}},(\gamma_j^+)_{w_j},\pt_\bullet,\cdots,\pt_\bullet}^{X/D^\pm}_{1,a_V\sfs+b_V\sfp} \text{ and }\gen{(\gamma_j^-)_{w_{j-1}},(\gamma_j^+)_{w_j}}^{X/D^\pm}_{0,w_j\sfp}.$$
	More precisely, choose the insertion $\gamma_1^+$: either $1$, $\pt$, $\alpha$ or $\beta$.
		\begin{itemize}[label=$\ast$]
		\item If $\gamma_1^+=\pt$, then $\gamma_2^-=1$ since they are dual basis elements, imposed by the diagonal insertion for the edge between the vertices $1$ and $2$. Then, the only way to get a non-zero invariant at the vertex $2$ is to have $\gamma_2^+=\pt$. Thus, the point constraint propagates. The total insertion is
		\begin{align*}
		 & \ev_1^+(\pt)\ev_2^-(1)\ev_2^+(\pt)\cdots\ev_L^-(1)\ev_L^+(\pt)\ev_1^-(1) \\
		= & \ev_1^-(1)\ev_1^+(\pt)\ev_2^-(1)\ev_2^+(\pt)\cdots\ev_L^-(1)\ev_L^+(\pt). \\
		\end{align*}
		%We get
		%$$\prod_1^L a_j^{n_j-1}\sigma_1(a_j)\prod_1^L w_j^3.$$
		
		\item We conclude similarly if $\gamma_1^+=1$, since we get $\gamma_2^-=\pt$ and $\gamma_2^+=1$ and the insertion also propagates, alternating $1$ and $\pt$.% We get the same contribution.
		
		\item We now assume that the insertion $\gamma_1^+$ is $\alpha$. As the term appearing in the diagonal class is $-\alpha\otimes\beta$, we deduce that $\gamma_2^-=-\beta$. The only insertion that yields a non-zero local invariant at the vertex $2$ is $\alpha$, and we can also propagate the insertion. The difference is that there is some monodromy when going around the cycle: we get
		$$\ev_1^+(\alpha)\ev_2^-(-\beta)\ev_2^+(\alpha)\cdots\ev_L^-(-\beta)\ev_L^+(\alpha)\ev_1^-(f^*(-\beta)).$$
		As all the insertions are $1$-cycle, they anti-commute. Thus, when putting the $\ev_1^-(f^*(-\beta))$ back in front of the product, we get a factor $(-1)^{2L-1}=-1$, yielding:
		$$-\ev_1^-(f^*(-\beta))\ev_1^+(\alpha)\ev_2^-(-\beta)\ev_2^+(\alpha)\cdots\ev_L^-(-\beta)\ev_L^+(\alpha).$$

		\item If we had $\gamma_1^+=\beta$, as $\beta\otimes\alpha$ appears in the diagonal class, we get that $\gamma_2^-=\alpha$. Thus, the only insertion yielding a non-zero invariant at the vertex $2$ is $\gamma_2^+=\beta$, and we can propagate. In the end, we get
		\begin{align*}
		 & \ev_1^+(\beta)\ev_2^-(\alpha)\ev_2^+(\beta)\cdots\ev_L^-(\alpha)\ev_L^+(\beta)\ev_1^-(f^*\alpha) \\
		= & -\ev_1^-(f^*\alpha)\ev_1^+(\beta)\ev_2^-(\alpha)\ev_2^+(\beta)\cdots\ev_L^-(\alpha)\ev_L^+(\beta). \\
		\end{align*}
		\end{itemize}
		
		\textbf{Step 5: gathering the contributions.} We have found four possible choices of insertions, each being fully determined by the first insertion $\gamma_1^+$. For each of them we get
		$$\pm\int_{\prod [\M_j]}\prod_{j=1}^L (\ev_j^-(\gamma_j^-)\ev_j^+(\gamma_j^+)\prod\ev^*(\pt),$$
		where the second product contains the non-diagonal insertions: the point insertions for the non-flat vertices, and the insertions coming from vertices adjacent to the cycle. We can split this integral over each moduli space $[\M_j]$ parametrizing the curves encoded by the vertex $j$. We forget about the (unmarked) flat vertices since their contribution is cancelled by the product of edge weights due to their automorphisms. Consider a non-flat vertex $V$ of valency $n_V$. We have already computed
		$$\gen{\pt_0,\alpha_{\mu^V_1},\beta_{\mu^V_2},\pt_{\mu^V_3},\cdots,\pt_{\mu^V_{n_V}}}^{X/D^\pm}_{g,a_j\sfs+b_j\sfp} = -\mu^V_1\mu^V_2\cdot a_j^{n_j-1}\sigma_1(a_j),$$
		$$\gen{\pt_0,1_{\mu^V_1},\pt_{\mu^V_2},\pt_{\mu^V_3},\cdots,\pt_{\mu^V_{n_V}}}^{X/D^\pm}_{g,a_j\sfs+b_j\sfp} = (\mu^V_1)^2\cdot a_j^{n_j-1}\sigma_1(a_j).$$
		If $V$ is the $j$-th vertex in the cycle, $\mu^V_1=\pm w_j$ and $\mu^V_2=\pm w_{j+1}$, with signs depending on the orientation of the edges with respect to the orientation of $\RR/\ZZ$. In any case, up to sign, each of the four possible insertions yields
		$$\prod_1^L w_j^2 \prod_1^L a_j^{n_j-1}\sigma_1(a_j).$$
		Concerning the signs in case we insert a $1$-cycle, we have seen that is it is negative if and only if both constraint lie on the same component, \textit{i.e.} the points where the cycle goes in reverse direction when mapped to $\RR/\ZZ$. There are an even number of such point, so that the signs cancel each other. To finish, we only need to get rid of the monodromy. Assume that in the basis $(\alpha,\beta)$, the matrix of $f^*$ has the form $\left(\begin{smallmatrix} a & b \\ c & d \\ \end{smallmatrix}\right)$. Then, up to the global factor that we already took care of, only the first invariant depends on the intersection number between $f^*(-\beta)$ and $\alpha$ for the $\alpha$ insertion, and between $f^*\alpha$ and $\beta$ for the $\beta$ insertion. They are computed as follows:
		$$\begin{array}{rlrl}
		f^*(-\beta)\cdot\alpha = & (-b\alpha-d\beta)\cdot\alpha , & f^*\alpha\cdot \beta = & (a\alpha+c\beta)\cdot\beta  \\
		 = & d & 		 = & a.\\
		\end{array}$$
		As $a+d=\mathrm{tr}f^*$, taking into account the minus sign in front of the product of $1$-cycle insertions, we get
		$$2-\mathrm{tr}f^*.$$
		The $1+1=2$ is from the insertions of the point class and neutral class at $\gamma_1^+$, and the $\mathrm{tr}f^*=a+d$ from the insertions of $1$-cycle. Gathering this constant term, the vertex contribution and the product of edge weights, we get the result.
		\end{proof}

		\begin{rem}
		If the monodromy is trivial, as in the case of the product of elliptic curve $E\times F$, the term $2-\mathrm{tr}M^\nu=\tau(\nu)$ vanishes. we recover that the GW invariant of an abelian surface are $0$.
		If $M\neq I_2$ but $M^\nu=I_2$ for some $\nu$, we have some diagrams that give a $0$ contribution to the GW invariant.
		\end{rem}

	\begin{expl}
We explicit the formula on the diagrams from Figures \ref{fig-expl-pearl-diagram-1}, \ref{fig-expl-pearl-diagram-2} and \ref{fig-expl-pearl-diagram-3}.
	\begin{itemize}[label=$\ast$]
	\item On Figure \ref{fig-expl-pearl-diagram-1}. The product over edge weights give $2^3\cdot 1^3\cdot 1\cdot 1$. The vertex contribution is $a_1^2\sigma_1(a_1)\cdot a_3^2\sigma_1(a_3)$. Once we removed the flat-vertex, the only cycle is formed by the two edges linking the vertices $1$ and $3$, and the homotopy class it realizes is $2\in\pi_1(\RR/\ZZ)$. Thus, the multiplicity is
	$$\tau(2)\cdot a_1^2\sigma_1(a_1)\cdot a_3^2\sigma_1(a_3)\cdot 8.$$
	
	\item The pearl diagram on Figure \ref{fig-expl-pearl-diagram-2} is easier since there is no flat vertex. If the only edge would have weight $w$, we would get
	$$\tau(2)\cdot a_1\sigma_1(a_1)\cdot w^3.$$
	
	\item Last, for the diagram on Figure \ref{fig-expl-pearl-diagram-3}, the cycle obtained by removing the flat vertices is $2-4-5$, and realizes the homotopy class $0\in\pi_1(\RR/\ZZ)$, so that the diagram has in fact multiplicity $0$.
	\end{itemize}		
	\end{expl}

	\subsection{$\lambda$-classes and refinement}
	
	In this section, we explain how the multiplicities are modified to deal with the invariants with a $\lambda$-class insertion.

\subsubsection{Decomposition formula.} We now give the multiplicities that allow to compute the refined invariants. To do so, we use the decomposition formula for $\lambda$-class proven in \cite[Proposition 3.1]{bousseau2021floor}.

\begin{prop} \cite[Prop 3.1]{bousseau2021floor} \label{prop-splitting-lambda-class}
Let $\M$ be a finite Deligne-Mumford stack over $\CC$, $\Gamma$ a connected graph with first Betti number $g_\Gamma$. For each vertex $V$ let $\pi_V:\CCC_V\to\M$ be a family of genus $g_V$ prestable curves. For each flag $V\in e$, let $s_{V,e}:\M\to\CCC_V$ be a section avoiding nodes. Denote by $\pi:\CCC\to\M$ the curve obtained by gluing the sections corresponding to flags of the same edge, which is a family of prestable curves of genus $g=g_\Gamma+\sum g_V$. Then, for every $0\leqslant j\leqslant g$,
$$\lambda_{j,\pi}=\sum_{ \substack{0\leqslant j_V\leqslant g_V \\ \sum j_V=j} }\prod_V\lambda_{j_V,\pi_V}.$$
In particular, $\lambda_{g-j,\pi}=0$ if $j< g_\Gamma$.
\end{prop}

We now apply the decomposition formula as in Proposition \ref{prop-decomposition} to get the following.

\begin{prop}\label{prop-decomposition-lambda-case}
We have the following decomposition:
$$\gen{\lambda_{g-g_0};\pt^{g_0-1} }^S_{g,\varpi} = \sum_\PPP \sum_{\substack{0\leqslant j_V\leqslant g_V \\ \sum j_V=j} }
\frac{\prod_e w_e}{|\mathrm{Aut}(\PPP)|}\int_{\prod\vir{\M_V}} 
\prod_V\lambda_{j_V,\pi_V}
\prod_e\delta_e\prod_1^n\ev_i^*(\pt) ,$$
where the sum is over the decomposition diagrams with genus $g$ and degree $\varpi$.
\end{prop}

\begin{proof}
The proof is as the proof of Proposition \ref{prop-decomposition}, using the expression of the $\lambda$-class provided by \cite[Prop 3.1]{bousseau2021floor}.
\end{proof}

\subsubsection{From decomposition diagrams to discerning pearl diagrams.} Proposition \ref{prop-decomposition-lambda-case} assigns a multiplicity to each decomposition diagram, depending on the $\lambda$-class we insert. The latter expresses as a sum expanding the sum provided by the $\lambda$-class decomposition, and the diagonal classes using K\"unneth formula. We have Lemma \ref{lem-non-zero-lambda-inv} that restricts the possible insertions. We now need an equivalent of Lemma \ref{lem-genus-vertices-enum-case} to shrink the family of decomposition diagrams that have a non-zero multiplicity.

\begin{lem}\label{lem-genus-vertices-lambda-case}
Let $\PPP$ be a decomposition diagram with the insertion of a $\lambda$-class at each vertex, for which the multiplicity is non-zero. We then have the following:
\begin{enumerate}[label=(\roman*)]
\item the underlying graph $\PPP$ has genus $b_1(\PPP)\leqslant g_0$,
\item each flat vertex has $g_V=0$ and no $\lambda$-class insertion,
\item each non-flat vertex has genus $g_V\geqslant 1$ and an insertion $\lambda_{g_V-1}$.
\end{enumerate}
\end{lem}

\begin{proof}
Let $\PPP$ be a decomposition diagram with a $\lambda$-class $\lambda_{j_V}$ assigned to each vertex. We assume that the contribution is non-zero. We denote the genus of a vertex $V$ by $g_V$. First, as the $\lambda$-class insertion is $g-g_0$, by Proposition \ref{prop-decomposition-lambda-case}, we need to have $b_1(\PPP)\leqslant g_0$ for the contribution to be non-zero, proving (i). Moreover, after expanding the diagonal classes, the integral splits as a sum over the class insertions. We use the computations and Lemma \ref{lem-vanishing-top-lambda-class} to prove (ii) and (iii).
	\begin{itemize}[label=$\ast$]
	\item For a flat vertex, we need to have $j_V=g_V=0$ due to Proposition \ref{prop-computation-fibers} and the vanishing of invariants with a $\lambda$-class. Thus, we need to insert neutral class at the adjacent flags, and point classes at their other extremities.
	\item For a non-flat vertex, we have $j_V\leqslant g_V-1$ due to Lemma \ref{lem-vanishing-top-lambda-class}.
	\item A non-flat vertex needs to have a least one of the marked point, otherwise the invariant is always $0$ using the action of $\CC^*$ on the $\PP^1$ fibers of $X$. Furthermore, the dimension of $\vir{\M_V}$ at the given vertex is
	$$g_V+(n_V+1)-1.$$
	It needs to match the dimension of the insertions, which come from $\lambda_{j_V}$ of rank $j_V\leqslant g_V-1$, the point insertion of rank $2$, and the insertions $\gamma_{V,e}$ coming from the diagonal classes, whose ranks add up to at most $n_V-1$ thanks to \cite[Proposition 2.13]{blomme2021floor}. Thus, we get
	$$g_V+n_V=j_V+2+\sum_{e\ni V}\mathrm{rk}\gamma_{V,e} \leqslant j_V+2+n_V-1 \leqslant g_V+n_V.$$
	As we have equality, the rank of $\gamma_{V,e}$ add up to $n_V-1$ and the only possible value of $j_V$ for a non-flat vertex is actually $j_V=g_V-1$. %Moreover, Lemma \ref{lem-non-zero-lambda-inv}(i) yields that $j_V=g_V-1$, since it is the only insertion yielding a non-zero invariant for dimensional reasons.
	\end{itemize}
\end{proof}

Thus, the multiplicity associated to a decomposition graph and the $\lambda$-class $\lambda_{g-g_0}$ is
$$\frac{\prod_e w_e}{|\mathrm{Aut}(\PPP)|}\int_{\prod\vir{\M_V}} \prod_V\lambda_{g_V-1,\pi_V}\prod_e\delta_e\prod_1^n\ev_i^*(\pt).$$
The choice of genus at the non-flat vertices fully determines the decomposition of the $\lambda$-class. This is similar to the toric picture from \cite{bousseau2019tropical}, where the $\lambda$-class insertion used to fill the remaining constraints yields a top $\lambda$-class insertion at the vertex level. Here we also get a ``top" $\lambda$-class (equal to $\lambda_{g-1}$ due to the vanishing of $\lambda_g$). We gather the invariants by taking the generating series in $g$:
$$BG^S_{g_0,\varpi}(u) = \sum_{g\geqslant g_0} \gen{\lambda_{g-g_0};\pt^{g_0-1} }^S_{g,\varpi} u^{2g-2}.$$
Lemma \ref{lem-genus-vertices-lambda-case} proves that the genus at the non-flat vertices is at least $1$. Thus, we can restrict to the discerning pearl diagrams, provided we also sum over any possible genus addition to the non-flat vertices. We get the following multiplicity:
$$M(\PPP) = \frac{\prod_e w_e}{|\mathrm{Aut}(\PPP)|}
\sum_{g_V\geqslant 1} u^{2b_1(\PPP)+2\sum g_V-2}\int_{\prod\vir{\M_V}} \prod_V\lambda_{g_V-1,\pi_V}\prod_e\delta_e\prod_1^n\ev_i^*(\pt),$$ 
where the sum is over the assignments of genus at every non-flat vertex. Once expanded the diagonal classes and the integrals split, we show it is possible to factor this sum as a product over the vertices.

\subsubsection{Computation of the diagram multiplicity.} We now give a formula for the refined multiplicity of a discerning pearl diagram. As in the enumerative case, the multiplicity only depends on the underlying pearl diagram. Lemma \ref{lem-position-marked-flat-vertices-enum-case} still applies.

\begin{prop}\label{prop-computation-multiplicity-lambda-case}
		Let $\PPP$ be a pearl diagram. For each component $\PPP_k$ in the complement of marked flat vertices, let $\nu_k\in\ZZ\simeq\pi_1(\RR/\ZZ)$ be the class realized by its unique cycle. Through the change of variable $q=e^{iu}$, the multiplicity is given by
		$$M(\PPP)=\prod_k\tau(\nu_k)\prod_V \sfR_{a_V}(\mu_V) \prod_{E_\mathrm{flat}} w_e \prod_{E\backslash E_\mathrm{flat}} w_e^3,$$
		where the first product is over the the components $\PPP_k$, the second product over the non-flat vertices, third product over the edges adjacent to a flat vertex, and the last product over the edges not adjacent to a flat vertex.
\end{prop}

\begin{proof}
We proceed as in the proof of Proposition \ref{prop-computation-multiplicity}.

\textbf{Step 1: reducing down to the cycles.} The first step is verbatim to the enumerative case: we prune the branches adjacent to the cycles. We start by cutting the marked flat vertices. The unmarked flat vertices can be removed. When pruning a non-flat vertex, we get
$$\gen{\lambda_{g_V-1};\pt_0,1_{\mu_1},\pt_{\mu_2},\cdots,\pt_{\mu_n}}^{X/D^\pm}_{g_V,a_V\sfs+b_V\sfp}.$$
According to the computation from Proposition \ref{prop-computation-bdry-inv-pt-case}, we get $\mu_1^2$ times the coefficient of $\sfR_{a_V}(\mu_V)$ of degree $2g_V-2+n_V$.

\smallskip

\textbf{Step 2: the monodromy around the loop.} The latter has already been determined in the proof of Proposition \ref{prop-computation-multiplicity}.

\smallskip

\textbf{Step 3: taking care about each cycle.} As in the enumerative case, we need to replace each diagonal class by its K\"unneth decomposition, expand and split the integral over the vertices.

\smallskip

\textbf{Step 4: finding the insertions with non-zero contribution.} As in the enumerative case, the choice of the insertion $\gamma_1^+$ fully determines the remaining insertions. Thus, we get four terms depending on the value of $\gamma_1^+=1,\alpha,\beta$ or $\pt$.

\smallskip

\textbf{Step 5: gathering the contributions.} We proceed as in the enumerative case to deal with the remaining cycles: for each possible choice of insertions at the flag $\gamma_1^+$ at a fixed point over the cycle, the remaining insertions are fully determined, and we can split over the vertices.
	\begin{itemize}[label=$\ast$]
	\item If we insert the point class or the neutral class, after making the generating series over the possible genus of vertices, we get
	$$\prod_e w_e^2\prod_V\sfR_{a_V}(\mu_V),$$
	for each of them, where the first product is over the edges of the cycle, and the second product over its vertices.
	\item If we insert instead the class $\alpha$ or $\beta$, we can still make the generating series over the possible genus and use the skew-symmetry. The vertex contribution at a vertex $V$ is
	$$-\mu_1\mu_2\sfR_{a_V}(\mu_V).$$
	As in the enumerative case, if the vertex is the $j$-th vertex on the cycle, $-\mu_1\mu_2=\pm w_jw_{j+1}$ with a negative sign if and only if the flags lie on the same side of the vertex. As there are an even number of such vertices, they cancel each other, and we also get
	$$-\mathrm{tr}(M^\nu)\prod_e w_e^2\prod_V \sfR_{a_V}(\mu_V).$$
	\end{itemize}
\end{proof}

\section{Regularity of the generating series}

In this section, we prove the quasi-modularity of the generating series of Gromov-Witten invariants for bielliptic surfaces. Let $S$ be a bielliptic surface of type $(a)$, $(b)$, $(c)$ or $(d)$, and $n=2,3,4$ or $6$ accordingly. We consider the following double generating series:
$$F_g^S(\sfp,\sfq)=\sum_{a,b}\gen{\pt^{g-1}}^S_{g,a\sfe+b\sff}\sfp^a \sfq^b.$$

\begin{theo}\label{theo-quasi-modularity-GW-invariants}
There exists a finite collection of graphs $\Gamma$ with each assigned: a quasi-modular form $A_\Gamma(\sfp)$ for $SL_2(\ZZ)$, a quasi-modular-form $B_\Gamma(\sfq)$ for $\Gamma_1(n)$, such that
$$F_g^S(\sfp,\sfq)=\sum_\Gamma A_\Gamma(\sfp) B_\Gamma(\sfq).$$
In particular, the generating series is quasi-modular in $\sfp$ for $SL_2(\ZZ)$, and in $\sfq$ for $\Gamma_1(n)$ (See below for definition of the latter).
\end{theo}

The graphs over which we sum are actually the pearl diagrams if we forget about the orientation, the weight and heights of the edges along with the homology classes assigned to the vertices. Generating series are recovered by summing over the possible assignments.

\smallskip

The section is organized as follows. In Section \ref{sec-jacobi-forms}, we introduce (quasi-)Jacobi forms for congruence subgroups of $SL_2(\ZZ)$ and prove quasi-modularity results concerning their Fourier coefficients. These results are probably already known to experts on modular forms. We still include them since finding references on congruence subgroups in the litterature is a little bit hard. The main result of Section \ref{sec-generating-series-graphs-sums} is Theorem \ref{theo-quasi-modularity-graph-sums}, that generalizes \cite[Theorem 6.1]{goujard2019counting}, dealing with generating series of graphs, by allowing us to impose congruence conditions on the height of edges. We then use Theorem \ref{theo-quasi-modularity-graph-sums} to prove Theorem \ref{theo-quasi-modularity-GW-invariants} and give some explicit computations.

	\subsection{Quasi-modular and quasi-Jacobi forms}
	\label{sec-jacobi-forms}
	
	We give in this section definitions about quasi-Jacobi forms along with some properties allowing us to construct quasi-modular forms from the latter. This builds on the work from E. Goujard and M. M\"oller \cite{goujard2019counting} and the more direct approach provided by G. Oberdieck and A. Pixton \cite{oberdieck2018holomorphic}. The difference is that here we are studying modularity properties with respect to some congruence subgroups of $SL_2(\ZZ)$. Though we follow the same steps, this difference explains the technicalities we meet along the way.
	
	\subsubsection{Modular group and congruence subgroups.} We consider the following subgroup of $SL_2(\ZZ)$, called congruence subgroup:
	$$\Gamma_1(n)=\left\{ g\in SL_2(\ZZ) \text{ s.t. }g\equiv\begin{pmatrix}
	1 & \ast \\ 0 & 1 \\
	\end{pmatrix} \modulo n \right\}.$$
	
	\begin{rem}
	We also have the congruence subgroup
	$$\Gamma_0(n)=\left\{ g\in SL_2(\ZZ) \text{ s.t. }g\equiv\begin{pmatrix}
	\ast & \ast \\ 0 & \ast \\
	\end{pmatrix} \modulo n \right\}.$$
	The difference is that diagonal elements can realize different invertible classes modulo $n$. For the $n$ of interest to us, \textit{i.e.} $n=2,3,4,6$, the only other invertible element is $-1$. Thus $\Gamma_0(n)$ is generated by $\Gamma_1(n)$ and $-I_2$, so that it does not matter for quasi-modularity properties. This is why we only care about $\Gamma_1(n)$. Generalizations of most statements to $\Gamma_0(n)$ are more technical.
	\end{rem}
	
	The group $SL_2(\ZZ)$, also called modular group, acts on $\ZZ^2$. It is thus possible to consider the semi-direct product $\ZZ^2\rtimes SL_2(\ZZ)$. If $\Lambda\subset\ZZ^2$ is a lattice preserved by a subgroup $\Gamma$ of $SL_2(\ZZ)$, then $\Lambda\rtimes\Gamma$ is a subgroup of $\ZZ^2\rtimes SL_2(\ZZ)$.
	
	\begin{expl}
	The lattice $\ZZ\oplus n\ZZ$ is preserved by $\Gamma_1(n)$, so that we can consider
	$$\overline{\Gamma}_1(n)=(\ZZ\oplus n\ZZ)\rtimes\Gamma_1(n).$$
	\end{expl}
	
	\begin{rem}
	A quasi-modular form of weight $w$ and depth $p$ for $\Gamma\subset SL_2(\ZZ)$ is a holomorphic function $f:\HH\to\CC$ such that there exists holomorphic functions $f_0=f,f_1,\dots,f_p$ and for any $\abcd\in\Gamma$:
	$$(c\tau+d)^{-w}f\left( \tabcd \right)=\sum_0^p f_i(\tau)\left(\frac{c}{c\tau+d}\right)^i.$$
	We denote by $\mathrm{QMod}_w(\Gamma)$ the set of quasi-modular forms of weight $w$ for $\Gamma\subset SL_2(\ZZ)$. We have a graded ring $\mathrm{QMod}(\Gamma)=\bigoplus_w \mathrm{QMod}_w(\Gamma)$. As these functions are $1$-periodic, they can be written in the variable $\sfq=e^{2i\pi\tau}$. The graded ring $\mathrm{QMod}(SL_2(\ZZ))$ is known to be generated by the Eisenstein series $G_2,G_4$ and $G_6$.
	\end{rem}
	
	\subsubsection{One and multi-variable quasi-Jacobi forms.} Before going to the multivariable case, we start defining a class of functions on $\CC\times\HH$ which satisfy some transformation property under the action of the groups defined above.
	
	\begin{defi}
	We say that $f:\CC\times\HH\to\CC$ is a quasi-Jacobi form of weight $w$ and depth $p$ for $\Lambda\rtimes\Gamma\subset\ZZ^2\rtimes SL_2(\ZZ)$ if it satisfies the following:
	\begin{enumerate}[label=(\roman*)]
	\item it is meromorphic,
	\item it is elliptic in $z$: for every $(\lambda,\mu)\in\Lambda\subset\ZZ^2$, we have
	$f(z+\lambda+\mu\tau;\tau)=f(z;\tau)$,
	\item it is quasi-modular in $\tau$: there exists meromorphic functions $f_i:\CC\times\HH\to\CC$ such that for any $\abcd\in\Gamma\subset SL_2(\ZZ)$, we have
	$$(c\tau+d)^{-w}f\left(\frac{z}{c\tau+d};\tabcd\right)=\sum_{i=0}^p f_i(z;\tau)\left(\frac{c}{c\tau+d}\right)^i.$$
	\end{enumerate}
	\end{defi}
	
	The definition is extended to the case of functions $f:\CC^N\times\HH\to\CC$ by taking instead of $z$ a tuple of complex numbers and requiring the periodicity in each of the coordinates.
	
	\begin{rem}
	Sometime, we may omit the argument $\tau$ when writing the functions.
	\end{rem}
	
	\begin{expl}\label{expl-quasi-jacobi-forms}
	\begin{itemize}[label=$\ast$]
	\item The Weierstrass function $\wp$ satisfies the equation
	$$(c\tau+d)^{-2}\wp\left(\frac{z}{c\tau+d},\tabcd\right)=\wp(z;\tau).$$
	In that case, since the depth $p$ is $1$, we say that it is a Jacobi form for $\ZZ^2\rtimes SL_2(\ZZ)$ and not only quasi-Jacobi form. It has a unique pole of order $2$ at $z\equiv 0\modulo 1,\tau$, where we have the following Laurent development:
	$$\frac{1}{(2i\pi)^2}\wp(z;\tau)=\frac{1}{(2i\pi z)^2}+\sum_{2}^\infty (2l-1)2lG_{2l}(\tau)(2i\pi z)^{2l-2},$$
	which may be taken as a definition of the Eisenstein series $G_{2l}$.
	\item By differentiating with respect to the $z$ variable, we get
	$$(c\tau+d)^{-s}\wp^{(s-2)}\left(\frac{z}{c\tau+d},\tabcd\right)=\wp^{(s-2)}(z;\tau),$$
	which is thus a Jacobi form of weight $s$. More generally, the derivative of a quasi-Jacobi form of weight $w$ is a quasi-Jacobi form of weight $w+1$.
	\item When computing generating series of graphs in Section \ref{sec-generating-series-graphs-sums}, we need to consider shifts of the Weierstrass function: for fixed $n$ and a given $k\in\ZZ/n\ZZ$, we set
	$$\wp_k(z;\tau)=\wp(z-k\tau;n\tau).$$
	It is only $n\tau$-periodic and not $\tau$-periodic. Moreover, it satisfies the following transformation law: if $\abcd\in\Gamma_0(n)$, using the modularity for $\wp$ with $\left(\begin{smallmatrix} a & nb \\ c/n & d \\ \end{smallmatrix}\right)$, we have
	$$(c\tau+d)^{-2}\wp_k\left(\frac{z}{c\tau+d};\tabcd\right)=\wp_{ak}(z;\tau).$$
	In particular, if $a\equiv 1\modulo n$, we get that $\wp_k$ is a Jacobi form for $\overline{\Gamma}_1(n)$, but not a priori for $\overline{\Gamma}_0(n)$. This is the main source of technicalities for the $\Gamma_0(n)$ case.
	\end{itemize}
	\end{expl}
	
	In particular, for each fixed value of $\tau$, a quasi-Jacobi form induces a meromorphic function on an elliptic curve $E_\tau$, which is $\CC/\gen{1,\tau}$ in the case of $\ZZ^2\rtimes SL_2(\ZZ)$, and $\CC/\gen{1,n\tau}$ in the case of $\overline{\Gamma}_1(n)$.
	
	\subsubsection{Log-derivative of the Jacobi $\theta$-function.} We will need the following auxiliary function. We refer to \cite[Lemma 18]{oberdieck2018holomorphic} for a precise definition. We denote it by $\sfA$. It satisfies
	$$\sfA(z+1;\tau)=\sfA(z;\tau),$$
	$$\sfA(z+\tau;\tau)=\sfA(z;\tau)-1.$$
It is not a quasi-Jacobi form since it is only $1$-periodic, and has no periodicity in the $\tau$-direction. The function has poles at $0$ modulo $1,\tau$. We have the following Laurent development at $0$:
	$$\sfA(z;\tau)=\frac{1}{2i\pi z}-\sum_{l=1}^\infty 2lG_{2l}(\tau)(2i\pi z)^{2l-1}.$$
	It is possible to show that it satisfies a transformation law under $SL_2(\ZZ)$:
	$$(c\tau+d)^{-1}\sfA\left(\frac{z}{c\tau+d};\tabcd\right)=\sfA(z;\tau)-z\frac{c}{c\tau+d}.$$
	Moreover, $\frac{1}{2i\pi}\frac{\partial}{\partial z}\sfA(z;\tau)=-\wp(z;\tau)-2 G_2(\tau)$.
	
	\medskip
	
	\subsubsection{Some rings.} We now consider some rings of quasi-Jacobi forms in one variable:
	\begin{itemize}[label=$\ast$]
	\item Let $\J$ be the ring of Jacobi forms for $\ZZ^2\rtimes SL_2(\ZZ)$ with at most poles at $0\modulo 1,\tau$. It is in fact generated by the quasi-modular forms for $SL_2(\ZZ)$ (\textit{i.e.} by the Eisenstein series $G_2,G_4,G_6$) and the derivatives of the Weierstrass function $\wp$. This fact is a special case of the result proven in \cite[Section 5]{goujard2019counting}.
	\item Let $\J(\Gamma_1(n))$ be the ring of quasi-Jacobi forms for $\overline{\Gamma}_1(n)$ having at most poles at $0\modulo 1,\tau$, \textit{i.e.} the $k\tau$ modulo $1,n\tau$ and $k\in\ZZ/n\ZZ$.
	\end{itemize}
	
	These rings are graded by the weight of the forms, and we denote by $\J^{(w)}$ the graded pieces.
	
	\begin{expl}
	All the quasi-Jacobi forms considered in example \ref{expl-quasi-jacobi-forms} belong to these rings.
	\end{expl}
	
Now we give some rings of several variables Jacobi forms.
	\begin{itemize}[label=$\ast$]
	\item Let $\JJJ_N$ be the ring of $N$ variables Jacobi forms for $\ZZ^2\rtimes SL_2(\ZZ)$ generated by $G_2$, $G_4$ and $G_6$ along with $\wp^{(s)}(z_i-z_j;\tau)$. They have poles at most on the diagonals $z_i\equiv z_j\modulo 1,\tau$.
	\item Let $\JJJ_N(\Gamma_1(n))$ be the ring of $N$ variables quasi-Jacobi forms generated by quasi-modular forms for $\Gamma_1(n)$ and the shifts $\wp^{(s)}(z_i-z_j-k\tau;n\tau)$. They have at most poles along the diagonals $z_i- z_j \equiv k\tau\modulo 1,n\tau$.
	\end{itemize}
	
	\subsubsection{Quasi-modular forms out of Quasi-Jacobi forms.} We now follow the steps from \cite[Appendix]{oberdieck2018holomorphic} and prove a quasi-modularity results for periods and coefficients of quasi-Jacobi forms. We start with the coefficients in the Laurent development of a quasi-Jacobi form, and the one variable case before getting to the multivariable setting. In each case, we state the result for $SL_2(\ZZ)$, already known from \cite{oberdieck2018holomorphic} and \cite{goujard2019counting}, and its analog for the congruence subgroups $\Gamma_1(n)$.
	
	\begin{lem}\label{lem-quasi-modularity-Laurent-coefficients}
	We have the following:
	\begin{enumerate}[label=(\roman*)]
	\item If $f\in\J^{(w)}$ and $f(z;\tau)=\sum_r c_r(\tau)z^r$, then $c_r(\tau)\in \mathrm{QMod}_{w+r}(SL_2(\ZZ))$.
	\item If $f\in\J^{(w)}(\Gamma_1(n))$ and $f(k\tau+\varepsilon;\tau)=\sum_r c_{k,r}(\tau)\varepsilon^r$, then $c_{k,r}(\tau)\in\mathrm{QMod}_{w+r}(\Gamma_1(n))$.
	\end{enumerate}
	\end{lem}
	
	\begin{proof}
	The quasi-modularity comes from the quasi-modularity for $f$, and identification of coefficients in the Laurent development. Let $f$ be a quasi-Jacobi form, and let $\abcd\in SL_2(\ZZ)$ for which it is quasi-modular. To avoid burdening notations, assume $f$ is in fact a Jacobi form. Let $k\tau$ be a pole of $f$. We write
	$$f(k\tau+\varepsilon;\tau)=\sum_{r=-M}^\infty c_{k,r}(\tau)\varepsilon^r.$$
	Thus, we get
	\begin{align*}
	(c\tau+d)^{-w}f\left(k\tabcd +\frac{\varepsilon}{c\tau+d};\tabcd\right) = & (c\tau+d)^{-w}\sum_{-M}^\infty c_{k,r}\left(\tabcd\right)\left(\frac{\varepsilon}{c\tau+d}\right)^r \\
	= & f(ak\tau+kb+\varepsilon;\tau) \\
	= & \sum_{-M}^\infty c_{ak,r}(\tau)\varepsilon^r. \\
	\end{align*}
	Identifying coefficients, we get
	$$(c\tau+d)^{-w-r}c_{k,r}\left(\tabcd\right)=c_{ak,r}(\tau).$$
	\begin{enumerate}[label=(\roman*)]
	\item For the first case we take $k=0$ and $\abcd\in SL_2(\ZZ)$ to conclude.
	\item For $\abcd\in\Gamma_1(n)$, we have $a\equiv 1\modulo n$, so that we also get the quasi-modularity.
	\end{enumerate}
	\end{proof}
	
	It is possible to relate periods of quasi-Jacobi forms to coefficients in the Laurent development, thus showing that they are also quasi-modular forms thanks to Lemma \ref{lem-quasi-modularity-Laurent-coefficients}. For a quasi-Jacobi form $f\in\J$, we define its constant term by
	$$[f]_0=\int_{\sfC_\alpha}f(z;\tau)\dd z,$$
	where $\sfC_\alpha$ is the path $t\in [0;1]\mapsto t+i\alpha$, and $0<\alpha<\Im \tau$. Due to the periodicity condition and the fact that $f(z;\tau)\dd z$ is closed, the result does not depend on the choice of $\alpha$. It is actually a period of the meromorphic form $f(z;\tau)\dd z$ on the elliptic curve $E_\tau$. For quasi-Jacobi forms for $\overline{\Gamma}_i(n)$, the result depends on $a\modulo n\Im\tau$ since we have poles not only at $0$. Thus, we define
	$$[f]_k=\int_{\sfC_\alpha}f(z;\tau)\dd z,$$
	with $k\Im\tau< \alpha <(k+1)\Im\tau$.

	\begin{lem}\label{lem-quasi-modularity-period}
	We have the following:
	\begin{enumerate}[label=(\roman*)]
	\item If $f\in\J^{(w)}$, then $[f]_0\in\mathrm{QMod}_w(SL_2(\ZZ))$.
	\item If $f\in\J^{(w)}(\Gamma_1(n))$,  then for any $k$, $[f]_k\in\mathrm{QMod}_w(\Gamma_1(n))$.
	\end{enumerate}
	\end{lem}
	
	\begin{proof}
	We use the function $\sfA$ which is $1$-periodic but satisfies $\sfA(z+\tau;\tau)=\sfA(z;\tau)-1$. Thus, given a quasi-Jacobi form $f$, the product $F(z;\tau)=f(z;\tau)\sfA(z;\tau)$ is $1$-periodic and satisfies
	$$F(z+n\tau;\tau)=F(z;\tau)-nf(z;\tau).$$	
	The integral of $F$ over $\sfC_\alpha$ and $\sfC_{\alpha+n\tau}$ differ precisely by $\int_{\sfC_{\alpha+n\tau}}nf(z;\tau)\dd z$:
	$$\int_{\sfC_\alpha}F\dd z-\int_{\sfC_{\alpha+n\tau}}F\dd z = n\int_{\sfC_\alpha}f\dd z.$$ Furthermore, as $\sfC_\alpha-\sfC_{\alpha+n\tau}$ is the boundary of $\alpha\leqslant\Im z\leqslant \alpha+n\Im\tau$ in the cylinder $\CC/\gen{1}$. Hence, by the residue theorem, the integral is equal to the sum of the residues of $F$ at its poles, which are precisely the $k\tau$ with imaginary part between $\alpha$ and $\alpha+n\Im\tau$. The developments of $f$ and $\sfA$ at $k\tau$ are
	$$f(k\tau+z;\tau)=\sum_r c_{k,r}(\tau)z^r \text{ and }\sfA(k\tau+z;\tau)=\frac{1}{2i\pi z}-k-\sum_{l=1}^\infty 2lG_{2l}(\tau)(2i\pi z)^{2l-1}.$$
	Thus, $2i\pi$ times the residue at $k\tau$ of their product is
	$$c_{k,0}(\tau)-(2i\pi)kc_{k,-1}(\tau)-\sum 2lG_{2l}(\tau)(2i\pi)^{2l}c_{k,-2l}(\tau).$$
	We then make the sum over $\alpha<k\Im\tau<\alpha+n\Im\tau$. As coefficients of $f$ at each pole are quasi-modular by Lemma \ref{lem-quasi-modularity-Laurent-coefficients}, we get the expected quasi-modularity.
	\end{proof}
	
	\begin{expl}
	The proposition allows for an explicit computation of the period provided we know the Laurent development of $f$ at its poles. For instance, for $f=\wp$, we get that its period is a modular form for $SL_2(\ZZ)$. For the $\wp_k$, we conclude that its period is modular for $\Gamma_1(n)$.
	\end{expl}

	\subsubsection{Quasi-modular forms out of multivariable quasi-Jacobi forms.} We now define periods for a multivariable quasi-Jacobi form. Recall that we view these as functions on the product of elliptic curves $E_\tau^N$, having poles at most along the diagonals $z_j-z_i\equiv k\tau\modulo 1,n\tau$ (with $n=1$ in the case of $SL_2(\ZZ)$). Taking the imaginary part of the coordinates, we have a projection $\pi:E_\tau^N\to R_\tau^N$, where $R_\tau=\RR/\gen{n\Im\tau}$. We denote by $\Delta_\tau$ the set of diagonals in $R_\tau^N$. We integrate quasi-Jacobi forms along fibers of this projection, which are real tori of dimension $N$, obtained by varying the real part of coordinates. This generalizes the $1$-dimensional case, since the paths $\sfC_\alpha$ are the fibers of the projection $E_\tau\to R_\tau$. The connected components of $R_\tau^N-\Delta_\tau$ are indexed by a discrete data $\Sigma$, and we label the corresponding component $U_\Sigma$. For $\alpha\in R_\tau^N-\Delta_\tau$, belonging to some connected component $U_\Sigma$, we denote by $\sfT_\alpha$ the torus fiber over $\alpha$.
	
	\smallskip
	
	Let $f\in\JJJ_N$ be a quasi-Jacobi form in $N$ variables. We set
	$$[f]_\Sigma=\int_{\sfT_\alpha}f(z_1,\cdots,z_N;\tau)\dd z_1\cdots\dd z_N,$$
	where $\alpha$ is any element of $U_\Sigma$. It is a function in $\tau$. It does not depend on $\alpha\in U_\Sigma$ since the fibers are cobordant and the holomorphic form is closed. The integral can be computed by successively integrating over each variable $z_i$.
	
	%\begin{figure}
	%{\color{red} dessin connected components of $U_\Sigma$.}
	%\end{figure}
	
	\begin{rem}
	In the case of $SL_2(\ZZ)$, a connected component of $R_\tau^N$ is fully determined by the cyclic order on its coordinates. For $\Gamma_1(n)$, we also need to recall the position of these coordinates with respect to the diagonals $x_j-x_i\equiv k\Im\tau\modulo n\Im\tau$ for $k\neq 0$. So they depend on a different discrete data $\Sigma$: the cyclic order of the remainder of the imaginary part modulo $\Im\tau$, and the $k_{ij}$ for which $k_{ij}\Im\tau<x_j-x_i<(k_{ij}+1)\Im\tau$.
	\end{rem}
	
	The goal is to show that these functions $[f]_\Sigma(\tau)$ are quasi-modular forms for $\Gamma_1(n)$.

	\begin{theo}\label{theo-quasi-modularity-periods-multivariable}
	We have the following:
	\begin{enumerate}[label=(\roman*)]
	\item If $f\in\JJJ_N^{(w)}$, then $[f]_\Sigma(\tau)\in\mathrm{QMod}_{\leqslant w}(SL_2(\ZZ))$.
	\item If $f\in\JJJ_N^{(w)}(\Gamma_1(n))$, then $[f]_\Sigma(\tau)\in\mathrm{QMod}_{\leqslant w}(\Gamma_1(n))$.
	\end{enumerate}
	\end{theo}
	
	The first statement is proven in \cite{oberdieck2018holomorphic}. We adopt a similar approach to prove the second statement, proceeding by induction on the number of variables. The proof goes through two main steps:
	\begin{itemize}[label=$\bullet$]
	\item proving that the residue along one of the diagonals after multiplication by some power of $\sfA_{1N}=\sfA(z_1-z_N)$ is still a quasi-Jacobi form,
	\item using the integration trick from the one variable case to compute at least one integral and reduce to the first case.
	\end{itemize}
	
	The proof requires the use of residues along the diagonals. In \cite{oberdieck2018holomorphic}, they are denoted as right operators $R_{ab}$ taking residue along $z_a\equiv z_b\modulo 1,\tau$. Notice that since we consider products with $\sfA$, we lose the $\tau$-periodicity. Here, as we have several diagonals, we denote by $R_{ab}^k$ the residue operator taking the residue over the diagonal $z_a-z_b\equiv k\tau\modulo 1$. We denote the composition of residue operators as follows:
	$$\R_{i_1\cdots i_l}^{k_2\cdots k_l}=R_{i_1i_2}^{k_2}\cdots R_{i_{l-1}i_l}^{k_l}.$$
	
	\begin{lem}\label{lem-quasi-jacobi-product-by-A}
	For $f\in\JJJ_N^{(w)}(\Gamma_1(n))$ and any $r\geqslant 0$, then
	$$(f\sfA_{i_1i_m}^r)\R_{i_1\cdots i_m}^{k_2\cdots k_m}\in\JJJ_{N-m+1}^{(w+r-m+1)}(\Gamma_1(n)).$$
	\end{lem}
	
	\begin{proof}
	The is verbatim to \cite[Lemma 20]{oberdieck2018holomorphic}. We prove it by induction assuming that $f$ is a monomial in the functions $\wp_k$:
	\begin{itemize}[label=$\ast$]
	\item Assume we have only one residue: $(f\sfA_{ab})R^k_{ab}$. We write $f=\widetilde{f}f_{kab}$, where $f_{ab}$ regroups all the $\wp_k^{(s)}(z_a-z_b;\tau)$, having a pole on the considered diagonal. As the Laurent coefficients of $\sfA_{ab}$ and $f_{kab}$ along the diagonal are quasi-modular forms, a direct computation shows that the residues of $\widetilde{f}(f_{kab}\sfA_{ab})$ are as well.
	\item To complete the induction, assume we take one residue of the form $(fA_{ac})R_{ab}^k$ with $b\neq c$. We still write $f=\widetilde{f}f_{kab}$. Now, $\sfA_{ac}$ has no pole along the diagonal $z_a-z_b\equiv k\tau\modulo 1,\tau$. Thus, we get
	$$(f\sfA_{ac}^r)R_{ab}^k = \sum \mathrm{Coeff}_{(z_a-z_b-k\tau)^{-l}}(f_{kab})\frac{1}{(l-1)!}\frac{\partial^{l-1}}{\partial z_a^{l-1}}(\widetilde{f}\sfA_{ac}^r)|_{z_a=z_b+k\tau}.$$
	As the derivative of $\sfA$ expresses in terms of $\wp$, the right-hand side expresses as a sum of terms of the form $g\sfA_{bc}^s$, and we conclude by induction.
	\end{itemize}
	\end{proof}
	
	\begin{lem}\label{lem-integration-one-N+1-to-N}
	Let $f\in\JJJ_N(\Gamma_1(n))$ be a quasi-Jacobi form. There exists polynomials $Q_{i_1\cdots i_l}^{k_2\cdots k_l}$ such that we have
	$$[f]_{\Sigma}=\sum_{l\geqslant 1}\sum_{\substack{i_1,\cdots,i_l \\k_2,\cdots,k_l\in\ZZ/n\ZZ}} \left[ fQ_{i_1\cdots i_l}^{k_2\cdots k_l}(\sfA_{1N})\R_{i_1\cdots i_l}^{k_2\cdots k_l} \right]_{\mathrm{ft}(\Sigma)} ,$$
	where $(i_1,\cdots,i_l)$ is an sequence of distinct elements of $[\![1;N]\!]$ with $i_1=1$, $i_l=N$, and $\mathrm{ft}(\Sigma)$ is the component of $R_\tau^{N-l+1}-\Delta$ obtained from $\Sigma$ by merging the coordinates $z_{i_1},\cdots,z_{i_l}$ through the residue operator.
	\end{lem}
	
	\begin{rem}
	In the case of $SL_2(\ZZ)$ treated in \cite{oberdieck2018holomorphic}, the choice of $\Sigma$ is equivalent to the choice of a permutation $\sigma$, there is no choice of $k_i$, and the polynomials $Q_{i_1\cdots i_l}$ are explicitly chosen to be
	$$Q(X)=\bino{X+l-2-g_{i_1i_2}-\cdots -g_{i_li_{l+1}}}{l-1},$$
	with $g_{ab}=\mathds{1}_{\sigma(a)>\sigma(b)}$. For the case of $\Gamma_1(n)$, the notations are more heavy, but the spirit is the same.
	\end{rem}
	
	\begin{proof}
	We proceed by induction to prove that for any $L$,
	$$[f]_\Sigma=\sum_{l=1}^L\sum_{\substack{i_1,\cdots,i_l \\0\leqslant k_2,\cdots,k_l< n}} \left[ fQ_{i_1\cdots i_l}^{k_2\cdots k_l}(\sfA_{1N})\R_{i_1\cdots i_l}^{k_2\cdots k_l} \right]_{0,\Sigma} ,$$
	with $i_k$ distinct elements, $i_1=1$, $i_l=N$ when $l<L$ (and no condition on $i_L$). It then suffices to take $L=N+1$. Let $\alpha\in U_\Sigma$, so that we compute the integral of $f$ on $\sfT_\alpha$, with $\alpha$ chosen such that $0<\alpha_i<n\Im\tau$.
	\begin{itemize}[label=$\circ$]
	\item The equality is true for $L=1$ since it just amounts to $[f]_\Sigma=[f]_\Sigma$.
	
	\item To get the spirit of the induction, we continue with the case $L=2$. To do so, we consider $f$ as a function of $z_1$ depending on the remaining variables. We compute the integral over $z_1$ using the function
	$$f\sfA_{1N}=f(z_1,\dots,z_N;\tau)\sfA(z_1-z_N;\tau),$$
	which satisfies
	$$(f\sfA_{1N})(z_1+n\tau,z_2,\dots,z_N;\tau)=(f\sfA_{1N})(z_1,\dots,z_N;\tau)-nf(z_1,\dots,z_N;\tau).$$
	Moreover, it has poles exactly at the $z_1\equiv z_j+k\tau\modulo 1,n\tau$ for every other coordinate $z_j$ and possible choice of $k\in\ZZ/n\ZZ$. We integrate over $\sfC_{\alpha_1}-\sfC_{\alpha_1+n\tau}$:
	$$\int_{\sfC_{\alpha_1}}f\sfA_{1N}\dd z_1 - \int_{\sfC_{\alpha_1+n\tau}}f\sfA_{1N}\dd z_1 = n\int_{\sfC_{\alpha_1}} f\dd z_1.$$
	Meanwhile, as $\sfC_{\alpha_1}-\sfC_{\alpha_1+n\tau}$ is the boundary of $\{ \alpha_1\leqslant\Im z_1\leqslant\alpha_1+n\Im\tau\}$, the residue theorem yields $2i\pi$ times the sum over the residues:
	$$\int_{\sfC_{\alpha_1}} f\dd z_1=\frac{1}{n}\sum_{\substack{i_2\neq 1 \\ \widetilde{k_2}}} (f\sfA_{1N})R_{1i_2}^{\widetilde{k_2}},$$
	where $\widetilde{k_2}$ is subject to the condition
	$$\alpha_1<\alpha_{i_2}+\widetilde{k_2}\Im\tau<\alpha_1+n\Im\tau.$$
	This latter condition is for the pole to belong to the domain with boundary $\sfC_{\alpha_1}-\sfC_{\alpha_1+n\tau}$. It is possible to reduce $\widetilde{k_2}$ modulo $n$ up to translating the argument. If we write $\widetilde{k_2}=n\cdot q(\widetilde{k_2})+k_2$, with $0\leqslant k_2<n$, then
	$$(f\sfA_{1N})R_{1i_2}^{\widetilde{k_2}} = f(\sfA_{1N}-nq(\widetilde{k_2}))R_{1i_2}^{k_2}.$$
	 Each term under the sum is now a function in the variables $(z_2,\cdots,z_N)$, since $z_1$ has been set equal to $z_{i_2}+\widetilde{k}_2\tau$ when taking the residue. Making the integral over the remaining variables yields the desired equality for $L=2$.
	 
	 %If $i_2=N$, Lemma \ref{lem-quasi-jacobi-product-by-A} ensures that we get a quasi-Jacobi form, so that the induction on the number of variables yields the quasi-modularity. If $i_2\neq N$, it does not apply, and we continue integrating. 
	
	\item We now assume the equality is true for some $L$. Consider a term for which $i_L\neq N$:
	$$fQ_{i_1\cdots i_L}^{k_2\cdots k_L}(\sfA_{1N})\R_{i_1\cdots i_L}^{k_2\cdots k_L}.$$
	It depends on the $N-L+1$ variables $z_1(=z_{i_2}=\cdots =z_{i_L})$ and $z_j$ for $j\neq 1,i_2,\cdots,i_L$. We now perform the integral over $z_{i_L}$. To use the same integration trick, we look for a function $\sfH$ such that
	$$\sfH(z_{i_L}+n\tau)=\sfH-fQ_{i_1\cdots i_L}^{k_2\cdots k_L}(\sfA_{1N})\R_{i_1\cdots i_L}^{k_2\cdots k_L}.$$
	Let $P$ be a polynomial such that
	$$P(X-n)-P(X)=-Q_{i_1\cdots i_L}^{k_2\cdots k_L},$$
	and set $\sfH=fP(\sfA_{1N})\R_{i_1\cdots i_L}^{k_2\cdots k_L}$. We thus have
	\begin{align*}
	 & fP(\sfA_{1N})(z_1+n\tau)-fP(\sfA_{1N}) \\
	= & fP(\sfA_{1N}-n)-fP(\sfA_{1N}) \\
	= & -fQ_{i_1\cdots i_L}^{k_2\cdots k_L}(\sfA_{1N}).\\
	\end{align*}
	As the residue commute to translation operators, successive integration show that $\sfH$ satisfies the desired relation. The choice of $P$ can be made explicitly with binomial coefficients in the $SL_2(\ZZ)$ case from \cite{oberdieck2018holomorphic}. We then apply residue theorem by integrating over $\sfC_{\alpha_{i_L}} - \sfC_{\alpha_{i_L}+n\tau}$. The function $\sfH$ has divisors along $z_{i_L}\equiv z_{i_{L+1}}+k_{L+1}\tau\modulo 1,n\tau$ for any possible choice of $i_{L+1}$ distinct from $i_1,\cdots,i_L$, and $k_{L+1}$. The residue theorem ensures that
	$$ \int_{\sfC_{\alpha_{i_L}}} fQ_{i_1\cdots i_L}^{k_2\cdots k_L}(\sfA_{1N})\dd z_{i_L} = \sum_{\substack{i_{L+1}\neq i_1,\cdots,i_L \\ \widetilde{k}_{L+1} }} fP(\sfA_{1N})\R_{i_1\cdots i_L}^{k_2\cdots k_L}R_{i_Li_{L+1}}^{\widetilde{k}_{L+1}}.$$
	The integer $\widetilde{k}_{L+1}$ has to satisfy
	$$\alpha_{i_L}<\alpha_{i_{L+1}}+\widetilde{k}_{L+1}\Im\tau < \alpha_{i_L}+n\Im\tau.$$
	As in the $L=2$ case, it is possible to translate each $\widetilde{k}_{L+1}$ to take it between $0$ and $n$, by writing $\widetilde{k}_{L+1}=n\cdot q(\widetilde{k}_{L+1})+k_{L+1}$, $0\leqslant k_{L+1}<n$, and setting
	$$Q_{i_1\cdots i_{L+1}}^{k_2\cdots k_{L+1}}=P(X-nq(\widetilde{k}_{L+1})).$$
	Integrating with the remaining variables allows us to conclude the induction.
	\end{itemize}
	\end{proof}
	
	We now get to the proof of the main result of the section.
	
	\begin{proof}[Proof of Theorem \ref{theo-quasi-modularity-periods-multivariable}]
	We proceed by induction on the number of variables $N$.
	\begin{itemize}[label=$\circ$]
	\item The result is true for $N=0$ since there is no variable. The result is also true for $N=1$ since there are no diagonals in this case, so no poles, implying that the functions are constant in $z$.
	\item We now assume that the result is true for function with at most $N-1$ variables. Using Lemma \ref{lem-integration-one-N+1-to-N}, we can write
	$$[f]_\Sigma=\sum_{l\geqslant 1}\sum_{\substack{1=i_1,\cdots,i_l=n \\0\leqslant k_2,\cdots,k_l< n}} \left[ fQ_{i_1\cdots i_l}^{k_2\cdots k_l}(\sfA_{1N})\R_{i_1\cdots i_l}^{k_2\cdots k_l} \right]_\Sigma .$$
	Then, Lemma \ref{lem-quasi-jacobi-product-by-A} ensures that each of the member is a quasi-Jacobi form with fewer variables. Therefore, the induction hypothesis allows us to conclude the quasi-modularity.
	\end{itemize}
	\end{proof}
	
	\begin{rem}
	As the polynomials $Q_\bullet^\bullet$ are not of pure degree, the quasi-modular forms we obtain are of mixed weight. In \cite{oberdieck2018holomorphic}, we recover quasi-modular forms of pure weight by summing over all the connected components of $R_\tau^N-\Delta_\tau$.
	\end{rem}

	\subsection{Quasi-modular forms out of graph sums}
	\label{sec-generating-series-graphs-sums}
	
	\subsubsection{Generating series of graphs.} In this section we prove a result that generalizes \cite[Theorem 6.1]{goujard2019counting}. The statement is as follows. We consider $\RR/\ZZ$ with its natural orientation, and with $N$ marked points $0<x_1<\cdots< x_N<1$, labeled by elements of $[\![1;N]\!]$. Let $\Gamma$ be a graph with vertices also labeled by $[\![1;N]\!]$. This does not determine a map from $\Gamma$ to $\RR/\ZZ$ since we also need to specify the image of the edges. This is done by the following data:
	\begin{itemize}[label=$\ast$]
	\item an orientation on $\Gamma$, which we denote by $G\in\Gamma$ as in \cite{goujard2019counting},
	\item the number of turns that the image of the edge makes in $\RR/\ZZ$, which amounts to give the length of any edge, pulling back the metric from $\RR/\ZZ$.
	\end{itemize}
	The second data amounts to assign to any edge the number of times it crosses $0\in\RR/\ZZ$, called \textit{height}. We have $h_e\geqslant 0$ if the image of the vertices are in the same order in $[\![1;N]\!]$, and $h_e\geqslant 0$ if they are in the reverse order.
	
	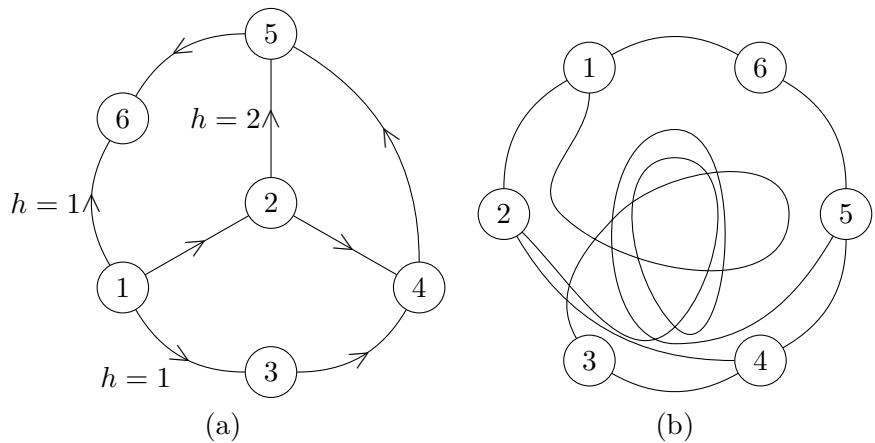
\begin{figure}
	\begin{center}
	\begin{tabular}{cc}
	\begin{tikzpicture}[line cap=round,line join=round,x=0.75cm,y=0.75cm]
	\flatpearl (1) at (-150:3) label=1;
	\flatpearl (2) at (180:0) label=2;
	\flatpearl (3) at (-90:3) label=3;
	\flatpearl (4) at (-30:3) label=4;
	\flatpearl (5) at (90:3) label=5;
	\flatpearl (6) at (150:3) label=6;
	\draw (1) to node[midway,sloped] {$>$}(2) ;
	\draw (1) to[bend right] node[midway,sloped] {$>$} node[midway,below left] {$h=1$} (3) ;
	\draw (3) to[bend right] node[midway,sloped] {$>$} (4) ;
	\draw (2) to node[midway,sloped] {$>$} node[midway,left] {$h=2$}  (5) ;
	\draw (2) to node[midway,sloped] {$>$} (4) ;
	\draw (4) to[bend right] node[midway,sloped] {$<$}  (5) ;
	\draw (5) to[bend right] node[midway,sloped] {$<$} (6) ;
	\draw (6) to[bend right] node[midway,sloped] {$<$} node[midway,left] {$h=1$} (1) ;
	\end{tikzpicture} & \begin{tikzpicture}[line cap=round,line join=round,x=0.75cm,y=0.75cm]
	\flatpearl (1) at (120:3) label=1;
	\flatpearl (2) at (180:3) label=2;
	\flatpearl (3) at (-120:3) label=3;
	\flatpearl (4) at (-60:3) label=4;
	\flatpearl (5) at (0:3) label=5;
	\flatpearl (6) at (60:3) label=6;
	\draw (1) to[bend right] (2) ;
	\draw (1) to[out=-90,in=135] (180:2) to[out=-45,in=-90] (0:2) to[out=90,in=45] (180:1) to[out=-135,in=120] (3) ;
	\draw (3) to[bend right] (4) ;
	\draw (2) to[out=-45,in=-135] (-90:2) to[out=45,in=0] (90:1) to[out=180,in=135] (-90:2) to[out=-45,in=0] (90:1.5) to[out=180,in=180] (-90:2.3) to[out=0,in=-120] (5) ;
	\draw (2) to[bend right] (4) ;
	\draw (4) to[bend right] (5) ;
	\draw (5) to[bend right] (6) ;
	\draw (6) to[bend right] (1) ;
	\end{tikzpicture} \\
	(a) & (b) \\
	\end{tabular}
	\caption{\label{fig-graph-map-to-R/Z}A graph with a specified orientation and height define a map to $\RR/\ZZ$.}
	\end{center}
	\end{figure}
	
	\begin{expl}
	On Figure \ref{fig-graph-map-to-R/Z} we give a graph $\Gamma$ with a given orientation in (a). We also give a height to the edges, which defines a map to $\RR/\ZZ$. The edges between 6 and 1, and 1 and 3 have height $1$ since they pass once through $0$ (\textit{i.e.} between $1$ and $6$). The edge between $2$ and $5$ passes through $0$ twice, so it has height $2$. 
	\end{expl}
	
	For every edge of $\Gamma$, we choose an even integer $m_e$, and denote $\mfk=(m_e)$. In \cite{goujard2019counting}, the authors consider the following generating series
	$$S(\Gamma,\mfk)=\sum_{G\in\Gamma}S(G,\mfk), \text{ where }S(G,\mfk)=\sum_{\bfw,\bfh}\prod_e w_e^{m_e+1}q^{w_eh_e}\prod_V\delta(V),$$
	where the sum is over the possible assignments $\bfw$ and $\bfh$ of weight and height to each edge of $\Gamma$, and $\delta(V)$ is the divergence condition at the vertex $V$: it is $1$ if the sum of incoming weights is equal to the sum of outcoming weights, and $0$ else. The weights $w_e$ are positive integers. Theorem 6.1 from \cite{goujard2019counting} states that the functions $S(\Gamma,\mfk)$ are quasi-modular forms of weight $\sum_e (2+m_e)$. Here, we refine this statement by allowing congruence conditions on the heights.
	
	\medskip
	
	\subsubsection{Generating series of graphs with congruence conditions.} Let $\Gamma$ be a graph, $\Lfk$ the set of edges that are loops, and $\efk$ a subset of edges of $\Gamma$. For an edge $e\in\efk$, we fix a condition:
	\begin{itemize}[label=$\ast$]
	\item If $e\in\Lfk$ is a loop, we fix a subset $\Upsilon_e\subset\ZZ/n\ZZ$ which is stable by multiplication by an element of $(\ZZ/n\ZZ)^\times$.
	\item If $e$ is not a loop, assume it is oriented, and denote by $\overline{e}$ the edge with reversed orientation. We fix an element $\mu_e\in\ZZ/n\ZZ$, and $\mu_{\overline{e}}=-\mu_e$ for the reverse orientation.
	\end{itemize}
	We have thus define a function $\mu$ on any edge chosen in $\efk$ (not a loop) with choice of orientation. We denote this set of constraints by $\bfmu$. We consider the refined generating series:
	$$S(\Gamma,\mfk,\efk,\bfmu)=\sum_{G\in\Gamma} S(G,\mfk,\efk,\bfmu), \text{ where }S(G,\mfk,\efk,\bfmu)=\sum_{\bfh,\bfw} \prod_e w_e^{m_e+1}q^{w_eh_e}\prod_V\delta(V),$$
	where in the sum over $\bfh,\bfw$ is over the assignments of height and weight as the classical case from \cite{goujard2019counting}, but where $\bfh$ is subject to the following additional condition: if $e\in\efk$ is a loop, $h_e\in \Upsilon_e$, and if $e$ is not a loop, oriented with the orientation induced by $G$, we need to have $h_e\equiv\mu_e\modulo n$.
	
	\begin{theo}\label{theo-quasi-modularity-graph-sums}
	For any graph $\Gamma$, subset of edges $\efk$ and function $\mu$, the refined generating series $S(\Gamma,\mfk,\efk,\bfmu)$ is a quasi-modular form for $\Gamma_1(n)$ of weight $\sum (2+m_e)$.
	\end{theo}

	\subsubsection{Propagator functions.} The proof of Theorem \ref{theo-quasi-modularity-graph-sums} follows the steps as in \cite{goujard2019counting}. The idea is to relate the generating series to the period of a suitably chosen quasi-Jacobi form. To find this quasi-Jacobi form, we introduce \textit{propagators} defined using the Weierstrass function. The latter admits the following Fourier development in $\zeta=e^{2i\pi z}$ and $\sfq=e^{2i\pi\tau}$, valid for $1>|\zeta|>|\sfq|$:
	$$\wp(z;\tau)=\frac{1}{12}+\frac{\zeta}{(1-\zeta)^2}+\sum_{h,w=1}^\infty w(\zeta^w+\zeta^{-w}-2)\sfq^{hw}.$$
	\begin{itemize}[label=$\circ$]
	\item We have the propagator from \cite{goujard2019counting}, defined as $P(z;\tau)=\wp(z;\tau)+2G_2(\tau)$. Assuming that $1>|\zeta|>|\sfq|$, it admits the following Fourier development:
		$$P(z;\tau)=\sum_{h=0}^\infty\sum_{w=1}^\infty w\zeta^wq^{hw} + \sum_{h=1}^\infty\sum_{w=1}^\infty w\zeta^{-w}q^{hw},$$
		and the following Laurent development at $0$:
		$$P(z;\tau) = \frac{1}{(2i\pi z)^2}+\sum_{l=1}^\infty (2l-1)2l G_{2l}(\tau)(2i\pi z)^{2l}.$$
	\item We then define the shifted propagators: $P_k(z;\tau)=P(z+k\tau;n\tau)$.
	\end{itemize}
	
	\begin{lem}\label{lem-dvp-shifted-propagators}
	The shifted propagators admit the following Fourier development, which is valid for $|\sfq|^{-k}>|\zeta|>|\sfq|^{n-k}$:
		$$P_k(z;\tau)=\sum_{\substack{h\geqslant 0 \\ h\equiv k[n]}}\sum_{w=1}^\infty w\zeta^w\sfq^{hw} + \sum_{\substack{h\geqslant 1 \\ h\equiv -k[n]}}\sum_{w=1}^\infty w\zeta^{-w}\sfq^{hw}.$$
	\end{lem}
	
	\begin{proof}
	We evaluate the Fourier development of the propagator $P$ at $(z+k\tau,n\tau)$. Since $\sfq=e^{2i\pi\tau}$ and $\zeta=e^{2i\pi z}$, this as the effect of replacing $\sfq$ by $\sfq^n$ and $\zeta$ by $\zeta\sfq^k$. The development is thus valid as long as $|\sfq|^n<|\zeta||\sfq|^k<1$. Thus, we get
	$$P_k(z;\tau)=\sum_{h=0}^\infty\sum_{w=1}^\infty w\zeta^w \sfq^{w(nh+k)}+\sum_{h=1}^\infty\sum_{w=1}^\infty w\zeta^{-w} \sfq^{w(nh-k)}.$$
	In the first sum, we set $h'=k+nh$, and in the second sum, $h'=nh-k$, yielding the result.
	\end{proof}

	\subsubsection{Loop contribution.} We start by deleting the loops from the graph $\Gamma$. Let $\Lfk$ be the set of loops in $\Gamma$ and let $\Gamma_0$ be the graph where we removed those loops. We still denote by $\mfk,\efk,\mu$ the data associated to $\Gamma_0$, forgetting about the loops.
	
	\begin{lem}\label{lem-factorization-loop-contribution}
	The generating series of $\Gamma$ factors as follows:
	$$S(\Gamma,\mfk,\efk,\bfmu)=S(\Gamma_0,\mfk,\efk,\bfmu)\prod_{e\in\Lfk}S_e,$$
	where $S_e$ is quasi-modular for $SL_2(\ZZ)$ is $e\notin\efk$, and for $\Gamma_1(n)$ if $e\in\efk$.
	\end{lem}
	
	\begin{proof}
	As loops are adjacent to a unique vertex, their weight does not impact the value of $\delta(V)$. Thus, weights and heights are chosen independently, and this choice does not depend on the orientation. Thus, we can factor $\sum_{h_e,w_e}w_e^{m_e+1}\sfq^{h_ew_e}$ out of $S(G,\mfk,\efk,\bfmu)$ for any orientation on $\Gamma$.
	\begin{itemize}[label=$\ast$]
	\item If $e\notin\efk$, then we are as in \cite{goujard2019counting}:
		$$\sum_{h,w=1}^\infty w^{m+1}\sfq^{hw} = E_m(\sfq),$$
		where the above is the definition of $E_m$, the Eisenstein series, shifted to get $0$ first coefficient, and are quasi-modular forms.
	
	\item If $e\in\efk$, we get
		$$\sum_{\substack{h=1 \\ h\in\Upsilon}}^\infty\sum_{w=1}^\infty w^{m_e+1}q^{hw}.$$
		where $\Upsilon=\Upsilon_e$ is the chosen congruence constraint. We show by induction on $n$ that for any $(\ZZ/n\ZZ)^\times$-orbit $\Upsilon$ of $\ZZ/n\ZZ$, the above series is a quasi-modular for $\Gamma_1(n)$, using $E_m(\sfq)$. The orbits of $\ZZ/n\ZZ$ are of the form $\Upsilon_n(d)=\{k\in\ZZ/n\ZZ |\mathrm{gcd}(k,n)=d\}$. We denote by $F_n^d(\sfq)$ the corresponding generating series.
			\begin{itemize}[label=$\circ$]
			\item If $n$ is a prime number, we have two possible classes $\Upsilon$: $\{0\}$ and $(\ZZ/n\ZZ)-\{0\}$, for which the series are respectively
			$$F_n^0(\sfq)=E_m(\sfq^n) \text{ and }F_n^1(\sfq)=E_m(\sfq)-E_m(\sfq^n).$$
			The first is quasi-modular for $\Gamma_1(n)$, and thus, so is the second.
			\item If $n$ is not prime, for $d\neq 0$, $h\in\Upsilon_n(d)$ if and only if $d|h$ and $\frac{h}{d}\in\Upsilon_{n/d}(1)$, so we have
			$$F_n^d(\sfq)=F_{n/d}^1(\sfq^d).$$
			In particular, the induction ensures the quasi-modularity for $d\neq 1$. For $d=0$, we have $F_n^0(\sfq)=E_m(\sfq^n)$, so we also have the quasi-modularity. As
			$$F_n^0(\sfq)+\sum_{d|n}F_n^d(\sfq)=E_m(\sfq),$$
			we also conclude to the quasi-modularity of $F_n^1(\sfq)$.
			\end{itemize}
	\end{itemize}
	\end{proof}
		
	\begin{expl}
	We have the following expressions:
		\begin{itemize}[label=$\circ$]
		\item $\sum_{h\equiv 0\modulo 2}w^{m+1}\sfq^{hw} = E_m(\sfq^2) ,$
		\item $\sum_{h\equiv 1\modulo 2}w^{m+1}\sfq^{hw} = E_m(\sfq)- E_m(\sfq^2) ,$
		\item $\sum_{h\equiv 0\modulo 3}w^{m+1}\sfq^{hw} = E_m(\sfq^3) ,$
		\item $\sum_{h\equiv 1,2\modulo 3}w^{m+1}\sfq^{hw} = E_m(\sfq)-E_m(\sfq^3) ,$
		\item $\sum_{h\equiv 2\modulo 4}w^{m+1}\sfq^{hw} = E_m(\sfq^2)-E_m(\sfq^4) .$
		\end{itemize}
	\end{expl}
	
	\begin{rem}
	It may be possible to enhance Lemma \ref{lem-factorization-loop-contribution} by proving the quasi-modularity for different type of congruence constraints. For instance, $\Upsilon_e$ being of the form $\{\pm k\}$. However, we are only interested in the values $n=2,3,4,6$, for which sets of this form are actually orbits, where the expression can be made explicit.
	\end{rem}
	
	\subsubsection{Contour integrals and graph propagator.} We now assume that the graph is without loop. Let $\Gamma$ be a loopless graph, with vertices indexed by $[\![1;N]\!]$. It is endowed with a natural orientation, where for any edge $e$, the labels of the extremities are increasing. For an edge $e$, let $e_-$ and $e_+$ be the labels of its extremities. We associate the following propagator function:
	$$P_{\Gamma,\mfk,\efk,\bfmu}(z;\tau) = \prod_{e\notin\efk}P^{(m_e)}(z_{e_-}-z_{e_+};\tau)\prod_{e\in\efk}P_{\mu_e}^{(m_e)}(z_{e_-}-z_{e_+};\tau).$$
	In other terms, we use the shifted propagator if $e\in\efk$, and the usual propagator if not. The variables are indexed by the labels of the vertices, \textit{i.e.} by $[\![1;N]\!]$.
	
	\begin{prop}\label{prop-generating-series-contour-integral-as-contour-integral}
	The refined generating series for $\Gamma$ is the constant Fourier coefficient for $P_{\Gamma,\mfk,\efk,\bfmu}$ on the domain $1>\frac{|\zeta_i|}{|\zeta_j|}>|\sfq|$ for any $i<j$:
	$$[P_{\Gamma,\mfk,\efk,\bfmu}]_\Sigma=S(\Gamma,\mfk,\efk,\bfmu).$$
	\end{prop}
	
	\begin{proof}
	On the given domain, the equations $|\sfq|^{-k}>\frac{|\zeta_i|}{|\zeta_j|}>|\sfq|^{n-k}$ are always satisfied, so that we are allowed to use the Fourier developments of the propagators. We thus expand the product to get a series with coefficients being Laurent monomials in the variables $\zeta_i$. Each propagator expresses as a sum of two series, the one where $\zeta$ has positive exponents, and the one where they are negative. Choosing one of them amounts to choose an orientation of the corresponding edge in $\Gamma$, which we reflect by the choice of $\epsilon_e=\pm 1$. A choice of $1$ means the orientation is with increasing labels, and decreasing labels if $-1$. We expand the graph propagator, yielding
	$$P_{\Gamma,\mfk,\efk,\bfmu} = \sum_{\epsilon_e=\pm 1}\sum_{\bfw,\bfh} \prod_e w_e^{m_e+1}\sfq^{h_ew_e} \left(\frac{\zeta_{e_+}}{\zeta_{e_-}}\right)^{\epsilon_e w_e},$$
	where the sum over $\epsilon_e=\pm 1$ means over all the assignments of a $\pm 1$ to each edge, and for each $e\in\efk$, we have $h_e\equiv\epsilon_e\mu_e$, $w_e$ going from $1$ to $\infty$.	The monomial $\prod_e\left(\frac{\zeta_{e_+}}{\zeta_{e_-}}\right)^{\epsilon_e w_e}$ can be rewritten $\prod_1^N \zeta_i^{\delta(i)}$, where
	$$\delta(i)=\sum_{e:e_+=i}\epsilon_e w_e-\sum_{e:e_-=i}\epsilon_e w_e,$$
	is the difference between the sum of ingoing and outgoing weights. Let $\gamma(t)=t+iy$ with $0\leqslant t\leqslant 1$ and $0<y<\Im\tau$, and $\widetilde{\gamma}=e^{2i\pi\gamma}$. We have
	$$\int_\gamma \zeta^w\dd z = \int_{\widetilde{\gamma}}\zeta^{w-1}\frac{\dd\zeta}{2i\pi}=\delta_{0,w}.$$
	Thus, the integral of a monomial is $1$ precisely when the choices of weights make the graphs balanced. Therefore, we have
	$$\int P_{\Gamma,\mfk,\efk,\bfmu} = \sum_{G\in\Gamma}\sum_{\bfw,\bfh}\prod_e w_e^{m_e+1}\sfq^{h_ew_e}\prod_V\delta(V),$$
	which is the required generating series.
	\end{proof}

	\begin{proof}[Proof of Theorem \ref{theo-quasi-modularity-graph-sums}]
	We have shown in Lemma \ref{lem-factorization-loop-contribution} that the generating series of any graph can be factored between the loop contribution, which is quasi-modular, and the generating series for the associated loopless graph. Furthermore, Proposition \ref{prop-generating-series-contour-integral-as-contour-integral} expresses the generating function as the constant Fourier coefficient of the propagator associated to the loopless graph. Theorem \ref{theo-quasi-modularity-periods-multivariable} states that the latter is quasi-modular since the propagator is a quasi-Jacobi form.
	\end{proof}

	\subsection{Generating series of Gromov-Witten invariants and explicit computations.}
	\label{sec-generating-series-GW-invariants}

	\subsubsection{General Statement.} We now apply Theorem \ref{theo-quasi-modularity-graph-sums} to get the quasi-modularity of generating series of Gromov-Witten invariants, proving Theorem \ref{theo-quasi-modularity-GW-invariants}.
	
	\begin{proof}[Proof of Theorem \ref{theo-quasi-modularity-GW-invariants}]
	For each given class $a\sfe+b\sff$, the genus $g$ GW-invariant decomposes as a sum over the genus $g$ pearl diagrams of degree $a\sfe+b\sfs$. Forgetting about the orientation, the homology classes $\varpi_V$ of the vertices, weights $w_e$ and heights $h_e$ of the edges, pearl diagrams are merely genus $g$ graphs with vertices labeled from $1$ to $N$ (the number of points constraints), of which there are a finite number.
	
	Conversely, given such graph $\Gamma$ we can get back a pearl diagram by choosing decorations: vertex homology classes given by $\mathbf{a}=(a_V)$, orientation $G\in\Gamma$, edge heights $\bfh=(h_e)$, and edge weights $\bfw=(w_e)$ such that we have a balanced graph. We denote by $\PPP_\Gamma(G,\mathbf{a},\bfh,\bfw)$ the obtained pearl diagram. We can consider the generating series
	$$S_\Gamma=\sum_{G\in\Gamma}\sum_{\mathbf{a}}\sum_{\bfh,\bfw} m\left(\PPP_\Gamma(G,\mathbf{a},\bfh,\bfw)\right)\prod_V\sfp^{a_V}\prod_e\sfq^{h_ew_e},$$
	so that the generating series of GW-invariants $F_g^S(\sfp,\sfq)$ is actually $\sum_\Gamma S_\Gamma$. To conclude, we use the expression for the multiplicity:
	$$m\left(\PPP_\Gamma(G,\mathbf{a},\bfh,\bfw)\right) = \prod_V a_V^{n_V-1}\sigma_1(a_V)\prod_k\tau(\nu_k)\prod_{E_\mathrm{flat}}w_e\prod_{E\backslash E_\mathrm{flat}}w_e^3.$$
	\begin{itemize}[label=$\circ$]
	\item First, we can factor the sum over $\mathbf{a}$, which gives
	$$A_\Gamma(\sfp)=\prod_V\left(\sum_{a_V=1}^\infty a_V^{n_V-1}\sigma_1(a_V)\sfp^{a_V}\right) = \prod_V D^{n_V-1}E_2(\sfp),$$
	where the product is over the non-flat vertices.
	\item Concerning the sum over $G\in\Gamma$, $\bfh$ and $\bfw$, it is handled by Theorem \ref{theo-quasi-modularity-graph-sums}. We choose $m_e=0$ for any edge $e$ in $E_\mathrm{flat}$, and $m_e=2$ else. The value of $\tau(\nu_k)$ only depends on the chosen orientation and heights. Moreover, the function $\tau$ takes the same values in $\nu$ and $-\nu$, which are precisely orbits modulo $(\ZZ/n\ZZ)^\times$ for $n=2,3,4,6$. Let $\gamma$ be one of the cycles involved.
		\begin{itemize}[label=$\ast$]
		\item If the cycle consists in a unique edge, necessarily a loop, $\nu$ is equal to the height of the unique edge. The function $\tau$ can thus be written as a linear combination of $\mathds{1}_\Upsilon(h_e)$, where $h_e$ is the height of the unique edge, and $\Upsilon$ goes over the $(\ZZ/n\ZZ)^\times$-orbits of $\ZZ/n\ZZ$, all of the form $\{\pm k\}$ for $n=2,3,4,6$.
		\item If the cycle $\gamma$ contains more than one edge. Choose an orientation $G\in\Gamma$, encoded by $(\epsilon_e)_e$. For $e\in\gamma$, let $\eta_e=\pm 1$ such that the cycle $\gamma$ is written
		$$\gamma=\sum_{e\in\gamma}\eta_e e.$$
		The class $\nu$ realized by the cycle $\gamma$ is
		$$\sum_{e\in\gamma} \eta_e \epsilon_eh_e\in\ZZ.$$
		Indeed, the edge $e$ with orientation given by $G$ realizes the class $h_e$, and the orientation in the cycle $\gamma$ differs by $\eta_e\epsilon_e$. The function $\tau_n(\nu)$ can be expressed as a linear combination of $\mathds{1}_{k+n\ZZ}(\nu)$, for $k\in\ZZ/n\ZZ$. Furthermore, for each $k$, there is only a finite number of $k_e\in\ZZ/n\ZZ$ such that
		$$\sum_e \eta_e k_e \equiv k\modulo n.$$
		Thus, we have
		$$\mathds{1}_{k+n\ZZ}\left(\sum \eta_e\epsilon_e h_e\right) = \sum_{(k_e):\sum\eta_ek_e\equiv k}\prod_e \mathds{1}_{k_e+n\ZZ}(\epsilon_eh_e) = \sum_{(k_e):\sum\eta_ek_e\equiv k}\prod_e \mathds{1}_{\epsilon_ek_e+n\ZZ}(h_e).$$
		Expanding, we get that the the generating series is a linear combination of generating with congruence constraints, as dealt with in Section \ref{sec-generating-series-graphs-sums}, allowing us to conclude that $B_\Gamma(\sfq)$ is quasi-modular for $\Gamma_1(n)$ by Theorem \ref{theo-quasi-modularity-graph-sums}.
		\end{itemize}
	\end{itemize}
	\end{proof}

	\subsubsection{Explicit computations.} To illustrate the regularity result, we now explicitly compute some of the generating series of GW-invariants for fixed genus. We set $E_{2k}(\sfp)=\sum_{w,h=1}^\infty w^{2k-1}\sfp^{wh}$ to be the shifted Eisenstein series with $0$ constant term.

	\begin{figure}
	\begin{center}
	\begin{tikzpicture}[line cap=round,line join=round,x=0.75cm,y=0.75cm]
	\pearl (1) at (180:2) label=1;
	\draw (1) to[out=-60,in=-120] (-60:1) to[out=60,in=-60] (60:1) to[out=120,in=60] (1) ;
	\end{tikzpicture}
	\caption{\label{fig-genus-2-pearl-diagram}The unique genus $2$ pearl diagram up to choice of weight/height.}
	\end{center}
	\end{figure}
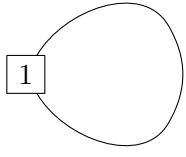
	
	\begin{expl}
	We start with the genus $2$ invariant. There is a unique pearl diagram of genus $2$, depicted on Figure \ref{fig-genus-2-pearl-diagram}, and it consists in a unique non-flat vertex, and a unique loop. The generating series over $a$ is
	$$A_\Gamma(\sfp)=\sum_{a=1}^\infty a\sigma_1(a)\sfp^a = DE_2(\sfp),$$
	and the generating series over $b$ depends on the choice of the weight and height for the unique edge. It consists exactly in the loop contribution:
	$$\sum_{w,h=1}^\infty \tau_n(h)w^3\sfq^{wh}.$$
	As the function $\tau_n(h)$ takes constant values for $h$ in the same $(\ZZ/n\ZZ)^\times$-orbit of $\ZZ/n\ZZ$, we get to express it as a sum of $E_4(\sfq^\bullet)$. More precisely, we get the following values for $B_\Gamma(\sfq)$ depending on the type of the bielliptic surface,
		\begin{enumerate}[label=$(\alph*)$]
		\item $4E_4(\sfq)-4E_4(\sfq^2)$,
		\item $3E_4(\sfq)-3E_4(\sfq^3)$,
		\item $2E_4(\sfq)+2E_4(\sfq^2)-4E_4(\sfq^4)$,
		\item $E_4(\sfq)+2E_4(\sfq^2)+3E_4(\sfq^3)-6E_4(\sfq^6)$.
		\end{enumerate}
	\end{expl}

	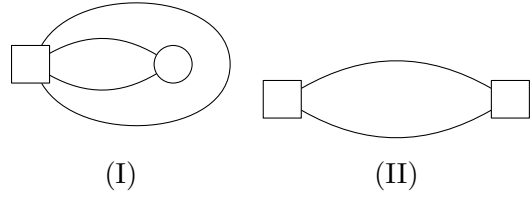
\begin{figure}
	\begin{center}
	\begin{tabular}{cc}
	\begin{tikzpicture}[line cap=round,line join=round,x=0.75cm,y=0.75cm]
	\pearl (1) at (180:2) label=;
	\flatpearl (2) at (0:0.5) label=;
	\draw (1) to[bend right] (2) ;
	\draw (2) to[bend right] (1) ;
	\draw (1) to[in=90,out=60] (0:1.5) to[out=-90,in=-60] (1) ;
	\end{tikzpicture} & \begin{tikzpicture}[line cap=round,line join=round,x=0.75cm,y=0.75cm]
	\pearl (1) at (180:2) label=;
	\pearl (2) at (0:2) label=;
	\draw (1) to[bend right] (2) ;
	\draw (2) to[bend right] (1) ;
	\end{tikzpicture} \\
	(I) & (II) \\
	\end{tabular}
	\caption{\label{fig-genus-3-pearl-diagrams} The two genus $3$ pearl diagrams up to labeling, orientation and choice of weight/height.}
	\end{center}
	\end{figure}

	\begin{expl}
	We now compute the generating series for genus $3$. This time we get two kinds of pearl diagrams, depicted on Figure \ref{fig-genus-3-pearl-diagrams}: the one with a unique non-flat vertex, a loop and a pair of edge connected to a flat vertex, and the one with two non-flat vertices.
		\begin{itemize}[label=$\circ$]
		\item For the first kind of diagram (type I), we have two labelings and the two orientations are the same up to symmetry. The generating series for $a$ is
		$$A_{\Gamma_I}(\sfp)=\sum_{a=1}^\infty a^3\sigma_1(a)\sfp^a = D^3E_2(\sfp).$$
		For the generating series in $b$, the loop contribution is the same as in the genus $2$ case, and for the pair of parallel edges linking both vertices, we get
		$$\sum_{\substack{w\geqslant 1 \\ h_1\geqslant 0,h_2\geqslant 1}}w^2\sfq^{w(h_1+h_2)}=\sum_{w,h=1}^\infty w^2h\sfq^{wh} = DE_2(\sfq).$$
		Thus, $B_{\Gamma_I}(\sfq)$ has the following values:
			\begin{enumerate}[label=$(\alph*)$]
		\item $DE_2(\sfq)\left[4E_4(\sfq)-4E_4(\sfq^2)\right]$,
		\item $DE_2(\sfq)\left[3E_4(\sfq)-3E_4(\sfq^3)\right]$,
		\item $DE_2(\sfq)\left[2E_4(\sfq)+2E_4(\sfq^2)-4E_4(\sfq^4)\right]$,
		\item $DE_2(\sfq)\left[E_4(\sfq)+2E_4(\sfq^2)+3E_4(\sfq^3)-6E_4(\sfq^6)\right]$.
		\end{enumerate}
		
		\item For pearl diagrams of the second kind (type II), the $a$-series is obtained by assigning numbers $a_1$, $a_2$ to the two vertices:
		$$A_{\Gamma_{II}}(\sfp)=\sum_{a_1,a_2} a_1\sigma_1(a_1)a_2\sigma_1(a_2) \sfp^{a_1+a_2} = DE_2(\sfp)^2.$$
		For the $b$-series, we have to choose the heights of each edge:
		$$\sum_{\substack{h_1\geqslant 1,h_2\geqslant 0 \\ w\geqslant 1 }} w^6\tau(h_1+h_2)\sfq^{w(h_1+h_2)} = \sum_{h,w\geqslant 1}w^5(hw)\tau(h)\sfq^{hw}.$$
		If we did not have the $\tau(h_1+h_2)$, we would get $\sum_{h,w\geqslant 1}w^5(hw)\sfq^{wh}=DE_6(\sfq)$. With the term, we get as in the genus $2$ case evaluations at some $\sfq^\bullet$, yielding the following values for $B_{\Gamma_{II}}(\sfq)$:
		\begin{enumerate}[label=$(\alph*)$]
		\item $4DE_6(\sfq)-8DE_6(\sfq^2)$,
		\item $3DE_6(\sfq)-9DE_6(\sfq^3)$,
		\item $2DE_6(\sfq)+4DE_6(\sfq^2)-16DE_6(\sfq^4)$,
		\item $DE_6(\sfq)+4DE_6(\sfq^2)+9DE_6(\sfq^3)-36DE_6(\sfq^6)$.
		\end{enumerate}
		
		\end{itemize}
		In total, for the surfaces of type $(a)$, we get
	$$F_3^S(\sfp,\sfq) = 4D^3E_2(\sfp)DE_2(\sfp)\left[E_4(\sfq)-E_4(\sfq^2)\right] + 4DE_2(\sfp)^2\left[DE_6(\sfq)-2DE_6(\sfq^2)\right].$$
	\end{expl}

	\begin{figure}
	\begin{center}
	\begin{tabular}{cccc}
	\begin{tikzpicture}[line cap=round,line join=round,x=0.75cm,y=0.75cm]
	\pearl (1) at (180:1.5) label=;
	\pearl (2) at (60:1.5) label=;
	\pearl (3) at (-60:1.5) label=;
	\draw (1) to[bend right] (3) ;
	\draw (3) to[bend right] (2) ;
	\draw (2) to[bend right] (1) ;
	\end{tikzpicture} & \begin{tikzpicture}[line cap=round,line join=round,x=0.75cm,y=0.75cm]
	\pearl (1) at (180:1) label=;
	\pearl (2) at (0:1) label=;
	\flatpearl (F) at (0:0) label=;
	\draw (1) to[in=90,out=120] (180:2) to[out=-90,in=-120] (1) ;
	\draw (2) to[in=90,out=60] (0:2) to[out=-90,in=-60] (2) ;
	\draw (1) to (F) to (2) ;
	\end{tikzpicture} & \begin{tikzpicture}[line cap=round,line join=round,x=0.75cm,y=0.75cm]
	\pearl (1) at (180:1) label=;
	\pearl (2) at (0:1) label=;
	\flatpearl (F) at (0:2.5) label=;
	\draw (1) to[in=90,out=120] (180:2) to[out=-90,in=-120] (1) ;
	\draw (2) to[bend right] (F) ;
	\draw (F) to[bend right] (2) ;
	\draw (1) to (2) ;
	\end{tikzpicture} & \begin{tikzpicture}[line cap=round,line join=round,x=0.75cm,y=0.75cm]
	\pearl (1) at (180:1) label=;
	\flatpearl (F1) at (0:2) label=;
	\flatpearl (F2) at (0:1) label=;
	\draw (1) to[in=90,out=60] (0:3) to[out=-90,in=-60] (1) ;
	\draw (1) to[bend right,out=-50] (F1) ;
	\draw (F1) to[bend right,in=230] (1) ;
	\draw (1) to[bend right] (F2) ;
	\draw (F2) to[bend right] (1) ;
	\end{tikzpicture} \\
	& \begin{tikzpicture}[line cap=round,line join=round,x=0.75cm,y=0.75cm]
	\pearl (1) at (180:1) label=;
	\pearl (2) at (0:1) label=;
	\flatpearl (F) at (90:0.5) label=;
	\draw (2) to[in=90,out=60] (0:2) to[out=-90,in=-60] (2) ;
	\draw (1) to[bend right] (2) ;
	\draw (2) to[bend right] (F) ;
	\draw (F) to[bend right] (1) ;
	\end{tikzpicture} & \begin{tikzpicture}[line cap=round,line join=round,x=0.75cm,y=0.75cm]
	\pearl (1) at (180:1) label=;
	\pearl (2) at (0:1) label=;
	\flatpearl (F) at (0:2.5) label=;
	\draw (2) to[bend right] (F) ;
	\draw (F) to[bend right] (2) ;
	\draw (1) to[bend right] (2) ;
	\draw (2) to[bend right] (1) ;
	\end{tikzpicture} & \begin{tikzpicture}[line cap=round,line join=round,x=0.75cm,y=0.75cm]
	\pearl (1) at (180:2) label=;
	\pearl (2) at (0:2) label=;
	\flatpearl (F) at (0:0) label=;
	\draw (1) to[bend right] (2) ;
	\draw (2) to[bend right] (1) ;
	\draw (1) to (F) to (2) ;
	\end{tikzpicture} \\
	\end{tabular}
	\caption{\label{fig-genus-4-pearl-diagrams}The different genus $4$ pearl diagrams up to labeling, orientation and choice of weight/height.}
	\end{center}
	\end{figure}
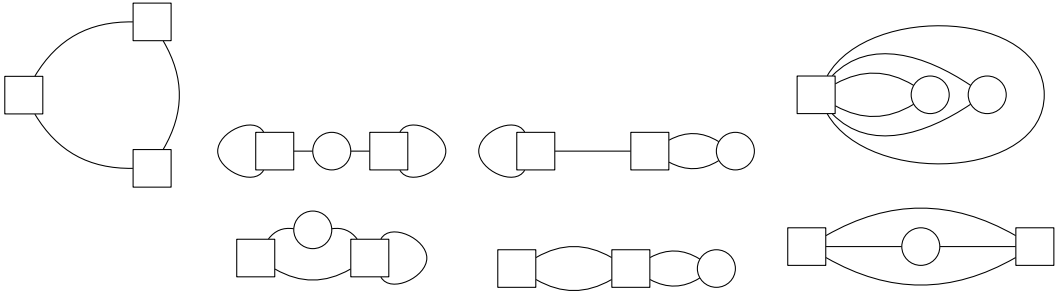
	
	\begin{expl}
	We do not make the extensive computation but provide the necessary diagrams involved in the computation of the generating series for genus $4$.
	\end{expl}

	\begin{expl}
	We compute the refined generating series of refined invariants for genus $2$. There is a unique diagram up to the labeling of of the edge by height and weight. Furthermore, when inserting the refined vertex multiplicity, we need to choose $k|a_V$. Setting $a_V=kl$, we sum over the choices $k,l,h,w$ and get the following generating series:
	$$-\sum_{k,l,h,w=1}^\infty \tau(h) (q^{wk/2}-q^{-wk/2})^2lw\sfp^{kl}\sfq^{wh}.$$
	Stress that the variables $q$ (variable of the refined invariants) and $\sfq$ (variable of the generating series) are not the same. To remove ambiguity, let us set $q=e^{iu}$ and expand to get the generating series of invariants with a $\lambda$-class:
	\begin{align*}
	(q^{wk/2}-q^{-wk/2})^2 = & -4\sin^2\left(\frac{kw}{2}u\right) \\
	= & -2(1-\cos(kwu) ) \\
	= & -2\sum_{g=2}^\infty \frac{(-1)^g}{(2g-2)!}(kwu)^{2g-2}.\\
	\end{align*}
	Thus, we get that the term in front of $u^{2g-2}$ is
	\begin{align*}
	 & 2\frac{(-1)^g}{(2g-2)!}\sum_{k,l,w,h=1}^\infty \tau(h)lk^{2g-2}w^{2g-1} \sfp^{kl}\sfq^{wh} \\
	= & 2DE_{2g-1}(\sfp)\sum_{h,w=1}^\infty \tau(h)w^{2g-1}\sfq^{wh}. \\
	\end{align*}
	For instance, in the case of bielliptic surfaces of type $(a)$, we get
	$$\sum_{a,b\geqslant 1}\gen{\lambda_{g-2};\pt}_{g,a\sfe+b\sff}^S\sfp^a\sfq^b =  8\frac{(-1)^g}{(2g-2)!}DE_{2g-2}(\sfp)(E_{2g}(\sfq)-E_{2g}(\sfq^2) ).$$
	\end{expl}

\bibliographystyle{plain}
\bibliography{biblio}

\end{document}